\documentclass[10pt,11pt]{amsart}%
\usepackage[dvips]{graphicx}
\usepackage{amsmath}
\usepackage{color}
\usepackage{amsfonts}
\usepackage{stmaryrd}
\usepackage{amssymb}%
\usepackage{tikz}
\providecommand{\U}[1]{\protect\rule{.1in}{.1in}}
\providecommand{\U}[1]{\protect\rule{.1in}{.1in}}
\providecommand{\U}[1]{\protect\rule{.1in}{.1in}}
\providecommand{\U}[1]{\protect\rule{.1in}{.1in}}
\providecommand{\U}[1]{\protect\rule{.1in}{.1in}}
\providecommand{\U}[1]{\protect\rule{.1in}{.1in}}
\providecommand{\U}[1]{\protect\rule{.1in}{.1in}}
\providecommand{\U}[1]{\protect\rule{.1in}{.1in}}
\newtheorem{theorem}{Theorem}[section]

\newtheorem{corollary}[theorem]{Corollary}

\newtheorem{definition}[theorem]{Definition}
\newtheorem{example}[theorem]{Example}

\newtheorem{lemma}[theorem]{Lemma}

\newtheorem{proposition}[theorem]{Proposition}
\newtheorem{remark}[theorem]{Remark}

\def \b{\beta}
\def \l{\lambda}

\usepackage{tikz}
\usepackage[notocbasic,refpage]{nomencl}
\makenomenclature
\title{Quantum cohomology of the Grassmannian and unitary Dyson Brownian motion}

\begin{document}
\thanks{This project is supported by the Agence Nationale de la
Recherche funding CORTIPOM ANR-21-CE40-0019.}

\author[J. Guilhot]{Jérémie Guilhot}
\address{Institut Denis Poisson, Université de Tours, France}
\email{jeremie.guilhot@lmpt.univ-tours.fr}

\author[C. Lecouvey]{Cédric Lecouvey} 
\address{Institut Denis Poisson, Université de Tours, France}
\email{cedric.lecouvey@lmpt.univ-tours.fr}
       
\author[P. Tarrago]{Pierre Tarrago} 
\address{Laboratoire de Probabilit\'es, Statistique et Mod\'elisation, Sorbonne Universit\'e, France}
\email{pierre.tarrago@sorbonne-universite.fr}

\maketitle
\begin{abstract}
We study a class of commuting Markov kernels whose simplest element describes the movement of $k$ particles on a discrete circle of size $n$ conditioned to not intersect each other. Such Markov kernels are related to the quantum cohomology ring of the Grassmannian, which is an algebraic object counting analytic maps from $\mathbb{P}^1(\mathbb{C})$ to the Grassmannian space of $k$-dimensional vector subspaces of $\mathbb{C}^n$ with prescribed constraints at some points of $\mathbb{P}^1(\mathbb{C})$. We obtain a Berry-Esseen theorem and a local limit theorem for an arbitrary product of approximately $n^2$ Markov kernels belonging to the above class, when $k$ is fixed. As a byproduct of those results, we derive asymptotic formulas for the quantum cohomology ring of the Grassmannian in terms of the heat kernel on $SU(k)$.  
\end{abstract}

\vspace{0.5 cm}

\begin{flushright}
\textit{To Philippe Biane, for his 60th birthday.}
\end{flushright}

\vspace{0.5cm}

\section{Introduction}
The present paper gives a probabilistic study of an integrable class of rooted graphs which describe movements of particles on a discrete circle conditioned to not intersect. The simplest example of this class is the graph $B_{k,n}$ encoding the transitions of $k$ particles moving clockwise with small steps on a discrete circle of $n$ sites without intersecting each other. This graph is rooted at a particular configuration where the $k$ particles are stacked next to each other after some chosen origin. We present such a graph for $k=2$ particles on a discrete circle of size $n=4$ in the figure below (where the dashed line gives the origin of the circle).

\begin{figure}[h!]
\centering
\includegraphics[scale=0.5]{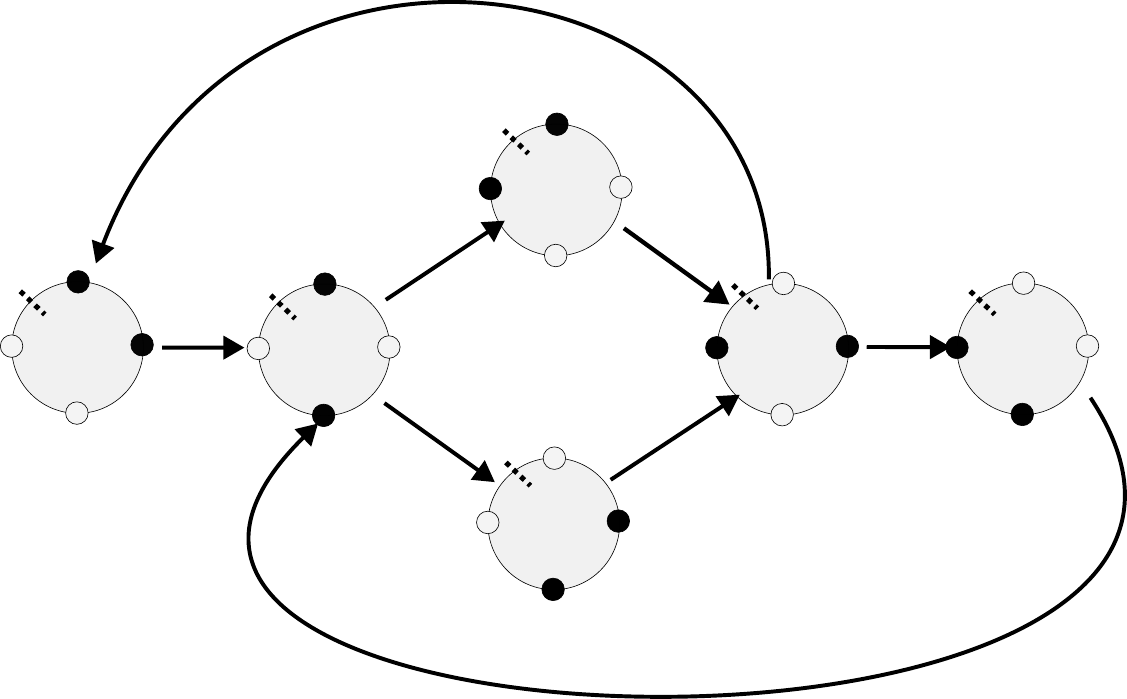}
\caption{\label{figPart}Graph $B_{2,4}$ encoding the movement of $2$ nonintersecting particles on a discrete circle of size $4$. The dashed line marks the origin of the circle.}
\end{figure}

 As we will see, the adjacency matrix of the graph $B_{k,n}$ generates an algebra which can be seen as a finite version of the algebra of symmetric functions in $k$ variables, and this algebra has a distinguished basis for which structure coefficients are nonnegative. Graphs whose adjacency matrices generate algebras having such properties have been studied in \cite{GLT} under the name of positively multiplicative graphs. In the case of $B_{k,n}$, the corresponding algebra is the small quantum cohomology of the Grassmannian, see \cite{B,BCF}. This algebra describes the geometry of the moduly space of rational maps from $\mathbb{P}^1(\mathbb{C})$ to the Grassmannian variety $G_{k,n}$, the latter describing the space of $k$-dimensional vector subspaces of $\mathbb{C}^n$. 
 
 The nonnegative coefficients that appear in the multiplication of the distinguished basis are generalizations of the Littlewood-Richardson coefficients called quantum Littlewood-Richardson coefficients \cite{BCF}. Exactly as the latter are central combinatorial objects in the representation theory of the general linear group $U(k)$ and the geometry of the Grassmannian variety $G_{k,n}$, the former play a prominent role in the combinatorics of the representation of the loop group on $U(k)$ and the moduli space of rational maps from $\mathbb{P}^1(\mathbb{C})$ to $G_{k,n}$. In the simplest nontrivial situation $n=2, k=1$, the quantum Littlewood-Richardson coefficients boil down to the fact that specifying the values of a degree $d$ rational function from $\mathbb{P}^1(\mathbb{C})$ to $\mathbb{P}^1(\mathbb{C})$ at 2d+1 generic points uniquely determines this function (see Example \ref{Example:QuantumCohomP1}).
 
The adjacency matrix of the graph $B_{k,n}$ can be turned into a Markov kernel up to conjugation by an appropriate diagonal matrix. Such a transformation, using the Perron-Frobenius eigenvector and eigenvalue of the adjacency matrix, is just a generalization of the Doob $h$-transform for conditioned random walks \cite{Do-87}. Then, from a probabilistic perspective, the notion of positively multiplicative graph translates into the existence of a large family of commuting Markov kernels which are obtained in this way from the algebra generated by the adjacency matrix of the graph. This family is the largest possible in the following sense: for any probability distribution on the set of vertices, there exists a Markov kernel of the aforementioned family mapping the Dirac mass on the root to this probability distribution. The basic example of positively multiplicative graphs is provided by the directed cycle graph of size $n$ rooted at one of its vertices. The underlying algebra of Markov kernels commuting with the adjacency matrix $S$ of this graph is the set of circulant matrices with nonnegative entries. It has a basis $\{S^k, 0\leq k\leq n-1\}$ with the multiplication rule $S^kS^l=S^{(k+l)[n]}$, where $(k+l)[n]$ means the value of $k+l$ modulo $n$. In this particular simple case, the \textit{h}-transform of the original adjacency matrix is trivial since it is already a Markov kernel.

The case of the directed cycle graph of size $n$ actually corresponds to the simplest case $B^{1,n}$ of our graphs. In a probabilistic context, it is known that products of linear combinations of $\{S^k, 0\leq k\leq n-1\}$ generate all spatially homogeneous random walks on a discrete circle. Moreover, we can diagonalize the action of any of these operators on the so-called Fourier basis, and the change of basis from the canonical basis to the Fourier basis is then called the discrete Fourier transform. As $n$ goes to infinity and after normalization, the discrete circle becomes a continuous circle of radius $2\pi$ and the discrete Fourier transform becomes the continuous Fourier transform for periodic functions on $\mathbb{R}$ with period $2\pi$. The interplay between discrete and continuous Fourier transforms allows to prove several precise probabilistic results such as the wrapped central limit theorem and their refined versions, local limit theorem and Berry-Esseen theorem, \cite{BobLed}. To put those three results in a nutshell, the wrapped central limit theorem states that a random walk on a discrete circle of size $n$ can be embedded in a continuous circle such that after a number $m$ of steps with order $n^2$, the distribution is close to a wrapped normal distribution on the circle. Then, the Berry-Esseen theorem gives quantitative convergence towards the limit distribution while the local limit theorem gives asymptotic formula for the probability to be at some position at time $m$ in terms of the density of the wrapped normal distribution.

The goal of the present paper is then to get similar probabilistic results when there are $k\geq 2$ particles on the discrete unit circle, with the additional requirement that particles do not intersect. An important property of the adjacency matrix of $B_{k,n}$ is to be explicitly diagonalizable. The eigenbasis is given in terms of specializations of Schur functions at $k$-tuples of distinct $n$-roots of unity and can thus be seen as a symmetrized version of the Fourier basis on $(\mathbb{Z}/n\mathbb{Z})^k$. We can then define a discrete Fourier transform on $B_{k,n}$, which turns out to be a discrete version of the Fourier transform on the conjugation invariant measures on $U(k)$, the unitary group in dimension $k$. By using these Fourier transforms, we get a Berry-Esseen theorem and a discrete local limit theorem, similar to the ones existing for $k=1$, for the complete family of Markov kernels generated by the positively multiplicative graph $B_{k,n}$. Instead of a convergence towards a wrapped Gaussian distribution, we get a convergence towards the eigenvalue process of a Brownian motion on $U(k)$, which is called a unitary Dyson Brownian motion. As a corollary of our local limit theorem, we get some asymptotic formulas for the multiplication in the quantum cohomology ring of the Grassmannian at fixed $k$ and growing $n$. From the geometric interpretation of the quantum cohomology, an informal description of this asymptotic formula says the following (the reader should refer to Corollary \ref{cor:cohomology} for an exact result):

\vspace{0.3cm}

 \textit{ The number of rational functions from $\mathbb{P}^1(\mathbb{C})$ to $G_{k,n}$ with prescribed constraints of small codimension at around $n^2$ generic points of $\mathbb{P}^1(\mathbb{C})$ is either zero, infinite or approximately described by the density of a unitary Dyson Brownian motion at a time depending on the number of constraints and their homology.}

\vspace{0.3cm}
There had been several combinatorial formulas for counting such numbers in specific cases, see \cite{RRW}, \cite{Sch} and also \cite{Sot} for applications to the pole placement problem. Up to our knowledge the result of Corollary \ref{cor:cohomology} is the first asymptotic result counting such rational functions.

The paper is organized as follows: Section \ref{sec:intro_model} formally introduces the graph $B_{k,n}$, the relation with the unitary group and the main results of the paper. Section \ref{sec:pm_graph} defines the general notion of integrable positively multiplicative graph and the corresponding convolution of probability measures which are well suited for the study of $B_{k,n}$, and the related Fourier transform. Section \ref{sec:formal_model_pm} applies these notions to our graph $B_{k,n}$ and relates the discrete Fourier transform we get to the continuous Fourier transform on the unitary group $U(k)$. The technical core of the this paper is done in Section \ref{sec:Fourier_large_product}, which gives an asymptotic formula for the Fourier transform of the convolution of a large sequence of probability measures on $B_{k,n}$. We kept trace of the dependence in $k$ of the formulas that we obtained for later purposes. Finally, Section \ref{sec:Berry-Esseen} and Section \ref{sec:LLT} respectively give proofs of the Berry-Esseen theorem and the local limit theorem with its geometric interpretation. We added in Appendix \ref{sec:quantum_cohomology} a short introduction to the quantum cohomology of the Grassmannian for probabilists.

\bigskip
\noindent\textbf{Acknowledgments:}  The authors are partially supported by the Agence Nationale de la Recherche funding ANR CORTIPOM 21-CE40-001.

\section{Presentation of the model and main results}
\label{sec:intro_model}
\subsection{$k$-configuration on a circle and continuous counterpart}\label{subsec:intro_model}

For $k,n\in \mathbb{N}^*$ with $k\leq n$, let $B_{k,n}$ denote the set of positions of $k$ nonintersecting and indistinguishable particles on a discrete circle of size $n$. Formally, $B_{k,n}$ is the set of increasing sequences of  $\llbracket 0,n-1\rrbracket$ of length $k$: \nomenclature{$B_{k,n}$}{Set of decreasing integer sequences in $\llbracket 0,n-1\rrbracket$} such a sequence gives the position of the $k$ particles in the clockwise direction starting from the origin $0$ of the discrete circle, see Figure \ref{figPart} and the corresponding $k$-subsets on Figure \ref{figBkn} in the case $k=2,n=4$.

The size $\langle I\rangle$ of an element $I=\{I_1>\ldots>I_k\}\in B_{k,n}$ is the sum of its parts (remark that the size is always positive when $I\in B_{k,n}$). The discrete set $B_{k,n}$ is given a graph structure by saying there is an oriented edge from $I$ to $J$ when $\langle J\rangle=\langle I\rangle +1\mod n$ and $\#\{1\leq l,l'\leq k, J_{l'}=I_l\}=k-1$. We then write $I\nearrow J$ when there is such an oriented edge from $I$ to $J$. The edges of $B_{k,n}$ encode the clockwise movement of the $k$ particles around the circle conditioned to never intersect. A distinguished vertex of $B_{k,n}$ is given by the sequence $I_0=\left\lbrace 0,\ldots,k-1\right\rbrace$, \nomenclature{$I_0$}{Root $\left\lbrace 0,\ldots,k-1\right\rbrace$ of $B_{k,n}$} which represents $k$ particles stacked after the origin on the discrete circle. See Figure \ref{figBkn} for the graph structure of $B_{2,4}$.

\begin{figure}[h!]
\centering
\includegraphics[scale=0.5]{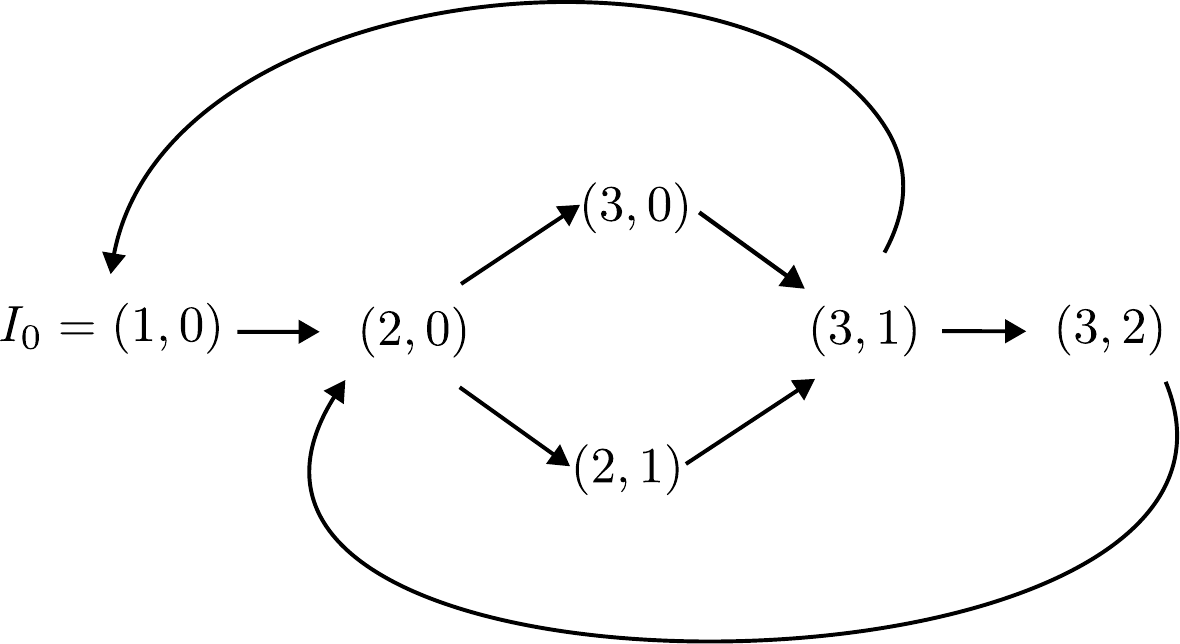}
\caption{Graph $B_{2,4}$ from a combinatorial perspective. \label{figBkn}}
\end{figure}

For fixed $k\geq 1$, the family of graphs $\{B_{k,n},\,n\geq 1\}$ has a continuous counterpart embodied by the set $ T_{k}$ of positions of $k$ indistinguishable particles on a continuous circle of radius $1$. Identifying the positions of the particles by their distance from a distinguished point of the circle in the clockwise direction, we formally identify $T_k$ with the subset of vectors $\vec{u}=(u_1,\ldots,u_k)\in [0,2\pi[^k$ such that 
$$2\pi>u_1\geq u_2\geq\dots\geq u_k\geq 0.$$
\nomenclature{$T_k$}{Set of non-increasing vectors in $[0,2\pi[^k$}Remark that we authorize particles to be at the same position in the continuous case, since this situation may theoretically appear as a limit case of discrete configurations as $n$ goes to infinity (although it will almost surely not be the case).

For $x\in \mathbb{R}$, denote by $\overline{x}$ the unique element of $[0,2\pi[$ such that $x-\overline{x}\in 2\pi\mathbb{Z}$. Then, reflecting the circular nature of the model, there is a shift action of $\mathbb{R}$ on $T_k$ given by 
$$R_{t}(\vec{u})=Sort(\overline{u_1-2t\pi},\ldots,\overline{u_k-2t\pi}),$$\nomenclature{$R_{t}$}{$R_{t}(\vec{u})=Sort(\overline{u_1-2t\pi},\ldots,\overline{u_k-2t\pi})$} where $Sort(\vec{u})$ is the sequence obtained from $\vec{u}\in [0,2\pi[^k$ by sorting it in decreasing order. For each $n\geq 1$, the graph $B_{k,n}$ can actually be embedded in $ T_{k}$ through the map
\begin{equation}\label{eq:definition_xi_n}
\xi_n:\left\lbrace \begin{aligned}
&B_{k,n}&\rightarrow &\quad T_{k}\\
&I&\mapsto& R_{\frac{k-1}{2n}} \left(\frac{2\pi }{n}I\right)
\end{aligned}\right..
\end{equation}
\nomenclature{$\xi_n$}{$\xi_n(I)=R_{\frac{k-1}{2n}} \left(\frac{2\pi }{n}I\right)$} Note that there is a shift of $\frac{\pi(k-1)}{n}$ which seems at first sight arbitrary: its role is mainly to simplify later expressions.

\subsection{A complete family of commuting Markov kernels associated to $B_{k,n}$}
In order to state our main results, let us briefly explain how to associate to $B_{k,n}$ a family of Markov chains. By analogy with the space $L^2(T_k)$ of square integrable functions on $T_k$, denote by $L^2(B_{k,n})$\nomenclature{$L^2(B_{k,n})$}{Hilbert space of complex functions on $B_{k,n}$} the Hermitian space of functions on $B_{k,n}$ with the distinguished orthonormal basis $\mathcal{B}_0=\left\{ e_{I}, I\in B_{k,n}\right\}$\nomenclature{$\mathcal{B}_0$}{Basis $\left\{ e_{I}, I\in B_{k,n}\right\}$ of $L^2(B_{k,n})$} given by $e_I(J)=\delta_{I,J}$. The edge structure of $B_{k,n}$ can then be encoded by the operator $A:L^2(B_{k,n})\rightarrow L^2(B_{k,n})$ given by 
$$A(e_I)=\sum_{I\nearrow J} e_{J}.$$
\nomenclature{$A$}{Adjacency matrix of $B_{k,n}$}
Since $B_{k,n}$ is strongly connected, the adjacency matrix $A$ is irreducible and thus there exist unique (up to a scaling) left and right Perron-Frobenius eigenvectors $h^l$ and $h^r$ of $A$ \nomenclature{$h^l,\,h^r$}{Left and right Perron-Frobenius eigenvectors of $A$}corresponding to the same Perron-Frobenius eigenvalue $\lambda$. The coordinates $h^r(I),h^l(I), \, I\in B_{k,n}$ are positive and $\lambda>0$ so that the transformed matrix $A^h$ \nomenclature{$A^h$}{$h$-transform of $A$} given by
$$A^h(e_J)=\frac{1}{\lambda}\sum_{I\in B_{k,n}}\frac{h^r(J)}{h^r(I)}A_{IJ}e_{I}$$ 
is a stochastic matrix with respect to the basis $\mathcal{B}_0$. The corresponding Markov chain started at any vertex converges then in probability to the invariant measure $\mu_h$\nomenclature{$\mu_h$}{Invariant measure of $A^h$} such that $\mu^h(I)=h^l(I)h^r(I)$ for $I\in B_{k,n}$, where $h^l$ and $h^r$ are normalized so that $h^l(I_0)=1$ and  $\langle h^l,h^r\rangle=1$. We call $A^h$ the \textit{$h$-transform} of $A$, by analogy with the specific situation where $A$ comes from the Markov kernel of a random walk on $\mathbb{Z}^k$ conditioned to stay in a finite connected domain $D\subset \mathbb{Z}^k$ (see \cite{Do-87}).

As we will see later, in the case of $B^{k,n}$ we have the explicit expression
$$h^l(I)=\frac{\vert V(\xi_n(I))\vert}{V(I_0)}\,,\quad \mu^h(I)=\frac{\vert V(\xi_n(I))\vert^2}{n^k},$$
where $V(\vec{u})$ denoting the Vandermonde determinant of $(\exp(iu_j))_{1\leq j\leq k}$ for $\vec{u}=(u_1,\ldots,u_k)\in\mathbb{R}^k$\nomenclature{$V(\vec{u})$}{Vandermonde determinant of $(\exp(iu_j))_{1\leq j\leq k}$ for $\vec{u}=(u_1,\ldots,u_k)\in\mathbb{R}^k$}, so that 
\begin{equation}\label{eq:definition_Vandermonde}
V(\xi_n(I))=\prod_{j<j'}\left(e^{2i\pi I_{j'}/n}-e^{2i\pi I_{j}/n}\right)=(2i)^{k(k-1)/2}e^{(k-1)i\pi\frac{\langle I\rangle}{n}}\prod_{j<j'}\sin\left(\frac{(I_{j'}-I_{j})\pi}{n}\right).
\end{equation}

As this can be readily seen from the graph structure of $B_{k,n}$,
$$A^h(e_{I_0})=\frac{1}{\lambda}\frac{\vert V(I_0)\vert}{\vert V(I^1)\vert}e_{I^1}$$
where $I^1=\left\{k,k-2,k-3,\ldots,0\right\}$ and $\lambda$ is the Perron-Frobenius eigenvalue of $A$, which can be proven to be equal to $\sum_{a=1}^ke^{\frac{2i\pi(k+1-a)}{n}}$.

Anticipating with the continuous case, let us denote by $\mathcal{M}_{\geq 0}(B_{k,n})$ the set of positive measures on $B_{k,n}$, which can be viewed as vectors of $L^2(B_{k,n})$ with nonnegative entries, and $\mathcal{M}_{1}(B_{k,n})\subset \mathcal{M}_{\geq 0}(B_{k,n})$ the subset of vectors of $\mathcal{M}_{\geq 0}(B_{k,n})$ integrating to $1$ with respect to $\mu^h$, namely
$$\mathcal{M}_{1}(B_{k,n})=\left\{\mu\in L^2(B_{k,n})\,\left\vert \sum_{I\in B_{k,n}} \mu(I)\mu^h(I)=1, \, \mu(I)\geq 0, I\in B_{k,n}\right.\right\}.$$
Elements of $\mathcal{M}_1(B_{k,n})$ are called \textit{$h$-probability measures} in the sequel.\nomenclature{$\mathcal{M}_{1}(B_{k,n})$}{Set of $h$-probability measures on $B_{k,n}$}

Our first results give a product structure on $L^2(B_{k,n})$ which restricts to a convolution product on the set $\mathcal{M}_{1}(B_{k,n})$ of $h$-probability measures.
\begin{proposition}\label{prop:result_1}
There exists a unique product $*:L^2(B_{k,n})\times L^2(B_{k,n})\rightarrow L^2(B_{k,n})$ such that 
\begin{enumerate}
\item $(L^2(B_{k,n}),+,*)$ is a commutative algebra,
\item $\mathcal{M}_{1}(B_{k,n})*\mathcal{M}_{1}(B_{k,n})\subset \mathcal{M}_{1}(B_{k,n})$,
\item $\frac{1}{\lambda}A^h(y)=\frac{e_{I^1}}{\mu^h(I^1)}*y$ for all $y\in L^2(B_{k,n})$,
\item $\left(\frac{e_{I}}{\mu^h(I)}\right)* \left(\frac{e_{I_0}}{\mu^h(I_0)}\right)=\frac{e_{I}}{\mu^h(I)}$ for all $I\in B_{k,n}$.
\end{enumerate}
\end{proposition}
The existence of this convolution product on $L^2(B_{k,n})$ reflects the fact that $B_{k,n}$ is a positively multiplicative graph, as introduced in \cite{GLT}. Such graphs have several properties which we recall in Appendix \ref{sec:pm_graph}. 
\begin{remark}\label{rem:qcoho_convolution}
The existence of this convolution product extending the action of $A$ and under which $\mathcal{M}_1(B_{k,n})$ is stable reflects the fact that $B_{k,n}$ encodes the Pieri rules in the small quantum cohomology of the Grassmannian (see Appendix \ref{sec:quantum_cohomology} for an introduction to this notion), which is an algebra having nonnegative structure coefficients in a particular basis. As we will see in Section \ref{sec:formal_model_pm}, the ring $(L^2(B_{k,n}),+,*)$ is actually isomorphic to the ring of symmetric functions specialized at $k$-tuples of distinct roots of $(-1)^{k+1}$. In particular, one can see through this isomorphism that the convolution by $\frac{e_{I^S}}{\mu^h(I^S)}$ for $I^S=(k,k-1,\ldots,1)$ corresponds to the operator $S^h$ on $L^2(B_{k,n})$, where $S$ is the shift operator 
$Se_{I}=e_{\tilde{I}}$ where $\{\tilde{I}_i,1\leq i\leq k\}=\{I_i+1 \,[n], 1\leq i\leq k\}$.
\end{remark}
\begin{remark}
The convolution algebra $\mathcal{M}_1(B_{k,n})$ is also deeply related to the fusion algebra for the semi-simple Lie algebra of type $A_k$, see \cite{De}. The latter is also a quotient of the small quantum cohomology ring of the Grassmannian, and the structure coefficients of each of those algebras can be easily deduced from the other.
\end{remark}
\subsection{Relation between $B_{k,n}$ and the unitary group}
Recall that we embedded $B_{k,n}$ into $T_k$ through the map $\xi_n$. This map can be transformed into a map, still denoted by $\xi_n$, from $\mathcal{M}_{\geq 0}(B_{k,n})$ to $\mathcal{M}_f( T_{k})$, the space of finite Radon measures on $ T_{k}$ with the formula
$$\xi_n(\mu)=\left(\frac{2\pi}{n}\right)^k\sum_{I\in B_{k,n}}\mu(I)\delta_{\xi_n(I)}.$$
As explained in the previous section, for $\mu\in\mathcal{M}_{1}(B_{k,n})$ we have $\sum_{I\in B_{k,n}}\mu^h(I)\mu(I)=1$ with $\mu^h(I)=\frac{\vert V(\xi_n(I)\vert^2}{n^k}$, so that 
\begin{equation}\label{eq:spherical_integral_xi(mu)}
\int_{ T_{k}}\frac{\left\vert V\left(\vec{u}\right)\right\vert^2}{(2\pi)^k}d\left[\xi_n(\mu)\right](\vec{u})=1.
\end{equation}
The latter equality is reminiscent of the spectral projection of conjugation invariant probability measures on the unitary group $U(k)$. Indeed, according to the Weyl integration formula \cite[Section 4]{AGZ}, the projection $p:U(k)\rightarrow  T_{k}$ mapping $U$ to the arguments of its eigenvalues sorted in decreasing order\nomenclature{$p$}{Map sending $U\in U(k)$ to the decreasing sequence of the arguments of its eigenvalues taken in $[0,2\pi]$} (and taken in $[0,2\pi)$) transforms the Haar measure $\mu_H$ on $U(k)$ into the measure 
$$dp_{\star}\mu_H(\vec{u})=\frac{\left\vert V\left(\vec{u}\right)\right\vert^2}{(2\pi)^k}d\vec{u}$$
on $ T_{k}$, where $p_{\star}$ denotes the push-forward of $\mu_H$ through $p$ \footnote{We recall that $p_{\star}\mu(A)=\mu(p^{-1}(A))$ for $A$ measurable set of $T_k$.}. For $\vec{u}\in  T_{k}$, $\frac{\left\vert V\left(\vec{u}\right)\right\vert^2}{(2\pi)^k}$ is exactly the area of the adjoint orbit $\mathcal{O}(\vec{u})$ of unitary matrices having eigenvalues $\left(e^{i u_1},\ldots,e^{iu_k}\right)$, see \cite[Section 4.1, Eq. (4.1.17)]{AGZ}. Hence, for any conjugation invariant probability measure $\tilde{\mu}$ on $U(k)$ with density $f$ with respect to $\mu_H$ (which means that $f=\tilde{f}\circ p$ for some measurable function $\tilde{f}:T_k\rightarrow \mathbb{R}$), we have 
$$dp_{\star}\tilde{\mu}(\vec{u})=\frac{\left\vert V\left(\vec{u}\right)\right\vert^2}{(2\pi)^k} \tilde{f}(\vec{u})d\vec{u},$$
on $T_k$. In particular, $\tilde{f}(\vec{u})d\vec{u}$ satisfies \eqref{eq:spherical_integral_xi(mu)}. Let us denote by $\mathcal{M}_1( T_{k})$ the closure in the weak topology of the set of positive measurable functions $\tilde{f}$ on $T_k$ such that \eqref{eq:spherical_integral_xi(mu)} is satisfied with $d\left[\xi_n(\mu)\right](\vec{u})$ replaced by $f(\vec{u})d\vec{u}$. Note that the set of extreme points of the convex set $\mathcal{M}_1( T_{k})$ includes the weighted Dirac masses $\frac{(2\pi)^k}{\vert V(\vec{u})\vert^2}\delta_{\vec{u}}$ for elements $\vec{u}$ of $T_k$ with $\vert V(\vec{u})\vert>0$. As a consequence, $\xi_n\left[\mathcal{M}_{1}(B_{k,n})\right]\subset\mathcal{M}_1(T_k)$.

For any element $\mu$ of $\mathcal{M}_1(T_k)$, there exists a unique conjugation invariant probability measure $p^{\star}\mu$ on $U(k)$ (meaning that $p^{\star}\mu(UEU^*)=p^{\star}\mu(E)$ for all $U\in U(k)$ and $E$ measurable set of $U(k)$) such that 
$$p_{\star}\left[p^{\star}\mu\right]=\mu,$$
so that $p^{\star}$ yields an embedding of $\mathcal{M}_1(T_k)$ into the set of conjugation invariant probability measures on $U(k)$, see \cite[Chapter 4]{AGZ}. Dirac masses on $\mathcal{M}_1(T_k)$ correspond through $p^{\star}$ to Dirac masses on orbits $\mathcal{O}(\vec{u})=\{U\in U(k),\, p(U)=\vec{u})$ for $\vec{u}\in T_k$. In particular, the set $\mathcal{M}_{1}(B_{k,n})$, which had been seen previously as a discrete subset of $\mathcal{M}_1(T_k)$, can now be considered through $p^{\star}$ as a set of conjugation invariant probability measures on $U(k)$. 

The relation between $B_{k,n}$ and $U(k)$ is actually much stronger. There exists a natural convolution product $*$ on the set of probability measures on $U(k)$ which is characterized for $\mu,\nu$ probability measures on $U(k)$ by the formula
$$\int_{U(k)}f(U)d\left[\mu*\nu\right](U)=\int_{U(k)\times U(k)}f(UV)d\mu(U)\otimes d\mu(V)$$
for any continuous map $f:U(k)\rightarrow \mathbb{R}$. Conjugating by $p^\star$ turns this product into a convolution product, also denoted by $*$, on the set $\mathcal{M}_1( T_{k})$.

The set $(\mathcal{M}_1(B_{k,n}),*)$ can be seen as a $n$-step discrete version of the space $(\mathcal{M}_1( T_{k}),*))$, although the former is not a subgroup of the latter.  However, it has been proven in \cite[Proposition 7.1]{De} that the convolution algebra $(\mathcal{M}_1(B_{k,n}),*)$ converges towards $(\mathcal{M}_1( T_{k},*))$ as $n$ goes to infinity in some sense. Indeed, for $\lambda,\mu\in  T_{k}$ which are decreasing, set $\lambda^{(n)}=\left\{\left\lfloor \frac{n\lambda_i}{2\pi}\right\rfloor\right\}_{1\leq i\leq k}$ for $n\geq 1$ and similarly for $\mu^{(n)}$, so that $\xi_{n}(\lambda^{(n)})$ (resp. $\xi_n(\mu^{(n)})$) converges to $\lambda$ (resp. $\mu$) as $n\rightarrow \infty$. Then, $\lambda^{(n)}, \mu^{(n)}\in B_{k,n}$ for $n$ large enough and, as $n\rightarrow \infty$,
$$\lim_{n\rightarrow \infty}\xi_n\left(\frac{e_{\lambda^{(n)}}}{\mu^h(\lambda^{(n)})}*_{B_{k,n}}\frac{e_{\mu^{(n)}}}{\mu^h(\mu^{(n)})}\right)=\delta(\lambda)*_{ T_{k}}\delta(\mu)
=p_{\star}(\delta_{\mathcal{O}(\lambda)}*_{U(k)}\delta_{\mathcal{O}(\mu)}),$$
where we wrote $*_F$ to specify that the convolution occurs in the group $\mathcal{M}_1(F)$.

\subsection{Main probabilistic theorems}\label{subsec:main_probabilistic_results}
Using the convolution product introduced in Proposition \ref{prop:result_1}, we can consider the convolution $\mu=\mu_1*\mu_2*\dots\mu_m$ of an arbitrary number of $h$-probability measures and its image by the map $\xi_n$. As already shown in \cite{De,We} in the independent and identically distributed case, we expect that there exists $\gamma>0$ depending on some second moments of $(\mu_i)_{1\leq i\leq m}$ such that $\mu$ is close to the eigenvalue distribution of a unitary Brownian motion at time $\gamma\frac{ m}{n^2}$ as $n$ goes to infinity and under the appropriate scaling. Note that we need some care to properly define the unitary Brownian motion, since there is a two-parameters family of bi-invariant Brownian motions on $U(k)$ because of the semisimple decomposition $U(k)\simeq SU(k)\times U(1)$. 

\subsubsection{Unitary Dyson Brownian motion}\label{subsubsec:Dyson_motion}

Recall that the $n^2$-dimensional Lie algebra $\mathfrak{g}_{U(k)}$ of $U(k)$, which is the tangent space of the manifold $U(k)$ at the identity, admits the decomposition 
$$\mathfrak{g}_{U(k)}=\mathfrak{g}_{SU(k)}\oplus \mathfrak{g}_{U(1)},$$
where $\mathfrak{g}_{SU(k)}$ is the Lie algebra of $SU(k)$ and $\mathfrak{g}_{U(1)}=i\mathbb{R}$ is the Lie algebra of $U(1)$. 
Each previous simple Lie algebra respectively admits a natural scalar product called the (negative) Killing form (see Section \ref{Sec:Brownian_motion} for an explicit expression). From these scalar products, one can define $\mathfrak{B}_{SU(k)}$ (resp. $\mathfrak{B}_{U(1)}$) as the unique Brownian motion on $\mathfrak{g}_{SU(k)}$ (resp. $\mathfrak{g}_{U(1)}$) whose covariance at time $t=1$ is the identity in an orthonormal basis of $SU(k)$ (resp. $U(1)$) with respect to its aforementioned scalar product.

For any $\alpha,\gamma\geq 0$ (with $(\alpha,\gamma)\not=0$), there exists then a unique bi-invariant Markov process $(\mathbf{B}^{\alpha,\gamma}(t))_{t\geq 0}$\nomenclature{$(\mathbf{B}^{\alpha,\gamma}(t))_{t\geq 0}$}{Brownian motion with parameters $(\alpha,\gamma)$ on $U(k)$} on $U(k)$ starting at $Id$ which is obtained from the Brownian motion $\gamma \mathfrak{B}_{SU(k)}\oplus \alpha \mathfrak{B}_{U(1)}$ on $\mathfrak{g}_{U(k)}$ by the so-called wrapping procedure (see \cite[Section 1.4]{L}, a more analytic approach is given in Section \ref{Sec:Brownian_motion}). By taking the eigenvalues of $\mathbf{B}^{\alpha,\gamma}(t)$, one gets a stochastic process $B^{\alpha,\gamma}(t)=p\left(\mathbf{B}^{\alpha,\gamma}(t)\right)$\nomenclature{$B^{\alpha,\gamma}(t)$}{$p\left(\mathbf{B}^{\alpha,\gamma}(t)\right)$} on $ T_{k}$ for $t>0$. This stochastic process is actually a continuous Markov process, whose generator is explicitly given in \eqref{eq:generator_Brownian_Motion}. In the particular case $\alpha=\gamma=1$, the process $B^{1,1}$ is just the stochastic process obtained by conditioning $k$ Brownian motions on $\mathbb{R}/(2\pi\mathbb{Z})$ to never intersect by a Doob h-transform (with the harmonic function being $\vec{u}\mapsto \left\vert V\left(\vec{u}\right)\right\vert$), see \cite{CL,HW}. This process is also called the unitary Dyson Brownian motion. The kernel $K^{(1,1)}_t(\vec{u},\vec{u}')$ of the Markov process $B^{1,1}(t)$ has a simple determinantal formula (see \cite[Proposition 1.1]{LW}) given by
$$K^{(1,1)}_t(\vec{u},\vec{u}')=\frac{\vert V(\vec{u}')\vert}{\vert V(\vec{u})\vert}\det(P_t(u_i,u'_j))_{1\leq i,j\leq k},$$
where $P_t(u,v)=\sqrt{\frac{k}{2\pi t}}\sum_{\ell\in \mathbb{Z}}(-1)^{\ell(k+1)}e^{-\frac{k(v-u+\ell 2\pi)^2}{2t}}$.  There is a natural repulsive behavior of the Brownian motions conditioned to not intersect, seen for example in the determinantal formula displayed above in the case $(\gamma,\alpha)=(1,1)$. Hence, as long as $\gamma>0$, $B^{\alpha,\gamma}(t)$ stays almost surely in $ T_{k}^{\circ}$ for $t>0$, where $T_k^{\circ}=\{\vec{u}\in T_k, u_1>u_2>\dots>u_k\}$.

When $\alpha=0$, the process $B^{0,\gamma}$ stays on the subdomain $\{\vec{u}\in  T_{k}, \sum_{i=1}^ku_i=0\,[2\pi]\}$, which can be identified with an alcove of type $A_{k-1}$. The distribution of this stochastic process can be translated on $T_k$ by any scalar $c\in\textbf{R}$ through the action $c\mapsto R_c$, yielding a Markov process on all of $T_k$ which we also call $B^{0,\gamma}$: it just corresponds to the projection $p\circ\mathbf{B}^{0,\gamma}$  when $\mathbf{B}^{0,\gamma}$ is started at some scalar multiple of the identity. Then, the Markov kernel $P^{\gamma}(x,\cdot)$ of the extended process does admit a density with respect to the Lebesgue measure on the subset $\{\vec{y}\in  T_{k}, \sum_{i=1}^ky_i=\sum_{i=1}^kx_i\,[2\pi]\}$. We denote by $K^{SU(k)}_{t}(x,y)$ this density. \nomenclature{$K^{SU(k)}_{t}$}{Kernel of $B^{0,1}$ on the subdomain $\{\vec{x},\vec{y}\in  T_{k}, \sum_{i=1}^ky_i=\sum_{i=1}^kx_i\,[2\pi]\}$}

\subsubsection{Moments of a h-probability measure} Let us introduce here several quantities which will be helpful in stating our main results. In the classical Berry-Esseen theorem, the speed of convergence towards the Gaussian distribution of a normalized sum of independent random variables is bounded by their third moment. We thus need to introduce similar quantities in our setting which is a bit different from the one on the real line. 

First, since the limiting distribution will be the marginal of the aforementioned radial unitary Brownian motion, which depends on two parameters, it is natural to expect two kinds of second and third moments. The second particularity of this model is the cyclic nature of the position. On $\mathbb{R}$, the barycenter of $k$ particles is straightforwardly given by the mean of their distance to zero. On a circle however, there are several possible barycenters depending on the choice of the origin of the circle. We chose one which is well suited for the enumerative application of the present paper and other choices can be easily adapted for other applications with a more probabilistic flavor.

Given a $h$-probability measure $\mu\in \mathcal{M}_{1}(B_{k,n})$, we denote by $I_{\mu}$\nomenclature{$I_{\mu}$}{Random variable on $B_{k,n}$ defined for $\mu\in\mathcal{M}_{1}(B_{k,n})$ following the distribution $\mathbb{P}(I_{\mu}=I)=\mu(I)\mu^h(I)$} a random element of $B_{k,n}$ whose distribution is $\mathbb{P}(I_{\mu}=I)=\mu(I)\mu^h(I)$, and by $\mathbb{E}_{\mu}$ the expectation with respect to $\mu$, namely for $f:B_{k,n}\rightarrow \mathbb{R}$,
$$\mathbb{E}_{\mu}(f(I_{\mu}))=\sum_{I\in B_{k,n}}\mu^h(I)\mu(I)f(I).$$ 
For $I\in B_{k,n}$, define
$$\tilde{I}=I-\frac{\langle I\rangle}{k}\mathbf{1}_k.$$
Remark that $\tilde{I}$\nomenclature{$\tilde{I}$}{$I-\frac{\langle I\rangle}{k}\mathbf{1}_k$} belongs to the $(k-1)$-dimensional subspace $\left\{\vec{u}\in \mathbb{R}^k, \sum_{i=1}^k u_i=0\right\}$. For example, for $I=I_0$ we have $\widetilde{I_0}=\left(\frac{k-1}{2},\frac{k-3}{2},\ldots,-\frac{k-1}{2}\right)$. Then, for $\mu \in\mathcal{M}_1(B_{k,n})$ and $p\geq 1$, we write 
$$\langle\mu\rangle=\mathbb{E}_{\mu}(\langle I_\mu\rangle),\quad Var_p(\mu)=\mathbb{E}_{\mu}\left[\left(\langle I_{\mu}\rangle-\langle\mu\rangle\right)^p\right],\quad \Vert \tilde{\mu}\Vert_p=\mathbb{E}\left[\left\Vert \widetilde{I_{\mu}}\right\Vert^p\right].$$ 
\nomenclature{$\langle\mu\rangle$}{$\mathbb{E}_{\mu}(\langle I_\mu\rangle)$}\nomenclature{$Var_p(\mu)$}{$\mathbb{E}_{\mu}\left[\left(\langle I_{\mu}\rangle-\langle\mu\rangle\right)^p\right]$}\nomenclature{$\Vert \tilde{\mu}\Vert_p$}{$\mathbb{E}\left[\left\Vert \widetilde{I_{\mu}}\right\Vert^p\right]$}Given a sequence $\mathbf{m}=(\mu_i)_{1\leq i\leq m}$ of $h$-probability measures on $B_{k,n}$ and $X$ one of the previously defined parameters on $h$-probability measures, we write 
$$X(\mathbf{m})=\left(\frac{1}{m}\sum_{i=1}^mX(\mu_i)\right)-X(I_0).$$ 
\nomenclature{$X(\mathbf{m})$}{Average of the value of $X(\mu)-X(I_0)$ for a sequence $\mathbf{m}=(\mu_i)_{1\leq i\leq m}$ of $h$-probability measures}
Hence, for example, $\langle\mathbf{m}\rangle=\frac{1}{m}\sum_{i=1}^m\langle\mu_i\rangle-\langle I_0\rangle$.
 Finally, we set 
\begin{equation}\label{eq:product_m_definition}
*\mathbf{m}=\mu_1*\mu_2*\dots*\mu_m.
\end{equation}
\nomenclature{$*\mathbf{m}$}{Convolution of a sequence $\mathbf{m}=(\mu_i)_{1\leq i\leq m}$ of $h$-probability measures}
\subsubsection{Berry-Esseen theorem} We can now state our first main result, which is a precise estimate of the convergence of a large convolution of h-probability measures $(\mu_i)_{1\leq i\leq m}$ towards the marginal of $B^{\alpha,\gamma}$ at time $t$ for some parameters $\alpha,\gamma,t$ depending on the sequence. This Berry-Esseen like result is given in terms of the $1$-Wasserstein distance, which is both tractable and natural to compare discrete measures to continuous ones. Recall that the $1$-Wasserstein distance $W_1^E$ on the set of probability measures on a metric space $(E,d)$ is defined as 
$$W_1^{E}(\mu,\nu)=\sup\left\{\left\vert \int_E fd\mu-\int_E fd\nu\right\vert,\quad f:E\rightarrow \mathbb{R}\text{ $1$-Lipschitz }\right\}.$$
\nomenclature{$W_1^{E}$}{$1$-Wassertein distance on the metric space $E$}In our case, we consider $T_k$ viewed as a symmetrized torus with the natural induced metric, namely with the distance 
$$d_{ T_{k}}(z,z')=\inf_{\sigma\in S_k}d_{(\mathbb{Z}/2\pi\mathbb{Z})^k}(z,\sigma\cdot z'),\quad z,z'\in T_{k},$$
where $d_{(\mathbb{Z}/2\pi\mathbb{Z})^k}$ is the usual metric on the torus $(\mathbb{Z}/2\pi\mathbb{Z})^k$, $S_k$ denotes the symmetric group on ${1,\ldots,k}$ and $(\sigma.z')_i=z'_{\sigma(i)}$ for $1\leq i\leq k$. Remark also that the action of $\mathbb{R}$ on $T_k$ extends to an action of $\mathbb{R}$ on $\mathcal{M}_1( T_{k})$ given by 
$$R_t[\mu](A)=\mu(R_t( A))$$
for $t\in\mathbb{R}$, $\mu\in \mathcal{M}_1(T_k)$ and $A$ measurable set of $T_k$, with $R_t(A)=\{R_t(x), x\in A\}$. 

The quantitative convergence theorem can then be written in the following simple form (the reader should refer to Theorem \ref{thm:Berry-Esseen} for a more precise and more general result), where we recall that $\xi_n$ is the embedding of $B_{k,n}$ in $T_{k}$ defined in Section \ref{subsec:intro_model} and we identify a random variable with its distribution.

\begin{theorem}[Berry-Esseen Theorem, simplified version]\label{thm:Berry-Esseen_simplified}
Suppose that $\mathbf{m}=(\mu_r)_{1\leq r\leq m}$ is a sequence of h-probability measures on $ B_{k,n}$ and set $M=\frac{\Vert \tilde{\mathbf{m}}\Vert_3}{k^3\Vert \tilde{\mathbf{m}}\Vert_2^{3/2}}$. Then, with $t_0=\frac{(2\pi)^2m}{n^2}$,
$$W_1^{ T_{k}}\left(R_{\frac{m\left\langle \mathbf{m}\right\rangle}{kn}}\left[\xi_{n}\left(*\mathbf{m}\right)\right],  B^{\alpha,\gamma}\left(t_0\right)\right)\leq C\frac{M\log(n)}{n},$$
where
$$\alpha=\frac{Var_2(\mathbf{m})}{2k^2},\,\gamma=\frac{k\Vert \tilde{\mathbf{m}}\Vert_2}{2(k^2-1)}$$ 
and $C$ depends on $k$, $t_0$ and $\alpha,\gamma$, see \eqref{eq:exact_bound}. 
\end{theorem}
Remark that $C$ depends decreasingly on $t_0$ and thus this theorem is useful only when $m$ is of order $n^2$ or larger. By taking appropriate Lipschitz functions, one can deduce from the previous results convergence estimates for probability of open sets of $ T_{k}$: namely for an open set $\mathcal{O}$ of $T_k$
$$\left\vert\mathbb{P}\left(I_{*\mathbf{m}}\in O\right)-\mathbb{P}\left(B^{\alpha,\gamma}\left(\frac{m}{n^2}\right)\in O\right)\right\vert\leq C_{O}\sqrt{\frac{\log n}{n}},$$
with $C_O$ depending on $O$. The theorem of the next section, which provides a quantitative local limit theorem, would improve this speed to $O\left(\frac{1}{n}\right)$ in the special case where $\alpha=0$.   

In comparison, the usual central limit theorem would only state that for each $O$ open set of $T_k$, we have $\left\vert\mathbb{P}\left(I_{*(\mathbf{m}_n)}\in O\right)-\mathbb{P}\left(B^{\alpha,\gamma}\left(t\right)\in O\right)\right\vert=o(1)$ as $n$ goes to infinity for some sequence of constant $tn^2$-uples of h-probability measures $\mathbf{m}_n=(\mu^n,\ldots,\mu^n)$ (where $\mu^n$ is a fixed $h$-probability measure in $B_{k,n}$ for each $n\geq 0$) with $\Vert \tilde{\mathbf{m}}_n\Vert_{2}, \,Var_2(\mathbf{m}_n)$ converging to some constant and $\Vert \tilde{\mathbf{m}}_n\Vert_{3},\, Var_3(\mathbf{m}_n)$ uniformly bounded. The fact that we only require a bound on moments of order $3$ to get a central limit theorem is not at all obvious for constrained random walks. In general, the minimal order of moment which is required to be finite to get a central limit theorem should depend on the geometry of the domain, and may be much higher than $3$, see \cite{DeWa15}. The fact that our central limit theorem holds for such minimal moment assumptions (which could actually be lower to require only bound on moments of order $2+\epsilon$ for some $\epsilon>0$) is due to the high symmetry of the considered transition kernels.
\begin{example}
For a concrete application of Theorem \ref{thm:Berry-Esseen_simplified}, let us consider the case where for all $1\leq i\leq m$, $\mu_i(I)=\frac{s}{\mu^h(I^{+})}e_{I^{+}}+\frac{t}{\mu^h(I^{-})}e_{I^{-}}$ with $s+t=1$ and $I^+=(k+1,k-1,\ldots,2)$, $I^-=(k,k-1,\ldots,2,0)$. Then, by Proposition \ref{prop:result_1} and the fact that the convolution algebra $L^2(B_{k,n})$ is stable by conjugation (see also Remark \ref{rem:qcoho_convolution}), 
$$\mu_i\ast y=\frac{1}{\lambda}\left[S(sA+tA^t)\right]^h(y)$$
for all $y\in L^2(B_{k,n})$, where $S$ is the shift operator introduced in Remark \ref{rem:qcoho_convolution}. Thus, up to the shift operator, convolution by $\mu_i$ corresponds to the graph of transitions of $k$ particles on $n$ sites, where at each time a uniformly randomly chosen particle jumps to the right with probability $t$ and to the left with probability $s$ (with the usual conditioning on non-intersection of the particles), see Figure \ref{Fig:transion_gd}. Beware that the transition probabilities of our model are different to the ones of ASEP, since we are conditioning the particles to not intersect instead of directly defining the transition probabilities which prevent the intersection. There is a repulsive behavior in the conditioned model which is not present in the ASEP situation, and the dynamics of both processes are very different.
\begin{figure}[h!]
\includegraphics[scale=1]{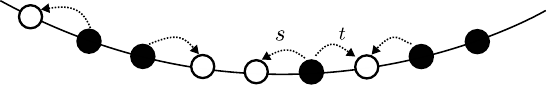}
\caption{\label{Fig:transion_gd} Particles hopping to the left with probability $t$ and to the right with probability $s$, conditioned to not intersect.}
\end{figure}

Then, Theorem \ref{thm:Berry-Esseen_simplified} shows a bound of $O(\frac{\log n}{n})$ on the $1$-Wasserstein distance between the law of the positions of the particles after rescaling and the distribution of $k$ non-intersecting Brownian motions on a circle conditioned to not intersect each other.
\end{example}
\subsubsection{Local limit theorem}
The second main result is a local limit theorem for the decomposition of the h-probability measure $*\mathbf{m}$ defined in \eqref{eq:product_m_definition} on the basis $\mathcal{B}_0=\{e_I\}_{I\in B_{k,n}}$. As usual for local limit theorems with discrete distributions, there may be some issues regarding the lattice generated by the distribution of each $\mu_i$. If the situation may be handled easily when all the distributions are the same (meaning that $\mu_i=\mu$ for all $1\leq i\leq k$), the problem is trickier when the random variables are not identically distributed, see \cite{Mu}. To simplify, we will assume that $Var_2(\mu_i)=0$ for all $1\leq i\leq k$. Beware that it does not implies that $\mu_i$ is a Dirac mass, but rather that $\langle I_{\mu_i}\rangle$ is constant. 

\begin{theorem}[Local limit Theorem]\label{thm:local-limit}
Suppose that $\mathbf{m}=(\mu_r)_{1\leq r\leq m}$ is a sequence of h-probability measure on $ B_{k,n}$ such that $Var_{2}(\mathbf{m})=0$. Then, for $I,I'\in B_{k,n}$,
$$\frac{1}{n}\left(\left(\ast \mathbf{m}\right)\ast \delta_I\right)[I']= \delta_{\langle I'\rangle=\langle I\rangle+\frac{m\langle\mathbf{m}\rangle}{k}[n]}K^{SU(k)}_{\gamma t_0}\left(\xi_{n}(I'),R_{\frac{- m\langle\mathbf{m}\rangle}{kn}}\left[\xi_n\left(I\right)\right]\right)+O\left(\frac{1}{n}\right),$$
where $\gamma=\frac{k\Vert \mathbf{m}\Vert_2}{k^2-1}$, $t_0=\frac{(2\pi)^2m}{n^2}$ and $K^{SU(k)}$ is the Markov kernel defined in Section \ref{subsubsec:Dyson_motion}, and $O(\cdot)$ depends on $k, t_0$ and $\Vert \mathbf{m}\Vert_i, i=2,3$.
\end{theorem}
In the case of a sequence of $h$-probability measures with $Var_2(\mathbf{m})=0$, the latter results implies the Berry-Essen type result of the previous theorem. However, in the general case the contribution of the fluctuations of each $\langle I_{\mu_i}\rangle$ prevents from easily giving a local limit theorem, as periodicity patterns may appear. Since this is mainly a one-dimensional issue, one may expect that a more general local limit theorem should hold and a Berry-Esseen theorem without the $\log n$ factor in the error term.

\subsection{A geometric corollary on quantum cohomology of the Grassmannian}
The graph $B_{k,n}$ and its convolution algebra are related to various other fields, out of which we singled out two main ones: the Wess-Zumino-Witten models of type $A_{k-1}$ at level $n$, which is a particular model of conformal field theory, and the quantum cohomology of the Grassmannian $G_{k,n}$, which counts rational curves with target space the space of $k$-dimensional vector subspaces of $\mathbb{C}^n$. Those two concepts are actually related \cite{G,GK,KS}. Let us first briefly sketch the second one (see Appendix \ref{sec:quantum_cohomology} for a more detailed introduction). 

Denote by $G_{k,n}$ the space of $k$-dimensional vector subspaces of $\mathbb{C}^n$. This space is a compact manifold of dimension $k(n-k)$ and even a projective variety, and thus we can study it from a cohomological perspective (which roughly amounts to counts the number of $\ell$-dimensional holes in $G_{k,n}$, $0\leq \ell \leq k(n-k)$). For any flag $W=0\subset W_1\subset\dots\subset W_n=\mathbb{C}^n$ there exists a natural decomposition of $G_{k,n}$ into cells $\Omega_{I}^W, I \in B_{k,n}$, where 
$$\Omega_I^W=\left\{ V\in G_{k,n}, \dim(V\cap W_{n-I_i})\geq i, 1\leq i\leq k\right\},$$ 
and $\Omega_I^W$ is a closed manifold of $G_{k,n}$ of dimension $\langle I\rangle-\langle I_0\rangle$ called Schubert variety. An important property of the family $\{\Omega^W_I, I\in B_{k,n}\}$ is that a generic submanifold $M$ of $G_{k,n}$ is completely characterized up to homotopy by
$$\{\#(M\cap  \Omega_I^W), \dim \Omega^W_I+\dim M=\dim G_{k,n}\},$$
see \cite{Fulton1997}.
We are interested in the set of maps from $\mathbb{P}^1(\mathbb{C})$ to $G_{k,n}$ whose values at generic points are restricted to lie on some submanifolds of $G_{k,n}$. Let $Hol(G_{k,n})$ be the set of rational maps from $\mathbb{P}^1(\mathbb{C})$ to $G_{k,n}$. We define the degree $\deg f$ of $f\in Hol(G_{k,n})$ as
$$\deg f=\#\{z\in \mathbb{P}^1(\mathbb{C}), f(z)\cap H\not=\emptyset, \quad H\in G_{n-k-1,n} \text{ generic}\}.$$
We should understand $\deg f$ as the number of singular points of $f$ (as the number of poles of a rational function, which actually corresponds to the special case $k=1$, $n=2$, see Appendix \ref{sec:quantum_cohomology}). For $d\geq 0$, let us write $Hol_d(G_{k,n})$ as the space of elements of $Hol(G_{k,n})$ of degree $d$. Up to some compactification, we can again give to $Hol_d(G_{k,n})$ the structure of a projective variety of dimension $k(n-k)+dn$. The enumerative problem related to the quantum cohomology is to count elements of $Hol_d(G_{k,n})$ having prescribed singular points. Namely, fix $a_0,\ldots,a_{m+1}$ generic points of $\mathbb{P}^1(\mathbb{C})$ and $M_0,\ldots,M_{m+1}$ $m+1$ generic sub-manifolds of $\mathbb{C}^n$ of respective dimension $1\leq d_0,\ldots,d_{m+1}\leq k(n-k)$. Set 
$$\mathcal{M}^{d,a_0,\ldots,a_{m+1}}_{M_0,\ldots,M_{m+1}}:=\{f\in Hol_d(G_{k,n}), f(a_i)\in M_{i}, 0\leq i\leq m+1\}.$$
\nomenclature{$\mathcal{M}^{d,a_0,\ldots,a_{m+1}}_{M_0,\ldots,M_{m+1}}$}{$\{f\in Hol_d(G_{k,n}), f(a_i)\in M_{i}, 0\leq i\leq m+1\}$}In general, $\mathcal{M}^{d,a_0,\ldots,a_{m+1}}_{M_0,\ldots,M_{m+1}}$ is non empty only if $\sum_{i=0}^{m+1} d_i\leq \dim Hol_d(G_{k,n})=k(n-k)+dn$ and then the resulting manifold $\mathcal{M}^{d,a_0,\ldots,a_{m+1}}_{M_0,\ldots,M_{m+1}}$ has dimension $k(n-k)+dn-\sum_{i=0}^{m+1}d_i$. In the particular case where $k(n-k)+dn=\sum_{i=0}^{m+1}d_i$, $\mathcal{M}^{d,a_0,\ldots,a_{m+1}}_{M_0,\ldots,M_{m+1}}$ is generically a discrete manifold, which can be shown to be finite (see Section \ref{sec:quantum_cohomology}). We address in this section the problem of counting $\#\mathcal{M}^{d,a_0,\ldots,a_{m+1}}_{M_0,\ldots,M_{m+1}}$ in the case $k(n-k)+dn=\sum_{i=0}^{m+1}d_i$ when there is a large number of constraints $M_i$. To this end, for $M$ a submanifold of dimension $d$ of $G_{k,n}$, let us decompose the homology class $[M]\in H_d(G_{k,n})$ of $M$ on the basis $\left\{\left[\Omega_I^W\right], I\in B_{k,n}\right\}$ as 
$$[M]=\sum_{\substack{I\in B_{k,n}, \langle I\rangle=d+\langle I_0\rangle}}c_I^M\left[\Omega_{I}^W\right].$$ 
We then define the \textit{quantum dimension} of $M$ as 
\begin{equation}\label{eq:definition_qDim}
qDim(M)=\sum_{I\in B_{k,n}, \langle I\rangle=d+\langle I_0\rangle}c_I^MS_I(\xi_{n}(I_0)),
\end{equation}
\nomenclature{$qDim(M)$}{$\sum_{I\in B_{k,n}, \langle I\rangle=d+\langle I_0\rangle}c_I^MS_I(\xi_{n}(I_0))$}where $S_I(\vec{u})$ is the evaluation at $(e^{iu_1},\ldots,e^{iu_k})$ of the Schur function associated to the partition $\lambda_I=(I_1-(k-1),\ldots ,I_k)$ (see Appendix \ref{sec:quantum_cohomology} for more details on Schur functions), and write $p_{I}^M=\frac{c^{M}_IS_{I}(\xi_n(I_0))}{qDim M}$ for $I\in B_{k,n}$, so that $\sum_{I\in B_{k,n}}p_{I}^M=1$ by \eqref{eq:definition_qDim}. Then, introduce for $a=2,3$ the statistic
$$\Vert M\Vert_a= \frac{1}{qDim(M)}\sum_{I\in B_{k,n}, \langle I\rangle=d+\langle I_0\rangle}p_I^M\sum_{j=1}^k(I_j-(d+\langle I_0\rangle)/k)^a.$$
In the following statement, we write $I^c$ for the dual configuration $I^c=(n-1-I_{k+1-j})_{1\leq j\leq k}$ of $I\in B_{k,n}$.
\begin{corollary}\label{cor:cohomology}
Let $M_0,M_1,\ldots, M_{m+1}$ be generic sub-manifolds of $G_{k,n}$ having respective dimension $d_0,\ldots,d_{m+1}$ and $d\geq 0$ such that $\sum_{i=0}^{m+1}d_i=k(n-k)+dn$. Then, for generic $a_0,\ldots,a_{m+1}\in \mathbb{P}^1(\mathbb{C})$ and setting $ x^n_I=R_{-\frac{\sum_{i=1}^md_i}{kn}}\left[\xi_n(I)\right]$ and $y^n_{I'}=\xi_n((I')^c)$
$$\#\mathcal{M}^{d,a_0,\ldots,a_{m+1}}_{M_0,\ldots,M_{m+1}}=\frac{V(\xi_{n}(I_0))^2\prod_{i=0}^{m+1}qDim(M_i)}{n^{k-1}}\left(\sum_{I,I'\in B_{k,n}}p_I^{M_0}p_{I'}^{M_{m+1}}K_{\gamma t_0}^{SU(k)}\left(x^n_I,y^n_{I'}\right)+O\left(\frac{1}{n}\right)\right),$$
where $\gamma=\frac{k\left[\left(\frac{1}{m}\sum_{i=1}^m\Vert M_i\Vert_2 \right)-\langle I_0\rangle_2\right]}{k^2-1}$, $t_0=\frac{(2\pi)^2m}{n^2}$ and $O(\cdot)$ depends on $k, \frac{m}{n^2}$ and $\left(\frac{1}{m}\sum_{i=1}^m\Vert M_i\Vert_a \right)$, $a=2,3$.
\end{corollary}
Similarly to the previous local limit theorem, the latter result is non trivial only when $m$ is of order $n^2$ or larger. In the particular case where $M_0=\Omega^{W_0}_I$ and $M_{m+1}=\Omega^{W_{m+1}}_{I'}$ for some generic flags $W_0,W_{m+1}$, the latter corollary yields for $d$ such that $\sum_{i=0}^{m+1}d_i-dn-k(n-k)=0$
$$\#\mathcal{M}^{d,a_0,\ldots,a_{m+1}}_{M_1,\ldots,M_m}=\frac{\vert V(\xi_{n}(I))V(\xi_{n}(I'))\vert\prod_{i=0}^{m+1}qDim(M_i)}{n^{k-1}}\left(K_{\frac{m}{n^2}}^{SU(k),\gamma}\left(x_I,y_{I'}\right)+O\left(\frac{1}{n}\right)\right).$$
\subsection{Related results and perspectives}
Let us give several comments on our results and their relation to previous works on the subject :
\begin{itemize}
\item An analogous problem has already been considered on the real line instead of the circle, see \cite{Bi,BG}. In this case, one still has a positively multiplicative graph whose corresponding algebra is the fusion ring for the representation theory of $U(k)$. The continuous limit of this discrete model corresponds to convolution of co-adjoint orbits of $U(k)$ instead of adjoint orbits \cite{Gra,BJ}, and there is a central limit theorem towards a Dyson Brownian motion on the real line. Up to our knowledge, no quantitative central limit theorem has been obtained in this framework; it is however expected that the method used in the present paper could be adapted to the real line.
\item There exists a natural quotient of $B_{k,n}$ which describes the movement of a particle on a discrete lattice of type $A_k$ conditioned to stay in a macroscopic alcove of size $n$. The graph is still positively multiplicative and the corresponding algebra is the fusion algebra of type $A_k$ at level $n$. Such processes have been intensively studied in the wider setting of affine semisimple Lie algebras by Defosseux \cite{De}. In particular, our result in the particular case $\mu_r=\mu$ for $1\leq r\leq m$ and $n,m$ going to infinity is similar to \cite[Theorem 8.1]{De}. The latter theorem also mention a Donsker's type theorem for the sequence $(\xi_n(\mu_1*\dots*\mu_i))_{1\leq i\leq m}$. Such results could also be obtained for general sequences $\mathbf{m}$ from our results by imposing that $(\Vert \mu_i\Vert_2)_{1\leq i\leq m}$ is constant and $(\Vert \mu_i\Vert_3)_{1\leq i\leq m}$ is bounded as $m,n$ goes to infinity. We mention also that asymptotic results in the simplest case of a random walk in an alcove have been obtained before by Grabiner \cite{Gra2}.
\item Considering nonintersecting particles on the discrete circle conditioned to not intersect is not new. A first work in this direction has been done by Fulmek \cite{Fulm} in the case where only one particle at a time jumps to a free next location: this correspond to considering the adjacency matrix of $B_{k,n}$ only, and not the other commutating adjacency matrices $A_I$ for $I\not=(k,k-2,\ldots,0)$. The author then proved exact formulas for the transition probabilities. This work has been later generalized by Krattenthaler to other small steps cases and similar graphs (corresponding to other affine Lie algebras) \cite{Kra}. In this paper, the author also obtained in the small step case local limit theorems similar to the ones of the present paper.

Beware that the stochastic processes studied above and in the present paper are different from the  the periodic TASEP with $k$ particles on a circle of size $n$: indeed, the TASEP is a continuous time Markov process whose transition matrix is given by the adjacency matrix of $B_{k,n}$, whereas our model is a discrete time Markov process whose Markov kernel is a Doob h-transform of this adjacency matrix. The two Markov processes have very different behaviors. For example, the stationary distribution of TASEP is the uniform distribution on $B_{k,n}$ whereas the stationary distribution of our model is given by $\mu^h(I)=\frac{\vert V(\xi_n(I))\vert^2}{n^k}$.
\item The cylindric hexagonal dimer model is a particular instance of the present framework by considering the transition matrix $A=\sum_{i=1}^{n-k}A_{I^{(i)}}$ with $I^{(i)}=(k-1+i,k-2,\ldots,0)$ (this is a consequence of the quantum Pieri rules for $QH(G_{k,n})$ from \cite[p.5]{B} and the bijection between hexagonal dimer models and nonintersecting paths, see \cite[Figure 1]{KW}). This relation has already been deeply studied when considering the configuration on the real line \cite{BG}, and it is expected that such correspondence also holds on the circle. See also \cite{BdT} for results on the winding number for the cylindric hexagonal dimer model when $k,n$ are large. 
\item The relation between nonintersecting particles on the circle and the unitary group has already been considered from a probabilistic perspective in the continuous case by Metcalfe, O'Connell and Warren \cite{MOW}. They build Markov processes on $T_k$ by using the convolution of a family of commuting Markov kernels which are continuous version of our Markov kernel $A_I^h$. They also relate this construction to the convolution of orbits on $U_k$. Our work can be seen as a discrete version of theirs.
\item Regarding the geometric application of our results, we mention that there have already been some combinatorial results on counting rational functions with prescribed critical points (see \cite{Sch}). From a quantum cohomological perspective, the relation between the quantum cohomology of the Grassmannian and the orbits of the unitary group goes back to \cite{Bel1}. It would be very interesting to pursue this relation from a probabilistic point of view. Note also that the quantum cohomology of the Grassmannian also appear in the study of the so-called rigid system, see \cite{Bel2}.
\end{itemize}
Our work is a first step towards two main interesting and challenging directions :
\begin{itemize}
\item The first natural question is the behavior of the system when $k$ grows with $n$. The probability distribution at time $\frac{t}{k}$ of a Dyson unitary Brownian motion with parameter $(1,1)$ converges in some sense toward the marginal $\mu^{\boxtimes t}$ of the free unitary Brownian, see \cite{Bi2} for a formal definition of the latter and a proof. It can be easily deduced from our work that under good normalization hypotheses, $\xi_n(\mathbf{m})$ is close to $\mu^{\boxtimes t}$ for some $t>0$ when $k=O(\log n)$. However, it is expected that such approximation holds for $k=O(n)$. It would be very interesting to study this problem and related questions (see also \cite{LW} for the distribution of the winding number in the continuous case when $k$ goes to infinity).
\item The graph $B_{k,n}$ is a particular instance of a Kirillov-Reshetikhin crystal, and there is a full family of similar graphs which extend the present framework to other Kirillov-Reshetikhin crystals, see \cite{GLT2}. They corresponds each time to movements of colored particles on a circle or a segment conditioned to not intersect, sometimes with some additional symmetries. We expect that our results could be extended to this larger setting in some cases: the main difficulty is to get a explicit diagonalization of the adjacency matrix $A$, which is not obvious in the more general setting (see \cite{Kos} for an approach to this problem). Another interesting direction would be to consider graphs related to the equivariant quantum cohomology of the Grassmannian instead of the usual one, see \cite{Bu}: the corresponding graph describes particles moving on a discrete circle and conditioned to not intersect, but with jumping rates depending on the location. Once again, the main challenge is to explicitly diagonalize the adjacency matrix; we expect that such a task would be manageable when all jumping rates but one are equal.
\end{itemize}
\section{Positively multiplicative graphs, random walks and Fourier transform}\label{sec:pm_graph}
In this section, we introduce the notion of integrable positively multiplicative graph which is the central property of the family of graph $\{B_{k,n}; 1\leq k\leq n\}$. We also introduce a general Fourier transform formalism to study a class of Markov chains naturally constructed from those graphs. Although some results of the present section are needed in the sequel of the paper, its content is quite general to enlighten a universal behavior.

Let us recall the definition of a positively multiplicative graph from \cite{GLT}. In this paper, a (directed) graph is the data of a finite set $V$ and a adjacency matrix $A$ on the vector space $L^2(V)=\bigoplus_{v\in V} e_v$, by which we mean an operator $A\in\mathcal{B}(L^2(V))$ such that $Ae_v=\sum_{v'\in V}A_{v'v}e_{v'}$ with $a_{vv'}\geq 0$ for all $v,v'\in V$. We add a Hilbert space structure on $L^2(V)$ with the scalar product $\langle e_v,e_{v'}\rangle=\delta_{vv'}$. The Hilbert space $L^2(V)$ can be seen as the algebra of $L^2$-functions on $V$ with respect to the counting measure, and for $x\in L^2(V)$, we write $x(v)$ for the coordinate of $x$ along $e_v$.
\begin{definition} [See \cite{GLT}]\label{def:PM-graph}
A finite graph $\Gamma=(V,A)$ is said positively multiplicative (abbreviated PM) at $v_0\in V$ if there is a commutative algebra $\mathcal{A}$ such that $\mathbb{C}[A]\subset \mathcal{A}\subset End(L^2(V))$ and such that $\mathcal{A}$ admits a basis $\mathcal{B}=\{s_v,v\in V\}$ with $s_{v}e_{v_0}=e_v$,
$$As_v=\sum_{v'\in V}A_{v'v}s_{v'},\, s_vs_{v'}\in\sum_{v''\in V}\mathbb{R}^+s_{v''}.$$ 
A PM graph is said integrable if $A$ is diagonalizable on $L^2(V)$.
\end{definition}
The algebra $\mathcal{A}$ is called the enveloping algebra of $\Gamma$ in the sequel. Remark that in general $\mathcal{A}$ is bigger that $\mathbb{C}[A]$, since $A$ may for example have eigenvalues with multiplicities larger than $1$. For a PM-graph there is a isomorphism 
\begin{equation}\label{eq:isomor_L2_CA}
S:\left\lbrace \begin{aligned} L^2(V)&\rightarrow &\mathcal{A}\\ e_v&\mapsto &s_v\end{aligned}\right.
\end{equation}
(which reflects the fact that $\dim \mathcal{A}=|V|$. It can then be deduced that $s_{v_0}=Id_{L^2(V)}$ (see \cite[Theorem 3.12]{GLT}). 
\subsection{Family of random walks and $h$-transform}
Given a PM-graph with distinguished vertex $v_0$, one can construct a family of commuting Markov kernels as follows. For sake of simplicity we assume here that $A$ is irreducible, although the same reasoning apply beyond irreducibility. Let $h^r$ (resp. $h^l$) be the right (resp. left) Perron-Frobenius vector of $A$, normalized so that $h^l(v_0)=1$ and $\langle h^l,h^r\rangle_{L^2(V)}=1$. By irreducibility of $A$, $h^l(v)>0$ and $h^r(v)>0$ for all $v\in V$. In the case where we are considering an adjacency matrix $M$ which is not necessarily irreducible, we say that $v$ is a positive eigenvector of $M$ if $v$ is an eigenvector of $M$ which has positive coefficients in its expansion in the canonical basis.

\begin{lemma}\label{lem:Perron_sv}
For all $v\in V$, $h^l$ and $h^r$ are respectively left and right eigenvectors of $s_v$ with eigenvalue $h^l(v)$.
\end{lemma}
\begin{proof}
The fact that $h^l$ and $h^r$ are eigenvectors of $s_v$ is deduced from the fact that both vectors are eigenvectors of $A$ and $s_v$ commutes with $A$. Hence, they are left and right positive eigenvectors of $s_v$, and thus correspond to the same eigenvalue we denote temporarily by $\lambda$. On the one hand, we have by definition of $s_v$
$$\langle h^l,s_v e_{v_0}\rangle=\langle h^l,e_v\rangle=h^l(v),$$
and on the other hand,
$$\langle  h^l,s_v e_{v_0}\rangle=\langle s_v^*h^l,e_{v_0}\rangle=\lambda\langle h^l,e_{v_0}\rangle=\lambda.$$
We conclude that 
$$\lambda=h^l_v.$$
\end{proof}
It will be convenient to introduce the Hadamard product on vectors $x,y\in L^2(V)\mapsto x\cdot y\in L^2(V)$ defined by 
$$(x\cdot y)(v)=x(v)y(v).$$
If $x$ has nonzero entries, we write $\frac{1}{x}$ for the vector having entry $\frac{1}{x(v)}$ at $v$.
\begin{definition}\label{def:h_transform}
The $h$-space $L^2_h(V)$ is the vector space $L^2(V)$ endowed with the Hilbert product
$$\langle x,y\rangle_h=\langle h^l\cdot x,h^r\cdot y\rangle=\sum_{v\in V}\mu^h(v)\overline{x(v)}y(v),$$
with $\mu^h=h^l\cdot h^r$ called the invariant probability measure of $A$.

The $h$-transform of an operator $B\in End(L^2(V))$ is the operator $B^h$ given by 
$$B^h(x)=\frac{1}{h^r}\cdot B(h^r\cdot x).$$
\end{definition}
The space $L^2_h(V)$ can be also seen as the Hilbert space of $L^2$-functions on $V$ with respect to the probability measure $\mu^h$.

Since $h^r$ has positive entries, the map $B\mapsto B^h$ is well-defined. Its is actually an algebra isomorphism of $End(L^2(V))$ which amounts to a conjugation by a diagonal matrix; we denote by $\mathcal{A}^h$ the algebra $\{B^h, B\in \mathcal{A}\}$, which is isomorphic to $\mathcal{A}$. It is readily seen that 
$$(B^h)^*(x)=h^r\cdot B^*\left(\frac{1}{h^r}\cdot x\right),$$
while taking the adjoint $(\cdot)^{h,*}$ of $B$ with respect to $\langle \cdot,\cdot\rangle_h$ yields 
$$B^{h,*}(x)=\frac{1}{\mu^h}\cdot B^*\left(\mu^h\cdot x\right).$$
For $v\in V$, let us write 
$$P^v=\frac{s_v^h}{h^l(v)}\in \mathcal{A}^h.$$
\begin{proposition}
The set $\left\{P^v, v\in V\right\}$ is a family of commuting Markov kernels on $L^2(V)$ with common invariant measure $\mu_h$.
\end{proposition}
\begin{proof}
Let $v\in V$ and denote by $\mathbf{1}\in V$ the vector with $\mathbf{1}(v)=1$ for all $v\in V$. Remark first that by Lemma \ref{lem:Perron_sv},
$$P^v(\mathbf{1}_V)=\frac{1}{h^l(v)}\frac{1}{h^r}\cdot s_v(h^r)=\frac{1}{h^r}\cdot h^r=\mathbf{1}.$$
Then, writing $s_{v}s_{v'}=\sum_{v''\in V}c_{vv'}^{v''}s_{v''}$ with $c_{vv'}^{v''}\geq 0$ for all $v,v',v''\in V$, we have
$$P^v(e_{v'})=h^r(v')\frac{1}{h^r}\cdot s_v(s_{v'}e_{v_0})=\sum_{v''\in V}h^r(v')c_{vv'}^{v''}\left(\frac{1}{h^r}\cdot s_{v''}(e_{v_0})\right)=\sum_{v''\in V}\frac{h^r(v')}{h^r(v'')}c_{vv'}^{v''}e_{v''}.$$
Since $\frac{h^r_{v'}}{h^r_{v''}}c_{vv'}^{v''}\geq 0$ for all $v,v',v''\in V$, $P^v$ is an adjacency matrix and thus a Markov kernel on $L^2_h(V)$. Since by Lemma \ref{lem:Perron_sv}
$$(s_v^h)^*[\mu^h]=h^r\cdot s_v^*\left(\frac{1}{h^r}\cdot \mu^h\right)=h^r\cdot s_v^*(h^l)=h^l(v)h^r\cdot h^l=h^l(v)\mu^h,$$
the vector $\mu_h$ is a positive eigenvector of $P^v$ with eigenvalue $1$. Hence, $\mu_h$ is an invariant measure of $P^v$.
\end{proof}
Remark that we can similarly view the set $\left\{P^v, v\in V\right\}$ as a family of commuting Markov kernels on $L^2_h(V)$ with common invariant measure having density $\mathbf{1}$ with respect to $\mu^h$. We have indeed similarly
$$(P^v)^{h,*}[\mathbf{1}]=\frac{1}{\mu^h}h^r\cdot \frac{s_v^*}{h^l(v)}\left(\frac{1}{h^r}\cdot \mu^h\cdot \mathbf{1}\right)=\frac{1}{h^l}\cdot \frac{s_v^*}{h^l(v)}(h^l)=\frac{1}{h^l}\cdot h^l=\mathbf{1}.$$
\subsection{Convolution product on $L^2_h(V)$}
Similarly to \eqref{eq:isomor_L2_CA} there is an isomorphism from $L^2_h(V)$ to $\mathcal{A}^h$ given by
\begin{equation}\label{eq:isomor_L2h_CA}
M^h:\left\lbrace \begin{aligned} L^2(V)&\rightarrow &\mathcal{A}^h\\ e_v&\mapsto &\mu^h(v)P^v\end{aligned}\right..
\end{equation}
We can then transpose the multiplicative structure on $\mathcal{A}^h$ to a multiplication on $L^2_h(V)$. This motivates the following definition.
\begin{definition}\label{def:convolution}
The convolution of $x,y\in L^2_h(V)$ is the vector $x*y$ defined by
$$x*y=\sum_{v\in V}\mu^h(v)x(v)P^v(y)=\sum_{v\in V}h^r(v)x(v)s_v^h(y).$$
A vector $u\in L^2_h(V)$ with nonnegative real entries is called a $h$-probability measure if 
$$\langle \mathbf{1},u\rangle_h=1.$$
\end{definition} 
We denote by $\mathcal{M}_1(V)$ the set of $h$-probability measures. We have then the following results analogous to the usual convolution on classical groups.
\begin{proposition}\label{prop:structure_conv}
The convolution product turns $L^2_h(V)$ into an algebra isomorphic to $\mathcal{A}$ through the map $M^h$. Moreover, the set $\mathcal{M}_1(V)$ is stable by convolution.
\end{proposition}
\begin{proof}
To prove the first statement, it is enough to prove that $M^h$ intertwines $*$ and the operator product on $\mathcal{A}^h$. By linearity, it is enough to prove it on $\frac{e_v}{h^r(v)},\frac{e_{v'}}{h^r(v')}$ for $v,v'\in V$. Let $v,v'\in V$. Then, using that $\mu^h=h^l\cdot h^r$ and $P^v=\frac{s^h_v}{h^l(v)}$,
\begin{align*}
\frac{e_v}{h^r(v)}*\frac{e_{v'}}{h^r(v')}=\sum_{v''\in V}\left(h^r\cdot \frac{e_v}{h^r(v)}\right)(v'')s_{v''}^h\left(\frac{e_{v'}}{h^r(v')}\right)=s_v^h\left(\frac{e_{v'}}{h^r(v')}\right)=&\frac{1}{h_r}\cdot s_v(e_{v'}).
\end{align*}
Since $e_{v'}=s_{v'}e_{v_0}$, writing $s_{v}s_{v'}=\sum_{v''\in V}c_{vv'}^{v''}s_{v''}$ we get 
$$s_{v}(e_v)=\sum_{v''\in V}c_{vv'}^{v''}e_{v''}.$$
Hence,
$$\frac{e_v}{h^r(v)}*\frac{e_{v'}}{h^r(v')}=\sum_{v''\in V}c_{vv'}^{v''}\frac{e_{v''}}{h^r(v'')}.$$
Since we have $M^h\left(\frac{e_v}{h^r(v)}\right)=s_v^h$, we deduce that 
$$M^h\left(\frac{e_v}{h^r(v)}*\frac{e_{v'}}{h^r(v')}\right)=M^h\left(\frac{e_v}{h^r(v)}\right)M^h\left(\frac{e_{v'}}{h^r(v')}\right).$$
Hence, $(L^2_h(V),+,*)$ is an algebra which is isomorphic to $\mathcal{A}^h$ (and thus to $\mathcal{A}$).

Suppose that $x,y\in \mathcal{M}_1(V)$. Then,
\begin{align*}
\langle \mathbf{1},x\ast y\rangle_h=&\sum_{v\in V}\mu^h(v)x(v)\langle \mathbf{1},P^v(y)\rangle_h\\
=&\sum_{v\in V}\mu^h(v)x(v)\left\langle h^l,h^r\cdot \frac{1}{h^r}\cdot \frac{s_v(h^r\cdot y)}{h^l(v)}\right\rangle=\sum_{v\in V}h^r(v)x(v)\left\langle h^l,s_v(h^r\cdot y)\right\rangle.
\end{align*}
By Lemma \ref{lem:Perron_sv}, $h^l$ is an eigenvector of $s_v^*$ with eigenvalue $h^l(v)$ for all $v\in V$, thus, using the fact that $x,y\in\mathcal{M}_{1}(V)$,
\begin{align*}
\langle \mathbf{1},x\ast y\rangle_h=&\sum_{v\in V}h^r(v)x(v)\left\langle s_v^*h^l,h^r\cdot y\right\rangle\\
=&\sum_{v\in V}\mu^h(v)x(v)\left\langle h^l,h^r\cdot y\right\rangle=\sum_{v\in V}\mu^h(v)x(v)\left\langle \mathbf{1},y\right\rangle_{h}=\langle \mathbf{1},x\rangle_{h}=1.
\end{align*}

\end{proof}

\subsection{Integrable PM-graph, Fourier transform and Verlinde formula}\label{sec:integrable_graph}
In this section, we assume that $A$ is diagonalizable over $\mathbb{C}$. As a direct consequence of commutativity, all $s_v, \, v\in V$ are diagonalizable. We denote by $\{u^{(i)},i=0,\ldots,|V|-1\}$\nomenclature{$u^{(i)}$}{Right eigenvector for an integrable PM graph with vertex set $V$} a basis of right eigenvectors and $\{w^{(i)},i=0,\ldots,|V|-1\}$ \nomenclature{$w^{(i)}$}{Left eigenvector with $w^{(i)}(v_0)=1$ for an integrable PM graph with vertex set $V$} the dual basis of left eigenvectors normalized so that $w^{(i)}(v_0)=1$ (as we will see in the next lemma, any eigenvector has a nonzero coordinate along $e_{v_0}$). We suppose that $u^{(0)}=h^r$ (and thus $w^{(0)}=h^l$). Finally, we denote by $\theta_v(i)$ the eigenvalue of $s_v$ on $u^{(i)}$, so that $\overline{\theta_v(i)}$ is the eigenvalue of $s_v$ on $w^{(i)}$). As in Lemma \ref{lem:Perron_sv}, the coordinates of $w^{(i)}$ along $\mathcal{B}_0$ are actually given by the eigenvalues of the family $s_v$.
\begin{lemma}\label{lem:eigenvalue_eigenvector_relation}
For all $i\in \{0,\ldots,|V|-1\}$,
$$w^{(i)}(v)=\overline{\theta_v(i)}.$$
\end{lemma}
\begin{proof}
Let $y$ be a left eigenvector with eigenvalue $\theta_v(i)$. Then, since $s_v(e_{v_0})=e_v$, 
$$\langle y,e_v\rangle=\langle s_v^*y,e_{v_0}\rangle=\theta_v(i)\langle y,e_{v_0}\rangle.$$
In particular, $\langle y,e_{v_0}\rangle\not=0$ for otherwise we would have $y=0$. Specifying to $y=w^{(i)}$ yields then
$$w^{(i)}(v)=\overline{\theta_v(i)}.$$
\end{proof}
Remark that $s_{v}^h$ is also diagonalizable on $L^2_h(V)$ with same eigenvalues $\theta_v(i),\overline{\theta_v(i)}$ and left and right eigenvectors
$$\tilde{u}^{(i)}=\frac{1}{h^r}\cdot u^{(i)},\, \tilde{w}^{(i)}=\frac{1}{h^l}\cdot w^{(i)}.$$
\nomenclature{$\tilde{u}^{(i)}$}{$\frac{1}{h^r}\cdot u^{(i)}$}\nomenclature{$\tilde{w}^{(i)}$}{$\frac{1}{h^l}\cdot w^{(i)}$} (beware that left eigenvectors are taken with respect to the adjoint $(\cdot)^{h,*})$).
 \begin{definition}\label{def:Fourier_transform}
The Fourier transform $\Phi_V$ on $L^2_h(V)$ is the morphism from $L^2_h(V)$ to $\mathbb{C}^{|V|}$ defined by 
$$\Phi_V[x](i)=\left\langle \tilde{w}^{(i)},x\right\rangle_h.$$
\end{definition}\nomenclature{$\Phi_V$}{Fourier transform on $L^2_h(V)$.}.
With this definition, a vector $u\in L^2_h(V)$ is a $h$-probability measure if and only if $\Phi_V[u](0)=1$.
\begin{proposition}\label{prop:diagonalization_Fourier}
For any $x,y\in L^2_h(V)$,
$$\Phi_V[x\ast y](i)=\Phi_V[x](i)\Phi_V[y](i).$$
Moreover, $\Phi_V$ is an isomorphism with inverse
$$\Phi_V^{-1}(e_{i})=\tilde{u}^{(i)},$$
where $\{e_i\}_{1\leq i\leq \vert V\vert}$ is the canonical basis of $\mathbb{C}^{\vert V\vert}$.
\end{proposition}
\begin{proof}
Let $x,y\in L^2_h(V)$ and $0\leq i\leq \vert V\vert-1$. Then, using the definition of the convolution product and the fact that $\tilde{w}^{(i)}$ is an eigenvector of $(s_v^h)^{h,*}$ with eigenvalue $\overline{\theta_v^i}$,
\begin{align*}
\Phi_V[x\ast y](i)=&\langle \tilde{w}^{(i)},x\ast y\rangle_h\\
=&\sum_{v\in V}(h^r\cdot x)(v)\left\langle \tilde{w}^{(i)},s_v^h(y)\right\rangle_h\\
=&\sum_{v\in V}(h^r\cdot x)(v)\left\langle (s_v^h)^{h,*}(\tilde{w}^{(i)}),y\right\rangle_h\\
=&\sum_{v\in V}(h^r\cdot x)(v)\theta_v(i)\Phi_V[y](i).
\end{align*}
Then, by Lemma \ref{lem:eigenvalue_eigenvector_relation},
\begin{align*}
\Phi_V[x\ast y](i)=\Phi_V[y](i)\sum_{v\in V}(h^r\cdot x)(v)\overline{w^{(i)}(v)}=&\Phi_V[y](i)\sum_{v\in V}(h^r\cdot x)(v) \overline{(h^l\cdot \tilde{w}^{(i)})(v)}\\
=&\Phi_V[y](i)\left\langle \tilde{w}^{(i)},x\right\rangle_h=\Phi_V[y](i)\Phi_V[x](i),
\end{align*}
which proves the first assertion. Then, for $0\leq i,j\leq \vert V\vert-1$,
$$\Phi_V\left[\tilde{u}^{(i)}\right](j)=\left\langle \tilde{w}^{(j)},\tilde{u}^{(i)}\right\rangle_h=\delta_{ij},$$
so $\Phi_V$ is an isomorphism with inverse 
$$\Phi_V^{-1}(e_i)=\tilde{u}^{(i)}.$$
\end{proof}
The latter proposition amounts to say that $\Phi_V$ diagonalizes the convolution product on $L^2_h(V)$. We can deduce from the above construction a general formula to compute arbitrary convolutions in $L^2_h(V)$. 
\begin{proposition}[Verlinde formula]\label{prop:Verlinde_formula_general}
Let $n\geq 2$. Then, for $v\in V$ and $x_1,\ldots,x_n\in L^2_h(V)$
$$\left[x_1*\dots *x_n\right](v)=\sum_{i=0}^{|V|-1}\prod_{j=1}^n\Phi_V[x_j](i)\tilde{u}^i(v),$$
and for $v_1,\ldots,v_n\in V$,
$$s_{v_1}\dots s_{v_n}=\sum_{v\in V}c_{v_1,\ldots,v_n}^vs_v$$ 
with 
\begin{equation}\label{eq:Verlinde_formula_general}
c_{v_1,\ldots,v_n}^v=\sum_{i=0}^{|V|-1}\prod_{j=1}^n\theta_{v_j}(i)u^{(i)}(v).
\end{equation}
\end{proposition}
\begin{proof}
Since $(\tilde{w}^{(i)})_{0\leq i\leq |V|-1}$ is the dual basis of $(\tilde{u}^{(i)})_{0\leq i\leq |V|-1}$ with respect to $\langle \cdot,\cdot\rangle_h$, we have
$$x_1*\dots *x_n=\sum_{i=1}^{|V|-1}\langle \tilde{w}^{(i)},x_1*\dots *x_n\rangle_h \tilde{u}^{(i)}=\sum_{i=1^{|V|-1}}\Phi_V[x_1*\dots *x_n](i) \tilde{u}^{(i)}.$$
By Proposition \ref{prop:diagonalization_Fourier}, we then have 
$$x_1*\dots *x_n=\sum_{i=1}^{|V|-1}\prod_{j=1}^n\Phi_V[x_j](i)\tilde{u}^{(i)},$$
and the first statement is deduced. For the second statement, remark that we also have
$$s_{v_1}^h\dots s_{v_n}^h=\sum_{v\in V}c_{v_1,\ldots,v_n}^vs_v^h$$
with the same coefficients (see Definition \ref{def:h_transform} and the remark afterwards). By Proposition \ref{prop:structure_conv} and the first statement,
\begin{align*}
s_{v_1}^h\dots s_{v_n}^h=&M^h\left(\frac{1}{h^r(v_1)}e_{v_1}*\dots *\frac{1}{h^r(v_n)}e_{v_n}\right)\\
=&\frac{1}{\prod_{j=1}^nh^r(v_j)}\sum_{i=1}^{|V|-1}\prod_{j=1}^n\Phi_V[e_{v_j}](i)M^h(\tilde{u}^i).
\end{align*}
On the one hand, for $1\leq j\leq n$, we have
$$\Phi_V[e_{v_j}](i)=\langle \tilde{w}^{(i)},e_{v_j}\rangle_h=\mu^h(v_j)\overline{\tilde{w}^{(i)}(v_j)}=h^r(v_j)\overline{w^{(i)}(v_j)}=h^r(v_j)\theta_{v_j}(i)$$
by Lemma \ref{lem:eigenvalue_eigenvector_relation}. On the other hand,
$$M^h(\tilde{u}^{(i)})=\sum_{v\in V}h^r(v)\tilde{u}^{(i)}(v)s^h_{v}=\sum_{v\in V}u^{(i)}(v)s^h_{v}.$$
Hence,
$$s_{v_1}^h\dots s_{v_n}^h=\sum_{v\in V}\left(\sum_{i=1}^{|V|-1}\prod_{j=1}^n\theta_{v_j}(i)u^{(i)}(v)\right) s_{v}^h,$$
and we thus get 
$$c_{v_1,\ldots,v_n}^v=\sum_{i=0}^{|V|-1}\left(\prod_{j=1}^n\theta_{v_j}(i)\right)u^{(i)}(v).$$
\end{proof}
\begin{remark}
The formula \eqref{eq:Verlinde_formula_general} is called the Verlinde formula, since it is exactly the same formula that appears in the special case where the graph comes from the fusion ring in conformal field theory, see \cite{Beauv}. Let $\{v_j\}_{0\leq j\leq |V|-1}$ be an enumeration of the vertices of $V$, with $v_0$ the root of the PM graph as before, and write $S=\left(\theta_{v_i}(j)\right)_{0\leq i,j\leq |V|-1}$. Then, by Lemma \ref{lem:eigenvalue_eigenvector_relation} and the definition of the dual basis $u^{(i)}$
$$\left(S^{-1}\right)_{ij}=u^{(i)}(v_j).$$
Hence, \eqref{eq:Verlinde_formula_general} specializes in the more common formula $s_{v_i}s_{v_j}=\sum_{k=0}^{|V|-1}N_{ij}^ks_{v_k}$ with
$$N_{ij}^k=\sum_{l=0}^{|V|-1}S_{il}S_{jl}\left(S^{-1}\right)_{lk}.$$
Note that a term $\frac{1}{S_{0l}}$ often appearing in the literature disappears in our case due to our normalization of the vectors $w^{(i)}$, $0\leq i\leq |V|-1$. In the special case of the quantum cohomology of the Grassmannian which is the subject of the present paper, this formula is known as Vafa and Intriligator’s formula, see \cite{B}.
\end{remark}

When $A$ is normal with respect to $\langle \cdot,\cdot\rangle$, then $A$ is diagonalizable in an orthogonal basis and, with our convention, we can choose 
\begin{equation}\label{eq:normal_adjacency}
u^{(i)}=\frac{1}{\Vert w^{(i)}\Vert^2}w^{(i)}
\end{equation}
for $0\leq i\leq |V|-1$. In particular, $h^r=\frac{1}{\Vert h^l\Vert^2}h^l$ and thus 
$$\mu^h=\frac{1}{\Vert h^l\Vert^2}h^l\cdot h^l.$$
In this case, $\lbrace \tilde{u}_i, 1\leq i\leq \vert V\vert\rbrace$ is an orthogonal basis of $L^2_h(V)$ with $\Vert \tilde{u}_i\Vert_{h}^2=\frac{\Vert h^l\Vert^2}{\Vert w^{(i)}\Vert^2}$ and then $\Phi_V$ is an isometry from $L^2_h(V)$ to $\mathbb{C}^{\vert V\vert}$ endowed with the Hilbert space structure $\langle e_i,e_j\rangle=\delta_{ij}\frac{\Vert h^l\Vert^2}{\Vert w^{(i)}\Vert^2}$.

\section{Noncolliding particles on a circle and eigenvalues of unitary matrices}\label{sec:formal_model_pm}

% Ici \Lambda_{k} sont des poids de U_n et B_k les poids shiftés. On autorise donc des entiers négatifs.
Recall from Section \ref{sec:intro_model} the definition of $B_{k,n}$. Due to the relation between $B_{k,n}$ and the quantum cohomology of the Grassmannian whose ring structure is described by Schur functions, it is sometimes better to see $B_{k,n}$ as a graph of partitions (see Appendix \ref{sec:quantum_cohomology} for some reminders on partitions). We denote by $\Lambda_{k}$ the set of generalized partitions of length $k\geq 1$, which are nondecreasing sequences of integers of length $k$ and $\Lambda_{k,n}$ the set of partitions of length $k$ such that the first part is smaller that $n-k$. There is a bijection $\phi$ from $\Lambda_{k,n}$ to $B_{k,n}$ defined by the shift
$$\phi(\lambda)=\lambda+I_0,$$
where we recall that $I_0=(k-1,\ldots,0)$. By the latter bijection, $\phi(\emptyset)=I_0$ and we then have $\phi^{-1}(I)=I-I_0$. Since we will make an extensive use of the maps $\phi$ and $\phi^{-1}$, we will simply write $I_{\lambda}=\phi(\lambda)$ and $\lambda_{I}=\phi^{-1}(I)$ when there is no ambiguity. Remark that the map $\phi$ is independent of $n$ and extends to a map from $\Lambda_{k}$ to $B_{k}$, where $B_{k}$ denotes the set of increasing sequences of integers of length $k$. The graph structure of $B_{k,n}$ yields a graph structure on $\Lambda_{k,n}$ given by the following rule, which is a specialization of the so-called quantum Pieri rule, see \cite{B}.
\begin{lemma}\label{lem:equivalence_graph_pieri}
There is an oriented edge between $I$ and $J$ in $B_{k,n}$ if and only if either $\lambda_{I}\nearrow \lambda_{J}$ in the usual sense of partition or if $(\lambda_{I})_1=n-k$, $(\lambda_{I})_{k}>0$ and $\lambda_{J}=\widetilde{\lambda_{I}}$, where for $\lambda\in \Lambda_{k,n}$ such that $\lambda_1=n-k$ and $\lambda_k>0$, we write 
$$\tilde{\lambda}=(\lambda_2-1,\ldots,\lambda_k-1,0).$$
\end{lemma}
The proof of the latter lemma is a simple transposition to partitions of the edge rule for $B_{k,n}$, which we do not detail here. Finally, recall from Section \ref{subsec:intro_model} that we also embed $B_{k,n}$ into $ T_{k}$ with the map
$$\xi_n:\left\lbrace \begin{aligned}
&\quad B_{k,n}&\rightarrow &\quad T_{k}\\
& \quad \quad I&\mapsto& R_{\frac{k-1}{2n}} \left(\frac{2\pi}{n}I\right)
\end{aligned}\right.,$$
and we have set $T_{k,n}=\xi_{n}(B_{k,n})$. To summarize, we have then three sets of cardinal $\binom nk$ related together as in the following picture.
\begin{center}
\begin{tikzpicture}
\draw (0,0) node (A) {$T_{k,n}\subset T_{k}$};
\draw (6,0) node (B) {$B_{k,n}\subset B_{k}$};
\draw (3,3) node (C) {$\Lambda_{k,n}\subset \Lambda_{k}$};
\draw [<-](A)--(B);
\draw [<-](B)--(C);
\draw [<->] (C)--(A);
\draw (2.5,-0.3) node {$\xi_{n}$};
\draw (4.7,1.7) node {$\phi$};
\end{tikzpicture}
\end{center}
\subsection{The graph $B_{k,n}$ as an integrable PM-graph}
Recall from Section \ref{sec:intro_model} that $L^2(B_{k,n})=\bigoplus_{I\in B_{k,n}}\mathbb{C} e_{I}$ denotes the Hilbert space of functions on $B_{k,n}$ with the scalar product $\langle e_{I},e_{J}\rangle=\delta_{I,J}$, where $\mathcal{B}_0=(e_{I})_{I\in B_{k,n}}$ is the canonical basis on $L^2(B_{k,n})$. Then, the graph structure on $B_{k,n}$ is encoded in the adjacency matrix $A\in \mathcal{B}(L^2(B_{k,n}))$ given by
$$Ae_I=\sum_{I\nearrow J}e_J.$$

In the sequel, we will use a specialization of the quantum cohomology of the Grassmannian $QH(G_{k,n})$ \nomenclature{$QH(G_{k,n})$}{Quantum cohomology of the Grassmannian $G_{k,n}$} to prove that $B_{k,n}$ is an integrable PM-graph as defined in Definition \ref{def:PM-graph}. The reader should refer to Appendix \ref{sec:quantum_cohomology} for any background on $QH(G_{k,n})$. We mainly need the fact that $QH(G_{k,n})=Sym[x_1,\ldots,x_k]\otimes \mathbb{C}[q]/Sym[x_1,\ldots,x_k]/\langle (x_1,\ldots,x_k)=q^{1/k}\left(e^{it_1},\ldots e^{it_k}\right), \vec{t}\in T_{k,n}\rangle$, see Appendix \ref{sec:quantum_cohomology}, and we will consider the specialization 
$$QH_{1}(G_{k,n}):=QH(G_{k,n})/\langle q=1\rangle=Sym[x_1,\ldots,x_k]/\langle(x_1,\ldots,x_k)=\left(e^{it_1},\ldots e^{it_k}\right), \vec{t}\in T_{k,n}\rangle.$$
The algebra $QH_{1}(G_{k,n})$ is spanned by the basis $\{s_{\lambda},\lambda\in \Lambda_{k,n}\}$, and we denote by $\ell:QH_1(G_{k,n})\rightarrow L^2(B_{k,n})$ the isomorphism defined by 
$$\ell(s_{\lambda})=e_{I_{\lambda}}.$$ 
Then, $B_{k,n}$ is a PM-graph in the sense of Definition \ref{def:PM-graph} with an enveloping algebra $\mathcal{A}\simeq QH_{1}(G_{k,n})$, as the next proposition shows.
\begin{proposition}\label{prop:equiv_graph_qh}
There is a commutative algebra $\mathcal{A}\subset \mathcal{B}(L^2(B_{k,n}))$ containing $\mathbb{C}[A]$ which is isomorphic to $QH_{1}(G_{k,n})$ through an isomorphism $T:QH_{1}(G_{k,n})\rightarrow\mathcal{A}$ such that $T(s_{\lambda})e_{I_0}=e_{I_{\lambda}}$ and $T(s_{(1)})=A$. 

As a consequence $B_{k,n}$ is a PM-graph rooted at $I_0$.
\end{proposition}
\begin{proof}
Remark first that since $\dim L^2(B_{k,n})=\vert B_{k,n}\vert=\binom nk$, for any commutative subalgebra $\mathcal{A}$ of $\mathcal{B}(L^2(B_{k,n}))$ we have $\dim \mathcal{A}\leq \binom nk$. Recall that we defined 
$$QH_1(G_{k,n})=Sym[x_1,\ldots,x_k]/\langle (x_1,\ldots,x_k)=\left(e^{it_1},\ldots e^{it_k}\right), \vec{t}\in T_{k,n}\rangle.$$
Since $T_{k,n}$ has cardinal $\binom nk$ and the map $(t_1,\ldots,t_k)\mapsto\left(e^{it_1},\ldots e^{it_k}\right)$ is injective on $T_{k,n}$, $QH_1(G_{k,n})$ is isomorphic to the space of functions on $\binom nk$ elements and thus $\dim QH_1(G_{k,n})=\binom nk$. Since $QH_1(G_{k,n})$ is spanned by $\{s_{\lambda}\}_{\lambda\in \Lambda_{k,n}}$, by a dimension argument the latter set is actually a basis of $QH_1(G_{k,n})$. Recall that $\ell:QH_1(G_{k,n})\rightarrow L^2(B_{k,n})$ is the isomorphism sending $s_\lambda$ to $e_{I_{\lambda}}$, and define $T:QH_1(G_{k,n})\rightarrow \mathcal{B}(L^2(B_{k,n}))$ by
$$T(x)e_{I}=\ell(x\ell^{-1}(e_I)),\,x\in QH_1(G_{k,n}),\,  I\in B_{k,n}.$$
Then, $T$ is an injective morphism of algebra, since it can be written as the composition of the multiplication map $m:QH_1(G_{k,n})\rightarrow \mathcal{B}(QH_1(G_{k,n}))$ with the conjugation map $C_{\ell}:\mathcal{B}(QH_1(G_{k,n}))\rightarrow \mathcal{B}(L^2(B_{k,n}))$ given by $M\mapsto \ell\circ M\circ \ell^{-1}$, both maps being injective morphisms of algebras. Since $QH_{1}(G_{k,n})$ is commutative, $T(QH_{1}(G_{k,n}))$ is a commutative subalgebra of $\mathcal{B}(L^2(B_{k,n}))$.

Since $\ell^{-1}(e_{I})=s_{\lambda_I}$, we have
$$T(s_{(1)})e_{I}=\ell(s_{(1)}s_{\lambda_{I}}).$$
The multiplication by $s_{(1)}$ in $QH(G_{k,n})$ is given by the quantum Pieri rule \cite[Eq. (22)]{BCF}, which gives for $\lambda\in \Lambda_{k,n}$ after specializing $q=1$
$$s_{(1)}s_{\lambda}=\sum_{\lambda\nearrow \mu}s_{\mu}+\delta_{\lambda_1=n-k,\lambda_k>0}qs_{\tilde{\lambda}}=\sum_{\lambda\nearrow \mu}s_{\mu}+\delta_{\lambda_1=n-k,\lambda_k>0}s_{\tilde{\lambda}},$$
where $\tilde{\lambda}=(\lambda_2-1,\ldots,\lambda_k-1,0)$. Hence, by Lemma \ref{lem:equivalence_graph_pieri}, $s_{(1)}s_{\lambda}=\sum_{I_{\lambda}\nearrow I}s_{\lambda_{I}}$, which yields for $I\in B_{k,n}$
$$T(s_{(1)})e_{I}=\ell(s_{(1)}s_{\lambda_{I}})=\sum_{I\nearrow J}\ell(s_{\lambda_{J}})=\sum_{I\nearrow J}e_{J},$$
and finally $T(s_{(1)})=A$, so that $\mathbb{C}[A]\subset T(QH_{1}(G_{k,n}))$.

For all $I\in B_{k,n}$,
$$T(s_{\lambda_I})e_{I_0}=\ell(s_{\lambda_I}s_{\emptyset})=\ell(s_{\lambda_I})=e_{I},$$
which implies that the map $T(QH_{1}(G_{k,n}))\rightarrow L^2(B_{k,n})$ given by $x\mapsto x(e_{I_0})$ is injective. Then, since $$\left[AT(s_{\lambda_I})\right]e_{I_0}=Ae_{I}=\sum_{I\nearrow J}e_{J}=\left(\sum_{I\nearrow J} T(s_{\lambda_J})\right)e_{I_0},$$
we deduce that 
$$AT(s_{\lambda_I})=\sum_{I\nearrow J} T(s_{\lambda_J})=\sum_{J\in B_{k,n}}A_{JI} T(s_{\lambda_J}).$$
Finally, by \eqref{eq:multiplication_quantum}, for all $\lambda,\mu \in \Lambda_{k,n}$ we have 
$$T(s_{\lambda})T(s_{\mu})=T(s_{\lambda}s_{\mu})=\sum_{d\geq 0,\nu\in \Lambda_{k,n}}\langle\lambda,\mu,\nu^c\rangle_dT(s_{\nu})$$
with each term $\langle\lambda,\mu,\nu^c\rangle_d\geq 0$.

By the previous results, $B_{k,n}$ is a PM-graph whose enveloping algebra is $\mathcal{A}=T(QH_1(G_{k,n}))$ and the corresponding basis is $\{T(s_{\lambda_I}),I\in B_{k,n}\}$.
\end{proof}
As a consequence of the latter result and the presentation of $QH_1(G_{k,n})$, we will prove that $B_{k,n}$ is an integrable PM-graph. Let us simply write $A_{I}=T(s_{\lambda_I})$. In the case of $B_{k,n}$, there is a natural indexation of left and right common eigenvectors of $\{A_I\}_{I\in B_{k,n}}$ by the set $B_{k,n}$ itself. We do not known if such natural indexation holds for any integrable PM-graph. As we will see in the next proposition, we are in the particular case where the adjacency matrix $A$ is normal, so that left and right eigenvectors $u^{(I)}$ and $w^{(I)}$, $I\in B_{k,n}$ of all operators $A_{I}$ coincide up to a constant, see \eqref{eq:normal_adjacency}. For $I\in B_{k,n}$, set
\begin{equation}\label{eq:def_vI}
w^{(I)}=\sum_{J\in B_{k,n}}\overline{S_{J}(\xi_{n}(I))}e_{J},
\end{equation}
where $S_J(\vec{u})=s_{\lambda_J}\left(e^{iu_1},\ldots,e^{iu_k}\right)$ \nomenclature{$S_J(\vec{u})$}{Schur function $s_{\lambda_J}$ evaluated at $(e^{iu_1},\ldots,e^{iu_k})$} for $\vec{u}\in \mathbb{R}^k$, and let $V(u_1,\ldots,u_k)=\prod_{j<j'}\left(e^{iu_{j'}}-e^{iu_{j}}\right)$ be the Vandermonde determinant of $\left(e^{iu_j}\right)_{1\leq j\leq k}$.  
\begin{lemma}\label{lem:action_s_mu}
The set $\{w^{(I)}\}_{I\in B_{k,n}}$ is an orthogonal basis of $L^2(B_{k,n})$, and for all $J\in B_{k,n}$,
$$A_{J}(w^{(I)})=S_{J}(\xi_{n}(I))w^{(I)}\quad
\text{   and    } \quad A_{J}^*(w^{(I)})=\overline{S_{J}(\xi_{n}(I))}w^{(I)}.$$
\end{lemma}
\begin{proof}
The orthogonality relation is given by the following Cauchy type formula from \cite[Proposition 4.3]{Ri}: for $I,I'\in B_{k,n}$,
\begin{equation}\label{eq:first_ortho_rietsch}
\sum_{J \in B_{k,n}}S_{J}(\xi_{n}(I))\overline{S_{J}(\xi_{n}(I'))}=\delta_{I,I'}\frac{n^k}{\vert V(\xi_{n}(J))\vert^2},
\end{equation}
which directly yields
\begin{align}
\langle w^{(I)},w^{(I')}\rangle=&\delta_{I,I'}\frac{n^k}{\vert V(\xi_{n}(J))\vert^2}\label{eq:orthogonality_relation}
\end{align}
 for $I,I'\in B_{k,n}$. 
 
 To prove the eigenvector property, it is easier to work on $QH_1(B_{k,n})$. Indeed, by Proposition \ref{prop:equiv_graph_qh}, it suffices to prove that $s_{\lambda_{J}}\cdot\ell^{-1}(w^{(I)})=S_{J}(\xi_{n}(I))\ell^{-1}(w^{(I)})$, where $\cdot$ denotes the product in $QH_{1}(G_{k,n})$. From the proof of Proposition \ref{prop:equiv_graph_qh}, $QH_{1}(B_{k,n})$ is the space of functions on $\mathcal{Z}_{k,n}=\{\exp(i\vec{u}), \vec{u}\in T_{k,n}\}$, where $\exp(i\vec{u})=\left(e^{iu_1},\ldots,e^{iu_k}\right)$ for $\vec{u}\in \mathbb{R}^k$. Hence, we need to show that the values of $s_{\lambda_{J}}\cdot\ell^{-1}(w^{(I)})$ and $S_{J}(\xi_{n}(I))\ell^{-1}(w^{(I)})$ agree on $\mathcal{Z}_{k,n}$. First, we have by the usual evaluation of multiplication of functions
\begin{align*}
\left[s_{\lambda_{J}}\cdot\ell^{-1}\left(w^{(I)}\right)\right](\exp(i\xi_{n}(I)))=&s_{\lambda_{J}}(\exp(i\xi_{n}(I)))\left[\ell^{-1}\left(w^{(I)}\right)\right](\exp(i\xi_{n}(I)))\\
=&\left[S_{J}(\xi_{n}(I))\ell^{-1}(w^{(I)})\right](\exp(i\xi_{n}(I))).
\end{align*}
Then, for $I'\not=I$,
\begin{align*}
&\left[s_{\lambda_{J}}\cdot\ell^{-1}(w^{(I)})\right](\exp(i\xi_{n}(I')))-\left[S_{J}(\xi_{n}(I))\ell^{-1}(w^{(I)})\right](\exp(i\xi_{n}(I')))\\
&\hspace{3cm}=\left(S_{J}(\xi_{n}(I'))-S_{J}(\xi_{n}(I))\right)\left[\sum_{J'\in B_{k,n}}\overline{S_{J'}(\xi_{n}(I))} s_{\lambda_{J'}}\right](\exp(i\xi_{n}(I')))\\
&\hspace{3cm}=\left(S_{J}(\xi_{n}(I'))-S_{J}(\xi_{n}(I))\right)\sum_{J'\in B_{k,n}}\overline{S_{J'}(\xi_{n}(I))} S_{J'}(\xi_{n}(I'))=0
\end{align*}
by \eqref{eq:orthogonality_relation}. Hence, we deduce that $s_{\lambda_{J}}\cdot\ell^{-1}(w^{(I)})=S_{J}(\xi_{n}(I))\ell^{-1}(w^{(I)})$, and the result is deduced by applying $\ell$. Since $A_{J}$ is diagonalized on an orthogonal basis, it is actually normal and we thus have 
$$A_{J}^*\left(w^{(I)}\right)= \overline{S_{J}(\xi_{n}(I))}w^{(I)}.$$
\end{proof}
Note that as a corollary of \eqref{eq:orthogonality_relation}, we have 
$$\Vert w^{(I)}\Vert_2=\frac{n^{k/2}}{\vert V(\xi_{n}(I))\vert}.$$
Hence, following the notation of Section \ref{sec:pm_graph}, we set 
\begin{equation}\label{eq:def_uI_Bkn}
u^{(I)}=\frac{\vert V(\xi_{n}(I))\vert^2}{n^k}w^{(I)},
\end{equation}
so that $\langle w^{(I)},u^{(J)}\rangle=\delta_{IJ}$.
\subsection{$h$-transform, convolution and Fourier transform on $B_{k,n}$}
As a consequence of the fact that $B_{k,n}$ is an integrable PM-graph and the expression of its eigenbasis, we can deduce several results from Section \ref{sec:pm_graph}. In order to do so, we will need several times the following inversion formula for evaluating $S_{I}(\xi_{n}(J))$ in terms of $S_{J}(\xi_{n}(I))$.
\begin{lemma}\label{lem:formula_inversion_Schur}
For $I,J\in  B_{k,n}$,
$$S_J(\xi_{n}(I))=\frac{S_J(\xi_{n}(I_0))}{S_I(\xi_{n}(I_0))}S_I(\xi_{n}(J)).$$
Moreover,
$$V(\xi_{n}(I))=\exp\left(\frac{(k-1)i\pi}{n}\left(\langle I\rangle-\langle I_0\rangle\right)\right)V(\xi_{n}(I_0))S_I(\xi_{n}(I_0)),$$
where $\langle I\rangle=\sum_{i=1}^kI_i$.
\end{lemma}
\begin{proof}
By definition of Schur functions as ratios of generalized Vandermonde determinants, for $I\in B_{k}$ and $\vec{u}\in \mathbb{R}^k$ we have
$$S_{I}(\vec{u})=\frac{a_{I}(\vec{u})}{a_{I_0}(\vec{u})},$$
where $a_{J}(\vec{u})=\det\left(e^{iu_lJ_m}\right)_{1\leq l,m\leq k}$ for $J\in B_{k}$. For $I,J\in B_{k,n}$,
\begin{align}
a_{J}(\xi_{n}(I))=&\det\left(\exp\left(\frac{2i\pi}{n}(I_l-(k-1)/2)J_m\right)\right)_{1\leq l,m,\leq k}\nonumber\\
=&\exp\left(\frac{(k-1)i\pi}{n}\left(\sum_{l=1}^kI_l-\sum_{m=1}^kJ_m\right)\right)\det\left(\exp\left(\frac{2i\pi}{n}I_l(J_m-(k-1)/2)\right)\right)_{1\leq l,m,\leq k}\nonumber\\
=&\exp\left(\frac{(k-1)i\pi}{n}(\langle I\rangle-\langle J\rangle)\right)\det\left(\exp\left(\frac{2i\pi}{n}I_l(J_m-(k-1)/2)\right)\right)_{1\leq l,m,\leq k}\nonumber\\
=&\exp\left(\frac{(k-1)i\pi}{n}(\langle I\rangle-\langle J\rangle)\right)a_{I}(\xi_{n}(J)).\label{eq:relation_vandermonde}
\end{align}
Applying the latter formula to the ratio $\frac{S_J(\xi_{n}(I))}{S_J(\xi_{n}(I_0))}$ yields
\begin{align*}
\frac{S_J(\xi_{n}(I))}{S_J(\xi_{n}(I_0))}=&\frac{a_{J}(\xi_{n}(I))}{a_{I_0}(\xi_{n}(I))}\frac{a_{I_0}(\xi_{n}(I_0))}{a_{J}(\xi_{n}(I_0))}\\
=&\frac{a_{I}(\xi_{n}(J))}{a_{I}(\xi_{n}(I_0))}\frac{a_{I_0}(\xi_{n}(I_0))}{a_{I_0}(\xi_{n}(J))}\exp\left(\frac{(k-1)i\pi}{n}\left(\langle I\rangle-\langle J\rangle+\langle I_0\rangle-\langle I\rangle+\langle J\rangle-\langle I_0\rangle\right)\right)\\
=&\frac{S_I(\xi_{n}(J))}{S_I(\xi_{n}(I_0))},
\end{align*}
which proves the first statement. Then, using \eqref{eq:relation_vandermonde} with $J=I_0$ yields
$$V(\xi_{n}(I))=a_{I_0}(\xi_{n}(I))=\exp\left(\frac{(k-1)i\pi}{n}(\langle I\rangle-\langle I_0\rangle\right)a_{I}(\xi_{n}(I_0))),$$
and, since $S_I(\xi_{n}(I_0))=\frac{a_{I}(\xi_{n}(I_0))}{V(\xi_{n}(I_0))}$,
$$V(\xi_{n}(I))=\exp\left(\frac{(k-1)i\pi}{n}(\langle I\rangle-\langle I_0\rangle)\right)V(\xi_{n}(I_0))S_I(\xi_{n}(I_0)).$$
\end{proof} 
Specializing \eqref{eq:def_vI} and \eqref{eq:def_uI_Bkn} to $I=I_0$ yields the following Lemma.
\begin{lemma}\label{lem:Perron-Frobenius}
The common left and right positive eigenvectors $h^l, h^r$ of all $A_{I},\, I\in B_{k,n}$ with the normalization $h^l(I_0)=1$, $\left\langle h^l,h^r\right\rangle=1$ are 
$$h^l(I)=S_I(\xi_{n}(I_0)),\quad h^r(I)=\frac{\vert V(\xi_{n}(I_0))\vert}{n^k}\cdot\vert V(\xi_{n}(I))\vert,$$
and the invariant probability measure of $B_{k,n}$ is 
$$\mu^h(I)=\frac{1}{n^k}\left\vert V(\xi_{n}(I))\right\vert^2.$$
\end{lemma}
\begin{proof}
Remark that with the notation of the statement, we have $w^{(I_0)}=h^l$. By \eqref{eq:def_vI}, \eqref{eq:def_uI_Bkn} and Lemma \ref{lem:action_s_mu}, $w^{(I_0)}$ and $u^{(I_0)}$ are both left and right eigenvectors of $A_{I}$ for $I\in B_{k,n}$ with eigenvalue $S_{I}(\xi_{n}(I_0))$. To conclude, we need to prove that $w^{(I_0)}(I)>0$ for all $I\in B_{k,n}$ and $h^r=w^{(I_0)}$  and $u^{(I_0)}=h^r$. Both assumptions are deduced if we prove that $\vert V(\xi_n(I_0))\vert S_I(\xi_{n}(I_0))=\vert V(\xi_{n}(I))\vert$, since then $h^l(I)=\frac{\vert V(\xi_{n}(I))\vert}{\vert V(\xi_n(I_0))\vert}>0$ and $u^{(I_0)}(I)=\frac{\vert V(\xi_n(I_0))\vert^2}{n^k}S_{I}(\xi_{n}(I_0))=\frac{\vert V(\xi_n(I_0))\vert}{n^k}\vert V(\xi_{n}(I))\vert$. 

By Lemma \ref{lem:formula_inversion_Schur}
$$V(\xi_{n}(I_0))S_{I}(\xi_{n}(I_0))=\exp\left(-\frac{(k-1)i\pi}{n}(\langle I\rangle-\langle I_0\rangle)\right)V(\xi_{n}(I)).$$
Then, using that $\langle I_0\rangle=\frac{k(k-1)}{2}$
\begin{align*}
&\exp\left(-\frac{(k-1)i\pi}{n}(\langle I\rangle-\langle I_0\rangle)\right)V(\xi_{n}(I))\\
=&\exp\left(-\frac{(k-1)i\pi}{n}\sum_{j=1}^k(I_j-(k-1)/2)\right)\prod_{1\leq l<m\leq k}\left( e^{2i\pi (I_m-(k-1)/2)/n}-e^{2i\pi (I_l-(k-1)/2)/n}\right)\\
=&\prod_{1\leq l<m\leq k}e^{-\frac{(I_m+I_l-(k-1))i\pi}{n}}\left(e^{2i\pi (I_m-(k-1)/2)/n}-e^{2i\pi (I_l-(k-1)/2)/n}\right)\\
=&(2i)^{k(k-1)/2}\prod_{l<m}\sin(\pi(I_m-I_l)/n).
\end{align*}
Since $I$ is an increasing sequence, $\prod_{l<m}\sin(\pi(I_m-I_l)/n)>0$. In particular, for $I=I_0$ we get
$V(\xi_{n}(I_0))=i^{k(k-1)/2}\vert V(\xi_{n}(I_0))\vert$. We deduce that 
\begin{equation}\label{eq:rel_S_V}
\vert V(\xi_n(I_0))\vert S_I(\xi_{n}(I_0))=\vert V(\xi_{n}(I))\vert,
\end{equation}
and the lemma is deduced.
\end{proof}
As a consequence of the latter lemma, the positive eigenvalue of $A_I, I\in B_{k,n}$ corresponding to the eigenvector $h^r$ is $S_{I}(\xi_n(I_0))$. In particular, Lemma \ref{lem:action_s_mu} yields that for all $I,J\in B_{k,n}$
\begin{equation}\label{eq:inequality_Perron_Frobenius}
\left\vert S_{I}(\xi_n(J))\right\vert\leq S_{I}(\xi_n(I_0)).
\end{equation}
Following notation of Section \ref{sec:formal_model_pm}, set $P^I=\frac{A_{I}^h}{S_{I}(\xi_{n}(I_0))}$. Translating Definition \ref{def:convolution} to $B_{k,n}$ yields the following convolution product on $L^2_h(B_{k,n})$.
\begin{definition}\label{def:dft_conv}
The convolution product on $B_{k,n}$ is the map $*:L^2(B_{k,n})\otimes L^2(B_{k,n})\rightarrow L^2(B_{k,n})$ defined by 
$$x\ast y=\sum_{I\in B_{k,n}}\mu^h(I) x(I)P^I(y)=\sum_{I\in B_{k,n}}\frac{\vert V(\xi_{n}(I))\vert^2}{n^k} x(I)P^I(y),$$
where $P^I(y)=\frac{1}{h^r}\cdot \frac{A_I}{S_I(\xi_{n}(I_0))}(h^r\cdot y)$  and $\cdot$ denotes the pointwise multiplication.

A $h$-probability measure $\mu$ on $B_{k,n}$ is a vector $\mu\in L^2(B_{k,n})$ such that 
$$\sum_{I\in B_{k,n}}\frac{\vert V(\xi_{n}(I))\vert^2}{n^k}\mu(I)=1.$$
\end{definition}
Recall that we denote then by $\mathcal{M}_1(B_{k,n})$ the set of h-probability measures on $B_{k,n}$. We also denote by $\delta_I$ the unique $h$-probability measure supported at $I$, namely
\begin{equation}\label{eq:definition_dirac}
\delta_I=\frac{n^k}{\vert V(\xi_{n}(I)\vert^2}e_I.
\end{equation}
The results of Section \ref{sec:pm_graph} yield then directly the proof of Proposition \ref{prop:result_1}.
\begin{proof}[Proof of Proposition \ref{prop:result_1}]
By Proposition \ref{prop:structure_conv}, $(L^2(B_{k,n}),+,*)$ is a commutative algebra such that $\mathcal{M}_1(B_{k,n})*\mathcal{M}_1(B_{k,n})\subset \mathcal{M}_1(B_{k,n})$, which proves $(1)$ and $(2)$. 

Then, applying Definition \ref{def:dft_conv} to $x=\frac{e_{I^1}}{\mu^h(I^1)}$ and $y\in L^2(B_{k,n})$ and using that $\mu^h(I)=\frac{\vert V(I)\vert^2}{n^k}$ for $I\in B_{k,n}$, we get
$$\frac{e_{I^1}}{\mu^h(I^1)}\ast y=P_{I^1}(y)=\frac{A^h}{\lambda},$$
where the last equality is due to Proposition \ref{prop:equiv_graph_qh} which implies that $A_{I^1}=T(s_{(1)})=A$. Doing the same with $x=\frac{e_I}{\mu^h(I)}$ and $y=\frac{e_{I_0}}{\mu^h(I_0)}$ yields 
$$\frac{e_{I}}{\mu^h(I)}\ast \frac{e_{I_0}}{\mu^h(I_0)}=\frac{1}{\mu^h(I_0)}P^I(e_{I_0})=\frac{h^r(I_0)}{h^r(I)h^l(I_0)h^r(I_0)h^l(I)}e_I=\frac{e_I}{\mu^h(I)},$$
where we used that $A_{I}(e_{I_0})=e_{I}$ and $h^l(I_0)=1$. Thus, (3) and (4) are proven.
\end{proof}
Likewise, since $B_{k,n}$ is an integrable PM-graph, we can define a Fourier transform as in Section \ref{sec:integrable_graph}. 
\begin{definition}\label{def:Fourier_Bkn}
The discrete Fourier transform on $B_{k,n}$ is the map $\Phi_n:L^2(B_{k,n})\rightarrow L^2(B_{k,n})$ defined  for $x\in L^2(B_{k,n})$ by
$$\Phi_n[x](I)=\langle \tilde{w}^{(I)},x\rangle_{h}=\sum_{J\in B_{k,n}}\frac{\vert V(\xi_{n}(J))\vert\cdot\vert V(\xi_{n}(I_0))\vert}{n^k}S_{J}(\xi_{n}(I))x(J).$$
\end{definition}
\nomenclature{$\Phi_n$}{Discrete Fourier transform on $B_{k,n}$}Then, Proposition \ref{prop:diagonalization_Fourier} directly yields for $x_1,\ldots,x_m\in L^2(B_{k,n})$
\begin{equation}\label{eq:relation_conv_dft}
\Phi_n[x_1*\cdots*x_m]=\Phi_n[x_1]\cdots\Phi_n[x_m].
\end{equation}
Let us more specifically study the Fourier transform of a $h$-probability measure. First, for $I,J\in B_{k,n}$, we have by Definition \ref{def:Fourier_Bkn}, \eqref{eq:definition_dirac} and \eqref{eq:rel_S_V},
$$\Phi_n[\delta_I](J)=\frac{\vert V(\xi_{n}(I))\vert\cdot\vert V(\xi_{n}(I_0))\vert}{n^k}S_{I}(\xi_{n}(J))\cdot \frac{n^k}{\vert V(\xi_{n}(I)\vert^2}=\frac{S_{I}(\xi_{n}(J))}{S_I(\xi_{n}(I_0))}.$$

Then, by Definition \ref{def:dft_conv} and \eqref{eq:definition_dirac}, a $h$-probability measure $\mu\in L_h^2(B_{k,n})$ can be written $\mu=\sum_{I\in B_{k,n}}a_{I}\delta_I$
with $a_I\geq 0$ for $I\in B_{k,n}$ and $\sum_{I\in B_{k,n}}a_I=1$. Hence, we have 
\begin{equation}\label{eq:Fourier_probability}
\Phi_n[\mu](J)=\sum_{I\in B_{k,n}}a_{I}\frac{S_{I}(\xi_{n}(J))}{S_I(\xi_{n}(I_0))}.
\end{equation}

\begin{example}
For $k=1$, $B_{1}=\mathbb{Z}$, and we can identify $B_{1,n}$ with $\mathbb{Z}/n\mathbb{Z}$, and $S_{J}(u)=e^{iJu}$ for $u\in \mathbb{R}, J\in \mathbb{Z}$. In particular, $I_0=0$ and $S_{I_0}(u)=1$ for $u\in \mathbb{R}$. For $I\in [0,n[\cap \mathbb{Z}$, $\vert V(\xi_{n}(I))\vert=1$, and thus 
$$w^{(I)}=\sum_{J\in  B_{k,n}}\exp(-2i\pi IJ/n)e_{J},\quad u^{(I)}=\frac{1}{n}\sum_{J\in  B_{k,n}}\exp(-2i\pi IJ/n)e_{J}$$
so that we recover the usual discrete Fourier vectors. For $x\in L^2(B_{k,n})$, we then have 
$$\Phi_{n}[x](I)=\frac{1}{n}\sum_{J\in \mathbb{Z}/n\mathbb{Z}}\exp(2i\pi IJ/n)x(J),$$
which is the usual (inverse) discrete Fourier transform. Finally, the map $P_{I}$ is the circulant operator sending $e_{J}$ to $e_{J+I}$, so that for $x,y\in L^2(B_{1,n})$, the convolution product is
$$x\ast y(J)=\frac{1}{n}\sum_{I\in \mathbb{Z}/n\mathbb{Z}}x(I)y(J-I),$$
which is the classical convolution product on $\mathbb{Z}/n\mathbb{Z}$. Then, \eqref{eq:relation_conv_dft} recovers the fact that 
$$\Phi_n(x\ast y)=\Phi_n(x)\Phi_n(y).$$
\end{example}

\subsection{Fourier transform on $ T_{k}$}\label{Sec:Fourier_continuous}
Let us turn to $T_{k}$, the continuous limit object of the sequence $\{B_{k,n}, n\geq k\}$. Recall that the set of continuous $k$-configurations on a circle $ T_{k}$ is the subset of $\mathbb{R}^k$ consisting of vectors $u=(u_1,\ldots,u_k)$ such that 
$$0\leq u_k\leq u_{k-1}\leq \dots\leq u_1<2\pi.$$
Before describing the relation with $B_{k,n}$, let us describe the Fourier transform on $ T_{k}$ coming from representations theory of $U(k)$, the Lie group of unitary matrices of rank $k$. Sending $U\in U(k)$ to the sequence of arguments of its eigenvalues taken in $[0,2\pi[)$ and sorted in decreasing order yields a map $p:U(k)\mapsto  T_{k}$. Characters of $U(k)$ are conjugation invariant functions $\{\chi_I, I\in B_{k}\}$ with the formula 
\begin{equation}\label{eq:character_Uk}
\chi_I(U)=S_{I}(p(U))=\frac{a_{I}(u_1,\ldots,u_k)}{a_{I_0}(u_1,\ldots,u_k)},
\end{equation}
where $\{u_1,\ldots,u_k\}$ are arguments of eigenvalues of $U\in U(k)$ and we recall that $a_{I}(u_1,\ldots,u_k)=\det(e^{iu_lI_m})_{1\leq l,m\leq k}$ for $u\in\mathbb{R}^k$ and $I\in B_{k}$, where we recall that $B_k$ is the set of increasing sequences of integers.

 By Weyl integration formula, the image of the Haar measure $\mu_H$ on $U(k)$ through the map $p$ sending a unitary matrix to the decreasing sequence of arguments of its eigenvalues yields a probability measure on $ T_{k}$ which is continuous with respect to the Lebesgue measure and has density
$$dp_{\star}\mu_H(u_1,\ldots,u_n)=\frac{1}{(2\pi)^k}\prod_{l<m}\left\vert e^{iu_l}-e^{iu_m}\right\vert^2=\frac{1}{(2\pi)^k}\vert V(\vec{u})\vert^2.$$
Reciprocally, any probability measure on $ T_{k}$ yields a unique probability measure $p^{\star}\mu$ on $U(k)$ which is invariant by conjugation. Denote by $L^2_{inv}(U(k))$ the vector space of square integrable functions on $U(k)$ which are invariant by conjugation. Each element $f$ of $L^2_{inv}(U(k))$ yields a function $\tilde{f}$ on $L^2( T_{k},dp_{\star}\mu_H)$ by restricting its values to diagonal matrices, and the map $f\mapsto \tilde{f}$ is an isomorphism by Weyl integration formula. By Plancherel theorem, the map
$$\Phi:\left\lbrace\begin{aligned} L^2_{inv}(U(k),d\mu_H)&\simeq &L^2( T_{k},dp_{\star}(\mu_H)) &\rightarrow &&\ell^2(B_{k})\\
&f&&\mapsto&&\left(\int_{ T_{k}}S_{J}(u)\tilde{f}(u)dp_{\star}\mu_H(u)\right)_{J\in B_{k}}\end{aligned}\right.$$
\nomenclature{$\Phi$}{Continuous Fourier transform on $L^2(T_k,dp_{\ast}(\mu_H))$}is an $L^2$-isometry. For $c=(c(J))\in \ell^2(B_{k})$, the inverse of $\Phi$ at $c$ is given by 
$$\Phi^{-1}[c](\vec{u})=\sum_{J\in B_{k}}d_{J}c(J)S_J(\vec{u})$$
for $\vec{u}\in  T_{k}$, where the sum almost-surely converges and $d_J=S_J(0,\ldots,0)$ is the dimension of the irreducible representation of $U(k)$ having character $\chi_J=S_J\circ p$.

The map $\Phi$ extends by continuity (in the weak-topology) to a unique map, also denoted $\Phi$ and called the {\it Fourier transform} on $T_k$, from the set $\mathcal{M}_{f}( T_{k})$ of finite measures on $ T_{k}$ to the topological space of maps from $B_{k}$ to $\mathbb{C}$ with the pointwise convergence. We then have for $m\in \mathcal{M}_{f}( T_{k})$ and $J\in B_{k}$
\begin{equation}\label{eq:fourier_transform_measure_Uk}
\Phi[m](J)=\frac{1}{(2\pi)^k}\int_{ T_{k}}S_{J}(\vec{u})\vert V(\vec{u})\vert^2dm(\vec{u}).
\end{equation}
As for finite measures on $\mathbb{R}$ or $\mathbb{R}/(2\pi\mathbb{Z})$, the Fourier transform characterizes finite measures and vague convergence on $T_k$. Denote by $\mathcal{M}_{1}( T_{k})$ the subset of $\mathcal{M}_{f}( T_{k})$ consisting in measures $m$ on $T_k$ such that $\Phi[m](I_0)=1$. 

If $\nu_1,\nu_2\in \mathcal{M}_{f}(U(k))$, the set of finite measures on $U(k)$, the convolution of $\nu_1$ with $\nu_2$ is the unique finite measure $\nu_1\ast\nu_2$ such that $\int_{U(k)}fd\nu_1\ast\nu_2=\int_{U(k)\times U(k)}f(g_1g_2)d\nu_1(g_1)d\nu_2(g_2)$ for all $f\in C^0(U(k))$. If $\mu_1,\mu_2\in \mathcal{M}_f( T_{k})$, we define
$$\nu_1\ast\nu_2=p_{\star}(p^{\star}(\mu_1)\ast p^{\star}(\mu_2)).$$
Then, the convolution product $\ast$ is commutative on $\mathcal{M}_f( T_{k})$, and 
$$\Phi[\mu_1\ast\mu_2]=\Phi[\mu_1]\cdot \Phi[\mu_2],$$
where the product on the right-hand side is the pointwise product on $M(B_{k})$.

Recall that $ B_{k,n}$ embeds in $ T_{k}$ through the map $\xi_{n}$ defined in Section \ref{subsec:intro_model}. By linearity, this map is extended to a map (also denoted by $\xi_{n}$) from $L^2(B_{k,n})$ to $\mathcal{M}_f( T_{k})$ with the formula 
$$\xi_{n}\left(\sum_{I\in B_{k,n}}\mu(I)e_I\right)=\left(\frac{2\pi}{n}\right)^k\sum_{I\in B_{k,n}}\mu(I)\delta_{\xi_n(I)}, \quad \mu\in L^2(B_{k,n}).$$
A crucial fact for the sequel is that $\xi_{n}$ intertwines the Fourier transforms $\Phi_n$ and $\Phi$.
\begin{proposition}\label{prop:equiv_dft_cft}
For all $\mu\in L^2(B_{k,n})$ and $J\in B_{k}$ such that $J_1< J_k+n$. Then, with $J^{(n)}\in B_{k,n}$ such that $\{J^{(n)}_i[n]\}_{1\leq i\leq k}=\{J_i[n]\}_{1\leq i\leq k}$,
$$\Phi[\xi_{n}(\mu)] (J)=S_{J}(\xi_{n}(I_0))\Phi_n[\mu](J^{(n)}).$$
\end{proposition}
\begin{proof}
By linearity, we only have to prove the relation on $\mu=\delta_I$ for $I\in B_{k,n}$, where $\delta_I$ is defined in \eqref{eq:definition_dirac}. First, \eqref{eq:Fourier_probability} and Lemma \ref{lem:formula_inversion_Schur} yield
$$\Phi_n[\delta_I](J^{(n)})=\frac{S_{I}(\xi_{n}(J^{(n)}))}{S_I(\xi_{n}(I_0))}=\frac{S_{J^{(n)}}(\xi_{n}(I))}{S_{J^{(n)}}(\xi_{n}(I_0))}.$$
On the other hand,  since $\xi_n(\delta_I)=\left(\frac{2\pi}{n}\right)^k\frac{n^k}{\vert V(\xi_n(I))\vert^2}\delta{\xi_n(I)}=\frac{(2\pi)^k}{\vert V(\xi_n(I))\vert^2}\delta{\xi_n(I)}$,
\begin{align*}
\Phi[\xi_{n}(\delta_I)](J)=&\frac{1}{(2\pi)^k}\int_{ T_{k}}S_{J}(\vec{u})\vert V(\vec{u})\vert^2\frac{(2\pi)^kd\delta{\xi_n(I)}}{\vert V(\xi_n(I))\vert^2}(\vec{u})\\
=&S_{J}(\xi_{n}(I))=S_{J^{(n)}}(\xi_{n}(I)),
\end{align*}
The last equality is a consequence to the fact that $J^{(n)}$ has the form $(J_r+(\ell+1)n,\ldots,J_k+(\ell+1)n, J_1+\ell n,\ldots, J_{r-1}+\ell n)$ for some $1\leq r\leq k$ and $\ell\in\mathbb{Z}$ and because of the determinantal formula of $S_J$ and $S_{J'}$. The result is then deduced.
\end{proof}
\subsection{Invariant Brownian motion on $U_k$ and unitary Dyson Brownian motion}\label{Sec:Brownian_motion}
Let us conclude this section by recalling rigorously the definition of the \textit{generalized} Dyson Brownian motion on $U_k$ and its eigenvalue process on $ T_{k}$, see \cite{L} for a detailed background on this stochastic process and \cite{Far} for most of the content of this section. The Lie group $U(k)$ is a Riemannian sub-manifold of $M_{k}(\mathbb{C})$ of real dimension $n^2$, and the tangent space of $U(k)$ at $Id$ is given by the Lie algebra 
$$\mathfrak{u}_k=\{M\in M_k(\mathbb{C}),\, M^*+M=0\}.$$
Remark that $\mathfrak{u}_k$ is not semisimple, since $\mathfrak{u}_k=\mathfrak{su}_k\oplus i\mathbb{R}Id$, where $\mathfrak{su}_k=\{M\in\mathfrak{u}_k,\, Tr(M)=0\}$ is simple. There is, up to a scalar, a unique invariant scalar product on $\mathfrak{su}_k$ given by $\langle \cdot,\cdot\rangle:(X,Y)\mapsto kTr(Y^*X)$. There is trivially a unique invariant metric on $i\mathbb{R}Id$ given by $(X,Y)\mapsto Tr(X) Tr(Y^*)$. Hence, we can deduce that for all $\alpha,\beta> 0$, there is a invariant scalar product $\langle\cdot,\cdot\rangle_{\alpha,\gamma}$ on $\mathfrak{u}_k$ given by
$$\langle X,Y\rangle_{\alpha,\gamma}=\gamma^{-1} k Tr(Y^*X)+(\alpha^{-1}-\gamma^{-1}) Tr(Y^*)Tr(X),$$
and that any invariant scalar product is of this type.

For $1\leq r,s\leq n$, write $E_{rs}$ for the matrix with only one nonzero entry on the $r$-th line and $s$-th column. Set $X_{rs}=\frac{1}{\sqrt{2 k}}
(E_{rs}-E_{sr}), X_{sr}=\frac{i}{\sqrt{2 k}}(E_{rs}+E_{sr})$ for $1\leq r<s\leq k$, $X_{rr}=\frac{i}{\sqrt{2 r(r-1) k}}\sum_{1\leq s< r}(E_{ss}-E_{rr})$ for $2\leq r\leq k$ and $X_{11}=\frac{1}{k}\sum_{r=1}^kE_{rr}$. Then, $\{X_{rs}\}_{1\leq r,s\leq k}$ is an orthonormal basis of $\mathfrak{u}_k$ with respect to the scalar product $\langle \cdot,\cdot\rangle_{1,1}$. For $Y\in \mathfrak{u}_k$ we denote by $\mathcal{L}_Y:C^1(U_k)\mapsto C^0(U_k)$ the differential operator induced by the unique left-invariant vector field equal to $Y$ at $Id$, namely for $f\in C^1(U_k)$ and $U\in U_k$,
$$\mathcal{L}_{Y}[f](U)=\lim_{t\rightarrow 0}\frac{f(U\exp(tY))-f(U)}{t}.$$
We can then define the Brownian motion with parameters $(\alpha,\gamma)$ on $U(k)$ as follows.
\begin{definition}\label{def:brownian_un}
Let $\alpha,\gamma\geq  0$. There exists a unique Markov process $\mathbf{B}^{\alpha,\gamma}(t))_{t\geq 0}$ with almost sure continuous trajectories on $U(k)$ such that 
$$\frac{\partial}{\partial t}\mathbb{E}(f(U\mathbf{B}^{\alpha,\gamma}(t)))_{\vert t=0}=\left(\gamma\sum_{1\leq i,j\leq n}\mathcal{L}_{X_{ij}}^2+(\alpha-\gamma)\mathcal{L}_{X_{11}}^2\right)f(U),$$
for all $f\in C^1(U_n)$ and $U\in U_n$. 
Moreover, such a Markov process is invariant by left and right translation by elements of $U(k)$, and any Markov process on $U(k)$ with the same invariance property and with almost sure continuous trajectories is equal to $\mathbf{B}^{\alpha,\gamma}$ for some $\alpha,\gamma\geq 0$.
\end{definition}
For $\alpha,\gamma>0$, the stochastic process $\mathbf{B}^{\alpha,\gamma}$ is the Brownian motion with respect to the left invariant metric whose value at $Id$ is $\langle\cdot,\cdot\rangle_{\alpha,\gamma}$. From the expression of the generator given in Definition \ref{def:brownian_un}, it can be seen that $\mathbf{B}^{0,\gamma}(t)\mathbf{B}^{0,\gamma}(t)^{-1}\in SU(k)$  (resp. $\mathbf{B}^{\alpha,0}(t)\mathbf{B}^{\alpha,0}(t)^{-1}\in\mathbb{T}\cdot Id$) a.s for all $t\geq 0$. 

For $\vec{u}\in  T_{k}$ with $u_k<\dots<u_1$, let $B^{\alpha,\gamma}_{\vec{u}}$ be the eigenvalues process of $U_{\vec{u}}\mathbf{B}^{\alpha,\gamma}$ where $U_{\vec{u}}$ is uniformly sampled on the adjoint orbit $\{U\in U(k), p(U)=\vec{u}\}$. Then, $B^{\alpha,\gamma}_{\vec{u}}$ is a Markov process on $ T_{k}$ starting at $\vec{u}$ and whose generator is given by 
\begin{equation}\label{eq:generator_Brownian_Motion}
\mathcal{L}_{\alpha,\gamma}=\frac{\gamma}{k}\left[\sum_{i=1}^k\frac{\partial^2}{\partial u_i^2}+\sum_{i<j}\cot \frac{u_j-u_i}{2}\left(\frac{\partial}{\partial u_j}-\frac{\partial}{\partial u_i}\right)\right]+\frac{\alpha-\gamma}{k^2}\frac{\partial^2}{\partial^2(\sum u_i)},
\end{equation}
the proof of this fact is given in \cite[Proposition 12.5.1]{Far} for the case $\alpha=\gamma=k$ and can be generalized for arbitrary $\alpha,\gamma\geq 0$. If $u_i=u_{i+1}$ for some $1\leq i\leq k$, one can still give a meaning to the corresponding heat equation in a weak sense, and in any case the corresponding kernel is $C^{\infty}$ for all $t>0$.  We call then the resulting stochastic process $(B^{\alpha,\gamma}_{\vec{u}}(t))_{t\geq 0}$ the generalized unitary Dyson Brownian motion with parameter $(\alpha,\gamma)$ starting at $\vec{u}$ (note that the subscript $\vec{u}$ will be dropped when $\vec{u}=(0,\ldots,0)$). 

When $\alpha,\gamma>0$ and for any initial value $\vec{u}\in T_{k}$, the probability measure $\mu_t$ given by the unitary Dyson Brownian motion with parameter $(\alpha,\gamma)$ starting at $\vec{u}$ has a density $K^{U(k),\alpha,\gamma}_{t}(\vec{u},\vec{v})\vert V(\vec{v})\vert^2$,  where $K^{U(k),\alpha,\gamma}_{t}(\vec{u},\vec{v})$ satisfies the parabolic differential equation 
$$\left\lbrace\begin{aligned}
\partial_tK^{U(k),\alpha,\gamma}_{t}(\vec{u},\cdot)&=\mathcal{L}_{\alpha,\gamma}K^{U(k),\alpha,\gamma}_{t}(\vec{u},\cdot),\\
K^{U(k),\alpha,\gamma}_{0}(\vec{u},\cdot)&=\frac{1}{\vert V(\vec{u})\vert^2}\delta_{\vec{u}},
\end{aligned}\right.$$
where the second equality is understood in the weak sense. For general initial distribution $\mu_0\in \mathcal{M}_1(T_k)$, the time-evolving distribution $(\mu_t)_{t\geq 0}$ can also be characterized by its Fourier transform: for any $J\in B_{k}$,
\begin{equation}\label{eq:Fourier_Dyson_Brownian_motion}
\Phi[\mu_t](J)=\exp(-\kappa_{\alpha,\gamma}(J)t)\Phi[\mu](J),
\end{equation}
where 
$$\kappa_{\alpha,\gamma}(J)=\frac{\gamma}{k}\left[\Vert \tilde{J}\Vert^2-\Vert \widetilde{I_0}\Vert^2\right]+\frac{\alpha}{k^2}(\langle J\rangle-\langle I_0\rangle)^2,$$
with $\langle J\rangle=\sum_{i=1}^k J_i$ and $\Vert \tilde{J}\Vert^2=\sum_{i=1}^k(J_i-\langle J\rangle/k)^2$ for $J\in B_{k}$. By \eqref{eq:Fourier_Dyson_Brownian_motion} and the inverse Fourier transform on the unitary group, for $\alpha,\gamma>0$ and $t>0$ we have
\begin{equation}\label{eq:series_representation_Brownian}
K^{U(k),\alpha,\gamma}_{t}(\vec{u},\vec{v})=\sum_{J\in B_{k}}\exp(-\kappa_{\alpha,\gamma}(J)t)\overline{S_{J}(\vec{u})}S_{J}(\vec{v}).
\end{equation}
In the case where $\alpha=0$, the stochastic process remains on the affine subspace $\{\vec{v}\in T_k,\sum_{i=1}
^kv_i=\sum_{i=1}^ku_i\}$ and thus has no density on $ T_{k}$ even for $t>0$. Set $\mathbb{A}=\{\vec{u}\in T_k,\sum_{i=0}^ku_i=0\}$, which corresponds to the set $\{p(U), U\in SU(k)\}$. Then, $B^{0,1}$ has a density $K^{SU(k),\gamma}_{t}(\vec{u},\vec{v})$ on $\mathbb{A}$ given by
\begin{equation}\label{eq:series_representation_Brownian_suk}
K^{SU(k),\gamma}_{t}(\vec{u},\vec{v})=\sum_{\substack{J\in B_{k}\\J_k=0}}\exp(-\kappa_{0,\gamma}(J)t)\overline{S_{J}(\vec{u})}S_{J}(\vec{v}).
\end{equation}
By extension, when $\sum_{i=1}^ku_i=\sum_{i=1}^kv_i$, we also set $K^{SU(k),\gamma}_{t}(\vec{u},\vec{v})=K^{SU(k),\gamma}_{t}(\ell\cdot\vec{u},\ell\cdot\vec{v})$ with $\ell=\sum_{i=1}^ku_i$.

By the diagonalization of $\mathcal{L}_{\alpha,\gamma}$ on the orthonormal basis $\{S_{J}, J\in B_k\}$ of the Hilbert space $L^2( T_{k},dp_{\star}(\mu_H))$, the kernel of any operator $\mathcal{L}_{\alpha,\gamma}$ with $\gamma,\alpha>0$ consists of the vector space $\mathbb{C}\mathbf{1}_{T_k}$, so that the unique invariant measure of the corresponding Brownian motion is
$$\vert V(u_1,\ldots,u_k)\vert^2=\prod_{i<j}\vert e^{2i\pi u_j}-e^{2i\pi u_i}\vert^2,$$
which is the density of the measure $p_{\star}(\mu_H)$ with respect to the Lebesgue measure on $ T_{k}$. Hence, for $\alpha,\gamma>0$, the generalized Dyson unitary Brownian motion has a unique invariant probability distribution on $ T_{k}$ which is precisely $p_{\star}(\mu_H)$.  When $\alpha=0$ (resp. $\gamma=0$), the space of invariant measures is larger than $p_{\star}(\mu_H)$.

 \section{The Fourier transform of a large sequence of convolution}\label{sec:Fourier_large_product}
The goal of this section is to estimate the discrete Fourier transform of the convolution of a sequence $(\mu_i)_{1\leq i\leq m}$ of $h$-probability measures in $B_{k,n}$ when $m$ is large compared to $n$. Mimicking the previous definition for elements of $B_{k,n}$, for $x\in \mathbb{R}^k$ we set $\langle x\rangle=\sum_{j=1}^kx_j$ and then define
$$\widetilde{x}=x-\frac{\langle x\rangle}{k}\mathbf{1}_k=\left(x_1-\frac{\langle x\rangle}{k},\ldots,x_k-\frac{\langle x\rangle}{k}\right).$$ 
We then introduce for $x\in \mathbb{R}^k$ the shifted norm 
$$K(x)=\Vert x\Vert^2-\Vert \tilde{I_0}\Vert^2,$$
where $\Vert\cdot\Vert$ denotes the usual Euclidean norm on $\mathbb{R}^k$. We will regularly use for $x\in\mathbb{R}^k$ the identity
\begin{equation}\label{eq:rho(tilde_I)}K\left(\tilde{x}\right)=K(x)-\frac{\langle x\rangle^2}{k},
\end{equation}
whose proof is left to the reader. Remark that we have in particular $K(\tilde{I_0})=0$ and that $I\in B_{k,n}$ satisfies $K(\tilde{I})=0$ if and only if $I=I_0+\ell\mathbf{1}_k$ for some $0\leq \ell\leq n-k$. Finally, we introduce for $I\in B_{k,n}$ the shifted configuration 
$$\hat{I}=I-I_k\mathbf{1}_k=(I_1-I_k,\ldots,I_{k-1}-I_k,0),$$
so that $\hat{I}\in B_{k,n}$.
We introduced in Section \ref{sec:intro_model} for $\mu\in \mathcal{M}_1(B_{k,n})$, the random variable $I_{\mu}\in B_{k,n}$ whose distribution is 
$$\mathbb{P}(I_{\mu}=I)=\mu(I)\mu^h(I)=\frac{\vert V(\xi_n(I))\vert^2}{n^k}\mu(I).$$
We then set 
\begin{equation}\label{eq:definition_moments}
\Vert\mu\Vert_r=\mathbb{E}\left[\Vert\tilde{I}_\mu\Vert^r\right],\quad \langle \mu\rangle=\mathbb{E}\left[\langle I_{\mu}\rangle\right],\,\langle\hat{\mu}-I_0\rangle_r=\mathbb{E}\left[\langle \widehat{I_{\mu}}-I_0\rangle^r\right]\text{ and } Var_{r}(\mu)=\mathbb{E}\left[\Big\vert\langle I_{\mu}\rangle-\langle \mu\rangle\Big\vert^r\right]
\end{equation}
for $r\geq 2$. We simply write $Var(\mu)$ for $Var_2(\mu)$ in the sequel, and we finally set 
$$K(\mu):=\Vert \mu\Vert_2-\Vert \tilde{I}_0\Vert^2=\mathbb{E}\left[K(\tilde{I}_{\mu})\right].$$
\nomenclature{$K(\mu)$}{$\mathbb{E}\left[\Vert\tilde{I}_\mu\Vert^2\right]-\Vert \tilde{I}_0\Vert^2$}
The first main estimate of the Fourier transform is the following lemma, which shows pointwise asymptotics of the Fourier transform as $n$ goes to infinity. For sake of completeness, we give asymptotic results which hold uniformly in $k$. In the following statement, recall that for $I\in B_{k}$ with $I_1<I_k+n$, we set $\lambda_I=I-I_0$ and $I^{(n)}$ is the unique element of $B_{k,n}$ such that $\{I^{(n)}_i[n]\}_{1\leq i\leq k}=\{I_i[n]\}_{1\leq i\leq k}$.

\begin{lemma}\label{lem:asymptoti_eigenvalue}
Let $n,k\geq 2$. For any $h$-probability measure $\mu\in B_{k,n}$ and $I\in B_{k}$ with $n\geq \Vert I\Vert{\infty}$,
\begin{align*}
e^{-2i\pi \frac{(\langle\mu\rangle-\langle I_0\rangle)\cdot\langle \lambda_I\rangle}{kn}}\Phi_{n}[\mu](I^{(n)})&=1-\frac{(2\pi)^2}{2n^2}\left[\frac{K(\mu)K(\tilde{I})}{k^2-1}+\frac{Var(\mu)\langle \lambda_I\rangle^2)}{k^2}\right]\\
&\hspace{3cm}+\frac{1}{n^3}O\left(\frac{Var_3(\mu)\cdot\vert\langle \lambda_I\rangle\vert^3}{k^3}+\langle\hat{\mu}-I_0\rangle_3\cdot\Vert \tilde{I}\Vert_\infty^3\right),
\end{align*}
and 
\begin{align*}
\left\vert \Phi_{n}[\mu](I^{(n)})\right\vert&\leq 1-\frac{(2\pi)^2}{2n^2}\frac{K(\mu)K(\tilde{I})}{k^2-1}+\frac{1}{n^3}O\left(\langle\hat{\mu}-I_0\rangle_3\cdot\Vert \tilde{I}\Vert_\infty^3\right),
\end{align*}
where $O(\cdot)$ only depends on numeric constants.
\end{lemma}
The main ingredient in the proof of the latter lemma is an asymptotic formula for ratios of Schur functions. Recall that $\Lambda_{k}$ denotes the set of partitions with at most $k$ parts and for $\vec{u}\in\mathbb{R}^k$ we abbreviate $\exp(i\vec{u})$ for the vector $(\exp(iu_1),\ldots,\exp(iu_k))$.
\begin{lemma}\label{lem:asymptotic_schur}
Let $k\geq 2$. For any $\lambda\in \Lambda_{k}$ with $\lambda_k=0$ and $\vec{u}\in \mathbb{R}^k$ such that $\sum_{i=1}^ku_i=0$ and $\langle \lambda\rangle \max(\Vert \vec{u}\Vert_\infty,k/3)\leq n$,
\begin{align*}
\frac{s_{\lambda}\left(\exp\Big(2\pi i\vec{u}/n\Big)\right)}{s_{\lambda}\left(\exp\left(2\pi i \widetilde{I_0}/n\right)\right)}=1-\frac{(2\pi)^2}{2(k^2-1)n^2}\Big[K(\widetilde{I_{\lambda}})K(\vec{u})\Big]+O\left(\left(\frac{\langle \lambda\rangle\cdot (\Vert \vec{u}\Vert_\infty+k)}{n}\right)^3\right),
\end{align*}
with $O(\cdot)$ only depending on numeric constants.
\end{lemma}
The proof of this lemma uses shifted Schur functions, as introduced in \cite{OkOl}. Shifted Schur functions form a particular basis $\{s_{\lambda}^*\}_{\lambda\in R}$ of the ring of polynomials functions in $\{x_i\}_{i\geq 1}$ which are symmetric in the variables $\{x_i-i\}_{i\geq 1}$. Their main useful property for us is the following generalized binomial formula for all $k\geq 1$ and $\mu\in \Lambda_{k}$, see \cite[Sec. 5]{OkOl},
\begin{equation}\label{eq:binomial_formula}
\frac{s_{\mu}(1+x_1,\ldots,1+x_k)}{s_{\mu}(1,\ldots,1)}=\sum_{\nu \subset \mu}\frac{1}{(k\upharpoonleft \nu)}s_{\nu}^*(\mu_1,\ldots,\mu_k)s_{\nu}(x_1,\ldots,x_k),
\end{equation}
where $s_{\nu}, \nu\in \Lambda_{k}$ are the usual Schur functions and $(k\upharpoonleft \nu)=\prod_{i=1}^k\frac{(\nu_i+k-i)!}{(k-i)!}=\prod_{i=1}^{l(\nu)}\prod_{j=1}^{\nu_i}(k-i+j)$. We will be mostly interested in the first shifted Schur functions, namely for $\vert\nu\vert\leq 2$. We then have 
$$s^*_{(1)}(x_{1},x_2,\ldots)=\sum_{i\geq 1}x_i,\, s^*_{(2)}(x_{1},x_2,\ldots)=\sum_{i\leq j}x_j(x_i-1),\, s^*_{(11)}(x_{1},x_2,\ldots)=\sum_{i<j}x_j(x_i+1).$$

\begin{proof}[Proof of Lemma \ref{lem:asymptotic_schur}]
We only prove this lemma for $\lambda\in \Lambda_{k}$ such that $\langle \lambda\rangle\geq 3$, the case $\langle \lambda\rangle\leq 2$ being done similarly. Let $k\geq 1$ and $\lambda\in \Lambda_{k}$ with $\langle \lambda\rangle\geq 3$. By \eqref{eq:binomial_formula}, for $(x_1,\ldots,x_k)\in \mathbb{C}^k$ we have 
\begin{align}
\frac{s_{\lambda}(1+x_1,\ldots,1+x_k)}{s_{\lambda}(1,\ldots,1)}=&\sum_{\nu \subset \lambda}\frac{1}{(k\upharpoonleft \nu)}s_{\nu}^*(\lambda)s_{\nu}(\vec{x})\nonumber\\
=&1+\frac{1}{(k\upharpoonleft (1))}s_{(1)}^*(\lambda)s_{(1)}(\vec{x})+\frac{1}{(k\upharpoonleft (2))}s_{(2)}^*(\lambda)s_{(2)}(\vec{x})\nonumber\\
+&\frac{1}{(k\upharpoonleft (11))}s_{(11)}^*(\lambda)s_{(11)}(\vec{x})+C_\lambda(\vec{x}) \label{eq:ratio_schur_first expansion},
\end{align}
with $C_\lambda(\vec{x)}=\sum_{\substack{\nu \subset \lambda\\\langle \nu\rangle \geq 3}}\frac{1}{(k\upharpoonleft \nu)}s_{\nu}^*(\lambda)s_{\nu}(\vec{x})$. Remark first that for $\vec{x}\in \mathbb{C}^k$, $\vert s_{\nu}(\vec{x})\vert \leq \Vert \vec{x}\Vert_{\infty}^{\langle \nu\rangle }s_{\nu}(1\ldots,1)$, so that, using that $ (k\upharpoonleft \nu)=s_{\nu}(\mathbf{1}_{k})H(\nu)$ with $\mathbf{1}_k$ the constant vector of length $k$ equal to $1$ and $H(\mu)$ the Hook length of $\mu$, see \cite[Eq. 5.5]{OkOl},
\begin{align}
\vert C_\lambda(\vec{x)}\vert \leq \sum_{\substack{\nu \subset \lambda\\\langle \nu\rangle \geq 3}}\frac{1}{(k\upharpoonleft \nu)}s_{\nu}^*(\lambda)s_{\nu}(\vec{x})
\leq &\sum_{\substack{\nu \subset \lambda\\\langle \nu\rangle \geq 3}}\frac{s_{\nu}(1\ldots,1)}{(k\upharpoonleft \nu)}s_{\nu}^*(\lambda)\Vert \vec{x}\Vert_{\infty}^{\langle \nu\rangle }\nonumber\\
\leq&\sum_{\substack{\nu \subset \lambda\\\langle \nu\rangle \geq 3}}\frac{s_{\nu}^*(\lambda)}{H(\nu)}\Vert \vec{x}\Vert_{\infty}^{\langle \nu\rangle }.\nonumber
\end{align}
Using that $\frac{(\vert \lambda \vert -\langle \nu\rangle )!}{\langle \lambda\rangle !}s^*_\nu(\lambda)=\frac{\dim V(\lambda\setminus \nu)}{\dim V(\lambda)}$, where $V(\lambda)$ denotes the Specht irreducible representation of the symmetric group $S_{\langle \lambda\rangle}$ associated to $\lambda$ (see \cite[Eq. 0.14]{OkOl}), and the classical Hook length formula $\dim V(\lambda)=\frac{\langle \lambda\rangle !}{H(\lambda)}$, we get  
\begin{align*}
\vert C_\lambda(\vec{x)}\vert \leq& \sum_{\substack{\nu \subset \lambda\\\langle \nu\rangle \geq 3}}H(\nu)^{-1}\frac{\langle \lambda\rangle !\dim V(\lambda\setminus\nu)}{(\langle \lambda\rangle -\langle \nu\rangle )!\dim V(\lambda)}\Vert \vec{x}\Vert_{\infty}^{\langle \nu\rangle }\\
\leq&\sum_{\substack{\nu \subset \lambda\\\langle \nu\rangle \geq 3}}\binom{\langle \lambda\rangle}{\langle \nu\rangle}\frac{\dim V(\nu)\dim V(\lambda\setminus\nu)}{\dim V(\lambda)}\Vert \vec{x}\Vert_{\infty}^{\langle \nu\rangle }\\
\leq &\sum_{l=3}^{\langle \lambda\rangle}\binom{\langle \lambda\rangle}{l}\Vert \vec{x}\Vert_{\infty}^{l}\left(\sum_{\substack{\nu\subset \lambda\\\langle \nu\rangle=l}}\frac{\dim V(\nu)\dim V(\lambda\setminus\nu)}{\dim V(\lambda)}\right).
\end{align*}
Remark that $\dim V(\lambda)=\sum_{\substack{\nu\subset \lambda\\\langle \nu\rangle=l}}\dim V(\nu)\dim V(\lambda\setminus\nu)$ for all $1\leq l\leq \langle \lambda\rangle$ by restricting the representation $V(\lambda)$ of $S_{\langle \lambda\rangle}$ to $S_{l}\times S_{\langle \lambda\rangle-l}$, so that $\sum_{\substack{\nu\subset \lambda\\\langle \nu\rangle=l}}\frac{\dim V(\nu)\dim V(\lambda\setminus\nu)}{\dim V(\lambda)}=1$ and finally 
\begin{align}
\vert C_\lambda(\vec{x)}\vert \leq&\sum_{l=3}^{\langle \lambda\rangle}\binom{\langle \lambda\rangle}{l}\Vert \vec{x}\Vert_{\infty}^{l}\nonumber\\
\leq& \Vert \vec{x}\Vert_{\infty}^{3}\frac{\langle \lambda\rangle !}{3!(\langle \lambda\rangle-3)!}(1+\Vert \vec{x}\Vert_{\infty})^{\langle \lambda\rangle-3}\leq \frac{\langle \lambda\rangle^3}{6}\Vert \vec{x}\Vert_{\infty}^{3}(1+\Vert \vec{x}\Vert_{\infty})^{\langle \lambda\rangle-3}\nonumber
\end{align}
as long as $\Vert x\Vert_{\infty}<1$, with $C$ only depending on $\lambda$. Remark that for $\Vert x\Vert_\infty\leq 1/\langle \lambda\rangle$, the latter bound simplifies as
\begin{equation}\label{eq:computation_ratio_first_part}
\vert C_\lambda(\vec{x)}\vert \leq e\cdot \frac{\langle \lambda\rangle^3}{6}\Vert \vec{x}\Vert_{\infty}^{3}.
\end{equation}

Let us then focus on the first terms of the right-hand side of \eqref{eq:ratio_schur_first expansion}. First, we have $\frac{1}{(k\upharpoonleft (1))}=\frac{1}{k}$ and $s_{(1)}^*=s_1$, so that 
$$\frac{1}{(k\upharpoonleft (1))}s_{(1)}^*(\lambda)s_{(1)}(\vec{x})=\frac{1}{k}s_{(1)}(\lambda)s_{(1)}(\vec{x}).$$
Set $x_i=\exp(iu_i/n)-1$, for $1\leq i\leq k$. We then have $\vert x_i\vert\leq (1\wedge \vert u_i\vert/n)$ and there is a numerical constant $C'>0$ such that $\left\vert x_i-iu_i/n+\frac{u_i^2}{2n^2}\right\vert \leq \frac{C'\vert u_i\vert^3}{n^3}$. Using that $\sum_{i=1}^ku_i=0$, this yields
$$\left\vert s_{(1)}(\vec{x})+\frac{1}{2n^2}\sum_{i=1}^ku_i^2\right\vert\leq \frac{kC'\Vert \vec{u}\Vert_{\infty}^3}{n^3}.$$
Hence, introducing the second power sum $p_2(\vec{u})=\sum_{i=1}^ku_i^2$,
\begin{equation}\label{eq:computation_ratio_second_part}\frac{1}{(k\upharpoonleft (1))}s_{(1)}^*(\lambda)s_{(1)}(\vec{x})=\frac{-s_{(1)}(\lambda)p_2(\vec{u})}{2kn^2}+K,
\end{equation}
with $\vert K\vert\leq C\langle \lambda\rangle \frac{\Vert \vec{u}\Vert_{\infty}^3}{n^3}$ with $C$ numeric.

Next, $\frac{1}{(k\upharpoonleft (11))}=\frac{1}{k(k-1)}$ and $\frac{1}{(k\upharpoonleft (2))}=\frac{1}{k(k+1)}$. Introducing the shifted power sum $p_2^*=\sum_{i=1}^k(x_i-i)^2-(-i)^2=\sum_{i=1}^kx_i(x_i-2i)$, we then have the shifted Newton identities 
$$s_2^*=\frac{1}{2}(s_1^2+p_2^*),\, s_{11}^*=\frac{1}{2}(s_{(1)}^2-p_2^*)-s_{(1)}.$$
Hence, using additionally the usual Newton identities yields 
\begin{align*}
&\frac{1}{(k\upharpoonleft (2))}s_{(2)}^*(\lambda)s_{(2)}(\vec{x})+\frac{1}{(k\upharpoonleft (11))}s_{(11)}^*(\lambda)s_{(11)}(\vec{x})\\
=&\frac{1}{4k(k+1)}\left[s_{(1)}^2(\lambda)+p_2^*(\lambda)\right]\cdot \left[s_{(1)}^2(\vec{x})+p_2(\vec{x})\right]\\
&\hspace{3cm}+\frac{1}{4k(k-1)}\left[s_{(1)}^2(\lambda)-p_2^*(\lambda)-2s_{(1)}(\lambda))\right]\cdot\left[s_{(1)}^2(\vec{x})-p_2(\vec{x})\right]\\
=&\frac{1}{2k(k^2-1)}\Big[s_{(1)}^2(\lambda)\left(ks_{(1)}^2(\vec{x})-p_2(\vec{x})\right)+p_2^*(\lambda)\left(kp_2(\vec{x})-s_{(1)}^2(\vec{x})\right)\\
&\hspace{7cm}-(k+1)s_{(1)}(\lambda)\left(s_{(1)}^2(\vec{x})-p_2(\vec{x})\right)\Big].
\end{align*}
There exists a numeric constant $C>0$ such that 
$$\left\vert x_i-iu_i/n\right\vert \leq \frac{C\vert u_i\vert^2}{n^2}.$$
Hence, since $s_{(1)}(\vec{u})=0$, we have $\left\vert s_{(1)}(\vec{x})\right\vert \leq \frac{Ck \Vert \vec{u}\Vert_\infty^2}{n^2}$,
which yields
\begin{align*}
\left\vert s_{(1)}(\vec{x})^2\right\vert\leq&  2Ck^2 \frac{\Vert \vec{u}\Vert_\infty^3}{n^3}
\end{align*}
for some numeric constant $C>0$ when $\Vert \vec{u}\Vert_{\infty}\leq n$. Likewise, $\left\vert x_i^2+u_i^2/n^2\right\vert\leq \frac{2C\vert u_i\vert^3}{n^3}$ for $1\leq i\leq k$, so that 
$$\left \vert p_2(\vec{x})+p_2(\vec{u})/n^2)\right\vert \leq \frac{2Ck\Vert \vec{u}\Vert_\infty^3}{n^3}.$$
Therefore,
\begin{align}
\frac{1}{(k\upharpoonleft (2))}s_{(2)}^*(\lambda)s_{(2)}(\vec{x})+&\frac{1}{(k\upharpoonleft (11))}s_{(11)}^*(\lambda)s_{(11)}(\vec{x})\nonumber\\
=&\frac{1}{2k(k^2-1)n^2}\left[s_{(1)}^2(\lambda)p_2(\vec{u})-kp_2^*(\lambda)p_2(\vec{u})-(k+1)s_{(1)}(\lambda)p_2(\vec{u})\right]+K'\label{eq:computation_ratio_third_part},
\end{align}
with 
\begin{align*}
K'\leq C\frac{\Vert \vec{u}\Vert_{\infty}^3}{n^3k(k^2-1)}\left[s_{(1)}^2(\lambda)(k^3+k)+k^2\vert p^*_2(\lambda)\vert+s_{(1)}(\lambda)(k^3+k^2)\right]\leq C'\frac{\Vert \vec{u}\Vert_{\infty}^3s_{(1)}^2(\lambda)}{n^3},
\end{align*}
for some numeric $C,C'>0$, where we used on the last inequality that $\vert p_2^*(\lambda)\vert /k\leq s_{(1)}(\lambda)^2$ for $\lambda\in \Lambda_{k}$. Remark that 
\begin{align*}
p_2^*(\lambda)+(k+1)s_{(1)}(\lambda)=\sum_{i=1}^k\lambda_i(\lambda_i+k-2i+1)=&\sum_{i=1}^k\left(\lambda_i+\frac{k-2i+1}{2}\right)^2-\left(\frac{k-2i+1}{2}\right)^2\\
=\Vert \lambda+\tilde{I}_0\Vert^2-\Vert \tilde{I}_0\Vert^2=K(\lambda+\tilde{I}_0),
\end{align*}
where $K$ is defined at the beginning of this section and we recall that $\tilde{I_0}=I_0-\frac{\langle I_0\rangle}{k}\mathbf{1}_k=\left(\frac{k+1-2i}{2}\right)_{1\leq i\leq k}$. Hence, putting \eqref{eq:computation_ratio_first_part}, \eqref{eq:computation_ratio_second_part} and \eqref{eq:computation_ratio_third_part} together yields then, for $\Vert \vec{u}\Vert_\infty/n\leq \frac{1}{\langle \lambda\rangle}$,
\begin{align}
\frac{s_{\lambda}(\exp(iu_1/n),\ldots,\exp(iu_k/n))}{s_{\lambda}(1,\ldots,1)}&=1-\frac{1}{2k(k^2-1)n^2}\Big[-s_{(1)}^2(\lambda)p_2(\vec{u})+kp_2(\vec{u})K(\lambda+\tilde{I_0})\Big]\nonumber\\
&\hspace{7cm}+O\left(\left(\frac{\langle \lambda\rangle\cdot \Vert \vec{u}\Vert_\infty}{n}\right)^3\right),\nonumber\\
&=1-\frac{1}{2(k^2-1)n^2}\left[p_2(\vec{u})\left(K(\lambda+\tilde{I_0})-\frac{s_{(1)}^2(\lambda)}{k}\right)\right]\nonumber\\
&\hspace{7cm}+O\left(\left(\frac{\langle \lambda\rangle\cdot \Vert \vec{u}\Vert_\infty}{n}\right)^3\right),
\end{align}
where the last term of the right-hand side of the first equality takes into account the correction coming from \eqref{eq:computation_ratio_second_part}. Since $s_{(1)}(\lambda)=\langle \lambda\rangle=\vert I_{\lambda}\vert-\langle I_0\rangle$ and $\lambda+\tilde{I}_0=I_{\lambda}+\frac{\langle I_0\rangle}{k}\mathbf{1}_k$, we have by \eqref{eq:rho(tilde_I)}
\begin{align*}
K(\lambda+\tilde{I_0})-\frac{s_{(1)}^2(\lambda)}{k}=K(\widetilde{\lambda+I_0})+\frac{\vert \lambda+\tilde{I_0}\vert}{k}-\frac{\langle \lambda\rangle^2}{k}=K(\widetilde{\lambda_I}).
\end{align*}
Hence, 
\begin{equation}\label{eq:character_final_asymptotic}
\frac{s_{\lambda}(\exp(iu_1/n),\ldots,\exp(iu_k/n))}{s_{\lambda}(1,\ldots,1)}=1-\frac{K(\widetilde{I_{\lambda}})\Vert \vec{u}\Vert^2}{2(k^2-1)n^2}+O\left(\left(\frac{\langle \lambda\rangle\cdot \Vert \vec{u}\Vert_\infty}{n}\right)^3\right),
\end{equation}
where we used that $p_2(\vec{u})=\Vert \vec{u}\Vert^2$. In the special case where $\vec{u}=2\pi \widetilde{I_0}=2\pi((k-1)/2,\ldots,-(k-1)/2)$, which satisfies $\sum_{i=1}^ku_i=0$, we have $\langle I_0\rangle_{\infty}\leq 2\pi k$, so that 
\begin{equation}\label{eq:rho_final_asymptotic}
\frac{s_{\lambda}(\exp(2i\pi \widetilde{I_0}/n))}{s_{\lambda}(1,\ldots,1)}=1-\frac{(2\pi)^2K(\widetilde{I_{\lambda}})\Vert \widetilde{I_0}\Vert^2}{2(k^2-1)n^2}+O\left(\left(\frac{\langle \lambda\rangle\cdot k}{n}\right)^3\right).
\end{equation}
Then, taking the ratio of \eqref{eq:character_final_asymptotic} with \eqref{eq:rho_final_asymptotic} yields, as long as $\langle \lambda\rangle\max(\Vert \vec{u}\Vert_\infty,k/3)\leq n$ (remark that the constant $1/3$ is arbitrary and chosen for later simplification), which implies $K(\widetilde{I_{\lambda}})\max(\Vert \vec{u}\Vert_\infty^2,k^2)\leq kn^2$, 
\begin{align*}
\frac{s_{\lambda}(\exp(2i\pi\vec{u}/n))}{s_{\lambda}(\exp(2i\pi \widetilde{I_0}/n))}&=1-\frac{(2\pi)^2}{2(k^2-1)n^2}\Big[K(\widetilde{I_{\lambda}})K(\vec{u})\Big]+O\left(\left(\frac{\langle \lambda\rangle\cdot (\Vert \vec{u}\Vert_\infty+k)}{n}\right)^3\right).
\end{align*}
\end{proof}

\begin{proof}[Proof of Lemma \ref{lem:asymptoti_eigenvalue}]
Let us fix $I\in B_{k}$ and first assume that $\mu=\delta_J$, $n\geq \langle J\rangle_{\infty}$. Then, for such $n$, by \eqref{eq:Fourier_probability} and the fact that $\xi_{n}(I)=\xi_{n}(I^{(n)})$, we have 
$$\Phi_{n}[\mu](I^{(n)})=\frac{S_{J}(\xi_{n}(I^{(n)}))}{S_{J}(\xi_{n}(I_0))}=\frac{S_{J}(\xi_{n}(I))}{S_{J}(\xi_{n}(I_0))}.$$
By definition of the map $\xi_n$ in \eqref{eq:definition_xi_n} and the definition of $S_J$ following \eqref{eq:def_vI},
$$S_{J}(\xi_{n}(I))=s_{\lambda_{J}}\left(\exp\left(2i\pi  \left(I-\frac{k-1}{2}\mathbf{1}_k\right)\right)\right)=s_{\lambda_{J}}\left(\exp\left(2i\pi \left(I-\frac{\langle I_0\rangle}{k}\mathbf{1}_k\right)\right)\right),$$
where we used that $\langle I_0\rangle=\sum_{j=1}^k(k-j)=\frac{k(k-1)}{2}$. Moreover, recall that $\tilde{I}=I-\langle I\rangle/k\mathbf{1}_k$ and $\hat{J}=J-J_{k}\mathbf{1}_k$ and remark that from the definition of $S_J$ we have for $\vec{u}\in \mathbb{R}^k$ such that $\vec{u}=0$,
$$S_{J}(\vec{u})=S_{\hat{J}+J_k\mathbf{1}_k}(\vec{u})=e^{iJ_k\sum_{l=1}^ku_l}S_{\hat{J}}(\vec{u})=S_{\hat{J}}(\vec{u}).$$
Hence, using that $S_J(\vec{u}+l\mathbf{1}_k)=\exp(i l\langle \lambda_J\rangle)S_{J}(\vec{u})$ for $l\in \mathbb{R}$, we get 
\begin{align}
\frac{S_{J}(\xi_{n}(I))}{S_{J}(\xi_{n}(I_0))}=&\exp(2i\pi (\langle I\rangle-\langle I_0\rangle)\cdot\langle \lambda_J\rangle/(kn))\frac{S_{J}\left(\xi_{n}(I)-\frac{2\pi}{n}\frac{\langle I\rangle-\langle I_0\rangle}{k}\mathbf{1}_k\right)}{S_{J}\left(2i\pi\tilde{I_0}/n\right)}\nonumber\\
=&\exp(2i\pi \langle \lambda_I\rangle\cdot\langle \lambda_J\rangle/(kn))\cdot \frac{S_{\hat{J}}(2\pi\tilde{I}/n)}{S_{\hat{J}}(2\pi\tilde{I_0}/n))}\label{eq:invariance_character_centered}
\end{align}
where we also used the fact that $\sum_{i=1}^k\xi_n(I)_i=\frac{2\pi(\langle I\rangle-\langle I_0\rangle}{n}$. Then, since $\hat{J}_k=0$, $\lambda_{\hat{J}}=\hat{J}-I_0\in \Lambda_{k}$ satisfies $\left(\lambda_{\hat{J}}\right)_k=0$. Moreover, remark that $\Vert \tilde{I}\Vert_\infty\geq (k-1)/2$, so that $\langle \lambda_{\hat{J}}\rangle\max(\Vert \tilde{I}\Vert_\infty,k/3)\leq n$ when $\langle \lambda_{\hat{J}}\rangle\cdot\Vert \tilde{I}\Vert_\infty=\langle  \hat{J}-I_0\rangle\cdot\Vert \tilde{I}\Vert_\infty\leq n$. In particular, as long as $\langle  \hat{J}-I_0\rangle\cdot\Vert \tilde{I}\Vert_\infty\leq n$, we have  by Lemma \ref{lem:asymptotic_schur} 
\begin{equation}
\frac{S_{\hat{J}}(2\pi \tilde{I}/n)}{S_{\hat{J}}(2\pi\tilde{I_0}/n)}=\frac{s_{\lambda_{\hat{J}}}\left(\exp\left(2i\pi\tilde{I}/n\right)\right)}{s_{\lambda_{\hat{J}}}\left(\exp\left(2i\pi\tilde{I_0}/n\right)\right)}=1
-\frac{(2\pi)^2}{2(k^2-1)n^2}\left[K(\tilde{I})K(\tilde{J})\right]+O\left(\left(\frac{\langle  \hat{J}-I_0\rangle\cdot \Vert \tilde{I}\Vert_\infty}{n}\right)^3\right)\label{eq:asymptotic_eigenvalue_main_contribution}
\end{equation}
where we used that $\Vert \tilde{I}\Vert_\infty\geq (k-1)/2$ to simplify the bound $O(\cdot)$ and that $\widetilde{\hat{J}}=\widetilde{J}$.

Consider now $\mu\in\mathcal{M}_1(B_{k,n})$. Using the previous relation for $J=I_{\mu}$ on the event $\langle  \hat{I}_{\mu}-I_0\rangle<\frac{n}{\Vert \tilde{I}\Vert_{\infty}}$, we get 
\begin{align*}
\mathbb{E}\left[\frac{S_{\hat{I}_{\mu}}(2\pi \tilde{I}/n)}{S_{\hat{I}_{\mu}}(2\pi\tilde{I_0}/n)}\right]&= 1-\frac{(2\pi)^2}{2(k^2-1)n^2}\mathbb{E}\left(\mathbf{1}_{\langle  \hat{I}_{\mu}-I_0\rangle<\frac{n}{\Vert \tilde{I}\Vert_{\infty}}}K\left(\widetilde{I_{\mu}}\right)\right)K\left(\widetilde{I}\right)+\frac{1}{n^3}O\left(\langle\hat{\mu}-I_0\rangle_3\Vert \tilde{I}\Vert_\infty^3\right)\\
&\hspace{8cm}+\mathbb{P}\left(\langle  \hat{I}_{\mu}-I_0\rangle>\frac{n}{\Vert \tilde{I}\Vert_{\infty}}\right),
\end{align*}
where we used that $\frac{S_{J}(\xi_{n}(I))}{S_{J}(\xi_{n}(I_0))}\leq 1$ for all $J\in B_{k,n}$ by \eqref{eq:inequality_Perron_Frobenius} to get a rough bound on the expectation $\mathbb{E}\left[\frac{S_{\hat{I}_{\mu}}(2\pi \tilde{I}/n)}{S_{\hat{I}_{\mu}}(2\pi\tilde{I_0}/n)}\mathbf{1}_{\langle  \hat{I}_{\mu}-I_0\rangle>\frac{n}{\Vert \tilde{I}\Vert_{\infty}}}\right]$. By Markov inequality, $\mathbb{P}\left(\langle  \hat{I}_{\mu}-I_0\rangle>\frac{n}{\Vert \tilde{I}\Vert_{\infty}}\right)\leq \frac{\Vert \tilde{I}\Vert_\infty^3}{n^3}\langle\hat{\mu}-I_0\rangle_3$, where we recall from \eqref{eq:definition_moments} that $\langle\hat{\mu}-I_0\rangle_3=\mathbb{E}\left[\langle \widehat{I_{\mu}}-I_0\rangle^3\right]$. Similarly, using \eqref{eq:rho(tilde_I)} to get $K(\tilde{I})\leq \langle \hat{I}-I_0\rangle^2$ for $I\in B_{k,n}$, we have
\begin{align*}
\frac{1}{k^2-1}\mathbb{E}\left(\mathbf{1}_{\langle  \hat{I}_{\mu}-I_0\rangle\geq\frac{n}{\Vert \tilde{I}\Vert_{\infty}}}K\left(\widetilde{I_{\mu}}\right)\right)\leq \frac{1}{k^2-1} \mathbb{E}\left(\mathbf{1}_{\langle  \hat{I}_{\mu}-I_0\rangle>\frac{n}{\Vert \tilde{I}\Vert_{\infty}}}\langle  \hat{I}_{\mu}-I_0\rangle^2\right)\leq  \frac{3\Vert \tilde{I}\Vert_{\infty}}{k^2n}\langle\hat{\mu}-I_0\rangle_3,
\end{align*}
so that, using also that $K(\tilde{I})\leq k\Vert \tilde{I}\Vert_{\infty}^2$, we get 
\begin{equation}\label{eq:asymptotic_eigenvalue_SUk}
\mathbb{E}\left[\frac{S_{\hat{I}_{\mu}}(2\pi \tilde{I}/n)}{S_{\hat{I}_{\mu}}(2\pi\tilde{I_0}/n)}\right]=1-\frac{(2\pi)^2}{2(k^2-1)n^2}\left[K(\mu)K(\tilde{I})\right]+\frac{1}{kn^3}O\left(\langle\hat{\mu}-I_0\rangle_3\Vert \tilde{I}\Vert_\infty^3\right),
\end{equation}
where we recall that $K(\mu)=\mathbb{E}_{\mu}(K(\widetilde{I_{\mu}}))$. A similar reasoning also yields 
\begin{equation}\label{eq:asymptotic_bound_eigenvalue_SUk}
\mathbb{E}\left[\left\vert\frac{S_{\hat{I}_{\mu}}(2\pi \tilde{I}/n)}{S_{\hat{I}_{\mu}}(2\pi\tilde{I_0}/n)}\right\vert\right]\leq1-\frac{(2\pi)^2}{2(k^2-1)n^2}\left[K(\mu)K(\tilde{I})\right]+\frac{1}{kn^3}O\left(\langle\hat{\mu}-I_0\rangle_3\Vert \tilde{I}\Vert_\infty^3\right).
\end{equation}
For the first statement, remark that we have by \eqref{eq:invariance_character_centered} , for $\alpha_0\in\mathbb{R}$,
\begin{align}
e^{-2i\pi \frac{\alpha_0\langle \lambda_I\rangle}{kn}}\mathbb{E}\left[\frac{S_{I_{\mu}}(\xi_{n}(I))}{S_{I_{\mu}}(\xi_{n}(I_0))}\right]=&\mathbb{E}\left[e^{2i\pi \frac{\langle \lambda_I\rangle\cdot(\langle \lambda_{I_{\mu}}\rangle-\alpha_0)}{kn}}\frac{S_{\hat{I}_{\mu}}(2\pi \tilde{I}/n)}{S_{\hat{I}_{\mu}}(2\pi\tilde{I_0}/n)}\right]\nonumber\\
=&\mathbb{E}\left[e^{2i\pi \frac{\langle \lambda_I\rangle\cdot(\langle \lambda_{I_{\mu}}\rangle-\alpha_0)}{kn}}\right]+\mathbb{E}\left[\frac{S_{\hat{I}_{\mu}}(2\pi \tilde{I}/n)}{S_{\hat{I}_{\mu}}(2\pi\tilde{I_0}/n)}\right]-1\nonumber\\
&\hspace{2cm}+\mathbb{E}\left[\left(e^{2i\pi \frac{\langle \lambda_I\rangle\cdot(\langle \lambda_{I_{\mu}}\rangle-\alpha_0)}{kn}}-1\right)\left(\frac{S_{\hat{I}_{\mu}}(2\pi \tilde{I}/n)}{S_{\hat{I}_{\mu}}(2\pi\tilde{I_0}/n)}-1\right)\right].\label{eq:final_estimate_decompo}
\end{align}
Using the Taylor expansion of $e^{it}$ for $t\in \mathbb{R}$ yields 
\begin{equation}\label{eq:asymptotic_eigenvalue_U1}
e^{2i\pi \frac{\langle \lambda_I\rangle\cdot(\langle \lambda_{J}\rangle-\alpha_0)}{kn}}=1+\frac{2i\pi \langle \lambda_I\rangle\cdot \left( \langle \lambda_J\rangle-\alpha_0\right)}{kn}-\frac{(2\pi)^2\langle \lambda_I\rangle^2\cdot \left( \langle \lambda_J\rangle-\alpha_0\right)^2}{2k^2n^2}+O\left(\frac{\langle \lambda_I\rangle^3\cdot\left( \langle \lambda_J\rangle-\alpha_0\right)^3}{k^3n^3}\right),
\end{equation}
for $I,J\in B_{k,n}$, with $O(\cdot)$ independent of $I,J$ and $\alpha_0$ because the term on the left-hand side is bounded. Hence, averaging on $I_{\mu}$ yields 
\begin{align}
\mathbb{E}\left[e^{2i\pi \frac{\langle \lambda_I\rangle\cdot(\langle \lambda_{I_{\mu}}\rangle-\alpha_0)}{kn}}\right]=&1+\frac{2i\pi \langle \lambda_I\rangle\cdot \left(\mathbb{E}(\langle \lambda_{I_{\mu}}\rangle)-\alpha_0\right)}{kn}-\frac{(2\pi)^2\langle \lambda_I\rangle^2\cdot \mathbb{E}\left[\left(\langle \lambda_{I_{\mu}}\rangle-\alpha_0\right)^2\right]}{2k^2n^2}\nonumber\\
&\hspace{6cm}+O\left(\frac{\langle \lambda_I\rangle^3\cdot\mathbb{E}\left[\left(\langle \lambda_{I_{\mu}}\rangle-\alpha_0\right)^3\right]}{k^3n^3}\right),\label{eq:asymptotic_eigenvalue_U1_averaged}
\end{align}
and putting \eqref{eq:asymptotic_eigenvalue_SUk} and \eqref{eq:asymptotic_eigenvalue_U1_averaged} together yields
\begin{align}
&\mathbb{E}\left[e^{2i\pi \frac{\langle \lambda_I\rangle\cdot(\langle \lambda_{I_{\mu}}\rangle-\alpha_0)}{kn}}\right]+\mathbb{E}\left[\frac{S_{\hat{I}_{\mu}}(2\pi \tilde{I}/n)}{S_{\hat{I}_{\mu}}(2\pi\tilde{I_0}/n)}\right]-1\nonumber\\
=&1+\frac{2i\pi \langle \lambda_I\rangle\cdot \left(\mathbb{E}(\langle \lambda_{I_{\mu}}\rangle)-\alpha_0\right)}{kn}-\frac{(2\pi)^2}{2(k^2-1)n^2}K(\mu)K(\widetilde{I})-\frac{(2\pi)^2\langle \lambda_I\rangle^2\cdot \mathbb{E}\left[\left(\langle \lambda_{I_{\mu}}\rangle-\alpha_0\right)^2\right]}{2k^2n^2}\label{eq:final_estimate_first_part}\\
&\hspace{6cm}+\frac{1}{n^3}O\left(\frac{\vert\langle \lambda_I\rangle\vert^3\cdot\mathbb{E}\left[\left(\langle \lambda_{I_{\mu}}\rangle-\alpha_0\right)^3\right]}{k^3}+\frac{1}{k}\langle\hat{\mu}-I_0\rangle_3\cdot\Vert \tilde{I}\Vert_\infty^3\right)\nonumber
\end{align}
Finally, we have by \eqref{eq:asymptotic_eigenvalue_main_contribution}, the fact that 
$\left(e^{2i\pi \frac{\langle \lambda_I\rangle\cdot(\langle \lambda_{I_{\mu}}\rangle-\alpha_0)}{kn}}-1\right)$ and $\left(\frac{S_{\hat{I}_{\mu}}(2\pi \tilde{I}/n)}{S_{\hat{I}_{\mu}}(2\pi\tilde{I_0}/n)}-1\right)$ are bounded by $2$ and the inequality $\vert e^{it}-1\vert\leq t$ for $t\in \mathbb{R}$,
\begin{align*}
&\left\vert\mathbb{E}\left[\left(e^{2i\pi \frac{\langle \lambda_I\rangle\cdot(\langle \lambda_{I_{\mu}}\rangle-\alpha_0)}{kn}}-1\right)\left(\frac{S_{\hat{I}_{\mu}}(2\pi \tilde{I}/n)}{S_{\hat{I}_{\mu}}(2\pi\tilde{I_0}/n)}-1\right)\right]\right\vert\\
=&\mathbb{E}\left[\left\vert e^{2i\pi \frac{\langle \lambda_I\rangle\cdot(\langle \lambda_{I_{\mu}}\rangle-\alpha_0)}{kn}}-1\right\vert\left\vert\frac{S_{\hat{I}_{\mu}}(2\pi \tilde{I}/n)}{S_{\hat{I}_{\mu}}(2\pi\tilde{I_0}/n)}-1\right\vert\mathbf{1}_{\langle  \hat{I}_{\mu}-I_0\rangle\leq n/\Vert I\Vert_{\infty}}\right]+O\left(\mathbb{P}(\langle  \hat{I}_{\mu}-I_0\rangle> n/\Vert I\Vert_{\infty})\right)\\
\leq& \mathbb{E}\left[ 2\pi \langle \lambda_I\rangle\cdot\frac{\left\vert\langle \lambda_{I_{\mu}}\rangle-\alpha_0\right\vert}{kn}\cdot\frac{(2\pi)^2}{2n^2}\frac{K(\tilde{I})K(\widetilde{I_{\mu}})}{k^2-1}\cdot\mathbf{1}_{\langle  \hat{I}_{\mu}-I_0\rangle\leq n/\Vert I\Vert_{\infty}}\right]+\frac{1}{kn^3}O\left(\langle\hat{\mu}-I_0\rangle_3\Vert \tilde{I}\Vert_\infty^3\right)\\
&\hspace{10cm}+O\left(\mathbb{P}(\langle  \hat{I}_{\mu}-I_0\rangle> n/\Vert I\Vert_{\infty})\right)\\
\leq & \frac{1}{n^3}O\left(\mathbb{E}\left(\frac{\langle \lambda_I\rangle \cdot\Big\vert\langle \lambda_{I_{\mu}}\rangle-\alpha_0\Big\vert}{k} \cdot\langle  \hat{I}_{\mu}-I_0\rangle^2\cdot\Vert \tilde{I}\Vert_{\infty}^2\right)+O\left(\langle\hat{\mu}-I_0\rangle_3\Vert \tilde{I}\Vert_\infty^3\right)\right),
\end{align*}
where we used Markov inequality to bound $\mathbb{P}(\langle  \hat{I}_{\mu}-I_0\rangle> k^{1/3}n/\Vert I\Vert_{\infty})$ by $\frac{1}{kn^3}\langle\hat{\mu}-I_0\rangle_3\Vert \tilde{I}\Vert_\infty^3$ and used the inequality $K(\tilde{I_{\mu}})\leq \langle  \hat{I}_{\mu}-I_0\rangle^2$ and $K(\tilde{I})\leq k\Vert\tilde{I}\Vert_{\infty}^2$ on the last inequality. By the Young inequality $ab\leq \frac{a^3}{3}+\frac{3b^{3/2}}{2}$ for $a,b\geq 0$, we finally get
\begin{align*}
&\mathbb{E}\left[\left(e^{2i\pi \frac{\langle \lambda_I\rangle\cdot(\langle \lambda_{I_{\mu}}\rangle-\alpha_0)}{kn}}-1\right)\left(\frac{S_{\hat{I}_{\mu}}(2\pi \tilde{I}/n)}{S_{\hat{I}_{\mu}}(2\pi\tilde{I_0}/n)}-1\right)\right]\\
&\hspace{7cm}=\frac{1}{n^3}O\left(\frac{\vert\langle \lambda_I\rangle\vert^3\cdot\mathbb{E}\left[\left(\langle \lambda_{I_{\mu}}\rangle-\alpha_0\right)^3\right]}{k^3}+\langle\hat{\mu}-I_0\rangle_3\Vert \tilde{I}\Vert_\infty^3\right).
\end{align*}
Putting the latter estimates with \eqref{eq:final_estimate_first_part} into \eqref{eq:final_estimate_decompo} and setting $\alpha_0=\mathbb{E}(\langle \lambda_{I_{\mu}}\rangle)=\langle\mu\rangle-\langle I_0\rangle$ yields then
\begin{align*}
&e^{-2i\pi\frac{(\langle\mu\rangle-\langle I_0\rangle)\cdot\langle \lambda_I\rangle}{kn}}\Phi_{n}[\mu](I^{(n)})\\
&=1+\frac{2i\pi}{kn}\langle \lambda_I\rangle\mathbb{E}_{\mu}\Big[\langle \lambda_{I_{\mu}}\rangle-(\langle\mu\rangle-\langle I_0\rangle)\Big]-\frac{(2\pi)^2}{2n^2}\left[\frac{K(\mu)K(\widetilde{I})}{k^2-1}+\frac{\langle \lambda_I\rangle^2Var(\mu)}{k^2}\right]\\
&\hspace{8cm}+\frac{1}{n^3}O\left(\frac{\vert\langle \lambda_I\rangle\vert^3\cdot Var_3(\mu)}{k^3}+\langle\hat{\mu}-I_0\rangle_3\Vert \tilde{I}\Vert_\infty^3\right)\\
&=1-\frac{(2\pi)^2}{2n^2}\left[\frac{K(\mu)K(\widetilde{I})}{k^2-1}+\frac{\langle \lambda_I\rangle^2Var(\mu))}{k^2}\right]+\frac{1}{n^3}O\left(\frac{\vert\langle \lambda_I\rangle\vert^3\cdot Var_3(\mu)}{k^3}+\langle\hat{\mu}-I_0\rangle_3\Vert \tilde{I}\Vert_\infty^3\right).
\end{align*}
 To get the second statement of the lemma, note that by \eqref{eq:invariance_character_centered}, 
$$\left\vert \Phi_{n}[\mu](I^{(n)})\right\vert=\left\vert\mathbb{E}\left[\frac{S_{I_{\mu}}(\xi_{n}(I))}{S_{I_{\mu}}(\xi_{n}(I_0))}\right]\right\vert\leq \mathbb{E}\left[\left\vert\frac{S_{I_{\mu}}(\xi_{n}(I))}{S_{I_{\mu}}(\xi_{n}(I_0))}\right\vert\right]=\mathbb{E}\left[\left\vert\frac{S_{\hat{I}_{\mu}}\left(2\pi \tilde{I}/n\right)}{S_{\hat{I}_{\mu}}\left(2\pi\tilde{I_0}/n\right)}\right\vert\right].$$ 
Using \eqref{eq:asymptotic_bound_eigenvalue_SUk} yields then the result.
\end{proof}
From Lemma \ref{lem:asymptoti_eigenvalue}, we can deduce asymptotics of the discrete Fourier transform of a convolution of Markov kernels on $ B_{k,n}$. For a sequence $\mathbf{m}=(\mu_r)_{1\leq r\leq m}$ of $h$-probability measures, recall the notation from Section \ref{subsec:main_probabilistic_results}
$$X(\mathbf{m})=\left(\frac{1}{m}\sum_{r=1}^mX(\mu_r)\right)-X(I_0),$$
for $X(\mu)=K(\mu)$ or any quantity defined in \eqref{eq:definition_moments}. We also set 
$$*\mathbf{m}=\mu_1*\mu_2*\dots\mu_m.$$
\begin{proposition}\label{prop:convergence_fourier}
Let $M>0$ and suppose that $\mathbf{m}=(\mu_r)_{1\leq r\leq m}$ is a sequence of $h$-probability measures on $ B_{k,n}$ such that $Var_{3}(\mathbf{m})\leq M Var(\mathbf{m})^{3/2}$ and $\langle\widehat{\mathbf{m}}\rangle_3\leq MK(\mathbf{m})
^{3/2}$. Then, for all $J\in B_{k}$ with $J_1<J_k+n$, we have 
$$\left\vert\Phi_{n}[*\mathbf{m}](J^{(n)})\right\vert \leq A(J) \exp\left(-\frac{m(2\pi)^2}{2n^2}\frac{K(\tilde{J})K(\mathbf{m})}{k^2-1}+\frac{mM}{n^3}O\left(\Big[kK(\tilde{J})K(\mathbf{m})\Big]^{3/2}\right)\right)$$
with $A(J)\leq 1$ and 
$$A(J)=\exp\Bigg[-\frac{m(2\pi)^2}{2n^2}\left[\frac{Var(\mathbf{m})\langle \lambda_J\rangle^2}{k^2}\right]+\frac{mM}{n^3}O\left(\frac{Var(\mathbf{m})^{3/2}\cdot\vert\langle \lambda_J\rangle\vert^3}{k^3}\right)\Bigg].$$
When $\Vert \tilde{J}\Vert_{\infty}\leq c\frac{n}{m^{1/3}M^{1/3}K(\mathbf{m})^{1/2}}$ and $\langle \lambda_J\rangle\leq c\frac{kn}{m^{1/3}M^{1/3}Var(\mathbf{m})^{1/2}}$ for some numeric $c>0$ small enough, we get
\begin{align*}
e^{-2i \pi \frac{m\langle\mathbf{m}\rangle)\cdot\langle \lambda_J\rangle}{kn}}\Phi_{n}[*\mathbf{m}](J^{(n)})
=&\exp\Bigg[-\frac{m(2\pi)^2}{2n^2}\left[\frac{K(\mathbf{m})K(\tilde{J})}{k^2-1}+\frac{Var(\mathbf{m})\langle \lambda_J\rangle^2}{k^2}\right]\\
&+\frac{mM}{n^3}O\left(\frac{Var(\mathbf{m})^{3/2}\cdot\vert\langle \lambda_J\rangle\vert^3}{k^3}+\Big[k \cdot K(\mathbf{m})\cdot K(\tilde{J})\Big]^{3/2}\right)\Bigg].
\end{align*}

\end{proposition}
%To prove this proposition, we will use the following useful inequality \footnote{We thank Quentin Berger for having given us this inequality.} whose proof is done by induction: if $u_1,\ldots,u_m\in\mathbb{C}$ are such that $\vert u_i\vert\leq 1$ for $1\leq i\leq m$, then
%\begin{equation}\label{eq:quentin_inequality}
%\left\vert u_1\dots u_m-1\right\vert\leq \sum_{i=1}^m\vert u_i-1\vert.
%\end{equation}
To prove this proposition, we will need the following estimate.
\begin{lemma}\label{lem:bound_K_norm_inf}
For all $I\in B_{k}$ with $I\not=I_0$,
$$\Vert \tilde{I}\Vert_{\infty}^2\leq kK(\tilde{I}).$$
As a consequence, for $\mu\in \mathcal{M}_{1}(B_{k,n})$,
$$\langle\hat{\mu}-I_0\rangle_3\leq 2\sqrt{2}k^3\Vert \tilde{\mu}\Vert_{3}.$$
\end{lemma}
\begin{proof}
First remark that either $\Vert \tilde{I}\Vert_{\infty}=\tilde{I}_1$ or $\Vert \tilde{I}\Vert_{\infty}=\tilde{I}_k$. Up to replacing $I$ by $(-I_k,\ldots,-I_1)$,which does not change $K(\tilde{I})$, we can assume that $\Vert \tilde{I}\Vert_{\infty}=\tilde{I}_1$. Then, recall that $K(\tilde{I})=\Vert \tilde{I}\Vert^2-\Vert I_0\Vert^2$ and write $\tilde{I}=I_0+x$ for some nondecreasing sequence $x\in \mathbb{R}^k$ satisfying $\sum_{i=1}^kx_i=0$. Then, we have
$$\sum_{i=2}^k((I_0)_i+x_i)^2=\sum_{i=2}^k(I_0)_i^2+2\sum_{i=2}^k(I_0)_ix_i+\sum_{i=2}^kx_i^2.$$
Since $I_0$ and $x$ are nondecreasing sequences, by Chebyschev's sum inequality, $\sum_{i=2}^k(I_0)_ix_i\geq \frac{1}{k-1}(\sum_{i=2}^k(I_0)_i)(\sum_{i=2}^kx_i)$, and by convexity, $\sum_{i=2}^kx_i^2\geq \frac{1}{k-1}\left(\sum_{i=2}^kx_i\right)^2$. Hence, replacing $x_i$ by $\frac{1}{k-1}\left(\sum_{i=2}^kx_i\right)$ yields a vector $\tilde{I}'=\tilde{I}_0+(x_1,\frac{1}{k-1}\sum_{i=2}^kx_i,\ldots)$ such that $\tilde{I}'_1=\tilde{I}_1$, $\sum_{i=1}^k\tilde{I}'_i=0$ and $K(\tilde{I})\geq K(\tilde{I}')$. We thus only need to show that $K(\tilde{I}')\geq \Vert \tilde{I}'_1\Vert^2=\tilde{I}_1$. 

Now, the condition $\sum_{i=1}^kx_i=0$ yields that $\frac{1}{k-1}\sum_{i=2}^kx_i=-\frac{x_1}{k-1}$. Hence,
$$kK(\tilde{I}')-\Vert \tilde{I}\Vert_{\infty}^2=(k-1)\left(x_1+(\tilde{I_0})_1\right)^2+k\sum_{i=2}^k\left((\tilde{I}_0)_i-\frac{x_1}{k-1}\right)^2-k\Vert \tilde{I}_0\Vert^2:=g(x_1).$$
Since $(I_0)_i=\frac{k+1-2i}{2}$ for $1\leq i\leq k$, we have $g'(x_1)=2(k-1)\left(x_1+\frac{k-1}{2}\right)+2k\sum_{i=2}^k\left(\frac{x_1}{k-1}-\frac{k+1-2i}{2}\right)\geq 0$ for $x_1\geq 0$. It is thus sufficient to check that $g(x_1)\geq 0$ for the smallest possible value of $x_1$ satisfying $x_1> 0$ and $x_1-\left(-\frac{x_1}{k-1}\right)\geq 1$ (since $\tilde{I}$ and thus also $\tilde{I}'$ satisfy the latter condition). This value is $x_1=\frac{k-1}{k}$, which actually corresponds to $I=(k,k-2,\ldots,0)$. Using the formula $\Vert x+\tilde{I}_0\Vert^2-\Vert \tilde{I}_0\Vert^2=\sum_{i=1}^kx_i(x_i+k+1-2i)$, we then have 
\begin{align*}
g\left(\frac{k-1}{k}\right)=&k\left(\frac{k-1}{k}\left(\frac{k-1}{k}+k-1\right)-\sum_{i=2}^k\frac{1}{k}\left(-\frac{1}{k}+k+1-2i\right)\right)-\left(\frac{k-1}{2}+\frac{k-1}{k}\right)^2\\
=&\left(\frac{(k-1)^2(k+1)}{k}-(k-1)\left(k-1-\frac{1}{k}\right)+k(k-1)\right)-\left(\frac{k+1}{2}-\frac{1}{k}\right)^2\\
=&k^2-1-\left(\frac{k+1}{2}\right)^2\geq 0.
\end{align*}
The second statement is a consequence of the first one. First, $\Vert \hat{I}\Vert_{\infty}\leq 2\Vert \tilde{I}\Vert_{\infty}$ and $\langle \hat{I}-I_0\rangle \leq k\Vert \hat{I}\Vert_{\infty}$ implies $\langle \hat{I}-I_0\rangle ^2\leq 2kK(\tilde{I})$ when $I\not=I_0$. Since the inequality is also satisfied when $I=I_0$,
$$\langle \hat{\mu}-I_0\rangle_3=\mathbb{E}\left[\langle I_{\mu}\rangle^3\right]\leq \mathbb{E}\left[\langle I_{\mu}\rangle^3\right]\leq 2\sqrt{2}k^3\mathbb{E}\left[K(\tilde{ I}_{\mu})^{3/2}\right]\leq 2\sqrt{2}k^3\mathbb{E}\left[\Vert\tilde{ I}_{\mu}\Vert^{3}\right]=\Vert \tilde{\mu}\Vert_3.$$
\end{proof}
\begin{proof}[Proof of Proposition \ref{prop:convergence_fourier}]
For $J=I_0$ the result is clear, since $\Phi_n[*\mathbf{m}](I_0)=1$ (because $*\mathbf{m}$ is again a $h$-probability measure) and $K(\tilde{I}_0)=\langle \lambda_{I_0}\rangle=0$. 

Let $J\not=I_0\in B_{k}$ with $J_1<J_k+n$ and let $\mathbf{m}=(\mu_r)_{1\leq r\leq m}$ be a sequence of $h$-probability measures satisfying the hypothesis of the proposition. 
First, Lemma \ref{lem:asymptoti_eigenvalue} yields that for all $1\leq r\leq n$,
\begin{align*}
\left\vert\Phi_{n}[\mu_r](J^{(n)})\right\vert&\leq 1-\frac{(2\pi)^2}{2n^2}\left[\frac{K(\mu_r)K(\tilde{J})}{k^2-1}\right]+\frac{1}{n^3}O\left(\langle\hat{\mu}_r-I_0\rangle_3\cdot\Vert \tilde{J}\Vert_\infty^3\right),
\end{align*}
with $O(\cdot)$ only depending on numeric constants. Using \eqref{eq:relation_conv_dft} with the inequality $1+u\leq e^u$ for $u\in\mathbb{R}$ yields then
\begin{align*}
\left\vert\Phi_{n}[*\mathbf{m}](J^{(n)})\right\vert=&\prod_{r=1}^m\left\vert \Phi_{n}[\mu_r](J^{(n)})\right\vert \\
\leq& \prod_{r=1}^m\left[1-\frac{(2\pi)^2}{2n^2}\left[\frac{K(\mu_r)K(\tilde{J})}{k^2-1}\right]+\frac{1}{n^3}O\left(\langle\hat{\mu}_r-I_0\rangle_3\cdot\Vert \tilde{J}\Vert_\infty^3\right)\right]\\
\leq&\exp\left(-\frac{(2\pi)^2}{2n^2}\frac{K(\tilde{J})\sum_{r=1}^mK(\mu_r)}{k^2-1}+\frac{1}{n^3}O\left(\Vert \tilde{J}\Vert_\infty^3\cdot\sum_{r=1}^m\langle\hat{\mu}_r-I_0\rangle_3\right)\right).
\end{align*}
Since $\Vert \tilde{J}\Vert^2_{\infty}\leq kK(\tilde{J})$ for $J\not=I_0$ by Lemma \ref{lem:bound_K_norm_inf}, we thus have, using the hypothesis of the proposition,
\begin{align*}
\frac{1}{n^3}O\left(\Vert \tilde{J}\Vert_\infty^3\cdot\sum_{r=1}^m\langle\hat{\mu}_r-I_0\rangle_3\right)=\frac{m}{n^3}O\left(k^{3/2}K(\tilde{J})^{3/2}\cdot\langle  \hat{\mathbf{m}}_r\rangle_3\right)=\frac{m}{n^3}O\left(k^{3/2}K(\tilde{J})^{3/2}\cdot K(\mathbf{m})^{3/2}\right),
\end{align*}
and thus
\begin{equation}\label{eq:proof_prop_first_ineq}
\left\vert\Phi_{n}[*\mathbf{m}](J^{(n)})\right\vert\leq\exp\left(-\frac{(2\pi)^2}{2n^2}\frac{K(\tilde{J})\sum_{r=1}^mK(\mu_r)}{k^2-1}+\frac{m}{n^3}O\left(k^{3/2}K(\tilde{J})^{3/2}\cdot K(\mathbf{m})^{3/2}\right)\right).
\end{equation}
Then, Lemma \ref{lem:asymptoti_eigenvalue} yields for all $1\leq r\leq n$,
\begin{align*}
e^{-2i\pi\frac{(\langle\mu_r\rangle-\langle I_0\rangle)\cdot\langle \lambda_J\rangle}{kn}}\Phi_{n}[\mu_r](J^{(n)})&=1-\frac{(2\pi)^2}{2n^2}\left[\frac{K(\mu_r)K(\tilde{J})}{k^2-1}+\frac{Var(\mu_r)\langle \lambda_J\rangle^2}{k^2}\right]\\
&\hspace{3cm}+\frac{1}{n^3}O\left(\frac{Var_3(\mu_r)\cdot\vert\langle \lambda_J\rangle\vert^3}{k^3}+\langle\hat{\mu}_r-I_0\rangle_3\cdot\Vert \tilde{J}\Vert_\infty^3\right),
\end{align*}
with $O(\cdot)$ only depending on numeric constants. Using the inequality $1+u\leq e^{u}$ for $u\in \mathbb{R}$, \eqref{eq:relation_conv_dft} and the hypothesis of the proposition then gives as before 
\begin{align*}
\left\vert\Phi_{n}[*\mathbf{m}](J^{(n)})\right\vert=\prod_{r=1}^m\left\vert \Phi_{n}[\mu_r](J^{(n)})\right\vert\leq&\exp\Bigg[-\frac{m(2\pi)^2}{2n^2}\left[\frac{K(\mathbf{m})K(\tilde{J})}{k^2-1}+\frac{Var(\mathbf{m})\langle \lambda_J\rangle^2}{k^2}\right]\\
&+\frac{mM}{n^3}O\left(\frac{Var_3(\mathbf{m})\cdot\vert\langle \lambda_J\rangle\vert^3}{k^3}+\Big[k \cdot K(\mathbf{m})\cdot K(\tilde{J})\Big]^{3/2}\right)\Bigg].
\end{align*}
The latter inequality together with \eqref{eq:proof_prop_first_ineq} imply the first statement of the proposition.

Since $K(\mu_r)\leq (\langle\hat{\mu}_r\rangle-\langle I_0\rangle)^2$ and $K(\tilde{J})\leq k\Vert \tilde{J}\Vert_{\infty}^2$, $\frac{K(\mu_r)K(\tilde{J})}{k^2-1}=O(\Vert \tilde{J}\Vert_{\infty}^2 \langle\hat{\mu}_r-I_0\rangle^2)$. Moreover, $Var_2(\mu_r)\leq Var_3(\mu_r)^{2/3}$. Hence, assuming that $\Vert \tilde{J}\Vert_{\infty}\leq c\frac{n}{\sup_{1\leq r\leq m}\langle\hat{\mu}_r-I_0\rangle}$ and $\langle \lambda_J\rangle\leq c\frac{kn}{\sup_{1\leq r\leq m}Var_3(\mu_r)^{1/3}}$ for $c$ small enough, we have 
$$\frac{(2\pi)^2}{2n^2}\left[\frac{K(\mu_r)K(\tilde{J})}{k^2-1}+\frac{Var(\mu_r)\langle \lambda_J\rangle^2}{k^2}\right]+\frac{1}{n^3}O\left(\frac{Var_3(\mu_r)\cdot\vert\langle \lambda_J\rangle\vert^3}{k^3}+\langle\hat{\mu}_r-I_0\rangle_3\cdot\Vert \tilde{J}\Vert_\infty^3\right)\leq 1/2.$$
The hypotheses $\langle\hat{\mathbf{m}}\rangle_3\leq M K(\mathbf{m})^{3/2}$ implies that $m^{1/3}M^{1/2}K(\mathbf{m})^{1/2}\geq \sup_{1\leq r\leq m} \langle\hat{\mu}_r-I_0r\rangle$. Hence, when $\Vert \tilde{J}\Vert_{\infty}\leq c\frac{n}{m^{1/3}M^{1/3}K(\mathbf{m})^{1/2}}$ and $\langle \lambda_J\rangle\leq c\frac{kn}{m^{1/3}M^{1/3}Var(\mathbf{m})^{1/2}}$ we have in particular  $\Vert \tilde{J}\Vert_{\infty}\leq c\frac{n}{\sup_{1\leq r\leq m}\langle\hat{\mu}_r-I_0\rangle}$. Similarly,  $\langle \lambda_J\rangle\leq c\frac{kn}{m^{1/3}M^{1/3}Var(\mathbf{m})^{1/2}}$ implies that $\langle \lambda_J\rangle\leq c\frac{kn}{\sup_{1\leq r\leq m}Var_3(\mu_r)^{1/3}}$. Hence, for such $J$, using the estimate $1+u=\exp(u+O(u^{3/2}))$ for $\vert u\vert\leq 1/2$ yields then
\begin{align*}
&e^{-2i\pi\frac{(\langle\mu_r\rangle-\langle I_0\rangle)\cdot\langle \lambda_J\rangle}{kn}}\Phi_{n}[\mu_r](J^{(n)})\\
&=\exp\Bigg[-\frac{(2\pi)^2}{2n^2}\left[\frac{K(\mu_r)K(\tilde{J})}{k^2-1}+\frac{Var(\mu_r)\langle \lambda_J\rangle^2}{k^2}\right]
+\frac{1}{n^3}O\left(\frac{Var_3(\mu_r)\cdot\vert\langle \lambda_J\rangle\vert^3}{k^3}+\langle\hat{\mu}_r-I_0\rangle^3\cdot\Vert \tilde{J}\Vert_\infty^3\right)\\
&\hspace{7.5cm}+\frac{1}{n^3}O\left(\frac{K(\mu_r)^{3/2}K(\tilde{J})^{3/2}}{k^3}+\frac{Var(\mu_r)^{3/2}\vert\langle \lambda_J\rangle\vert^3}{k^3}\right)\Bigg].
\end{align*}
Since $Var(\mu_r)^{3/2}\leq Var_3(\mu_r)$, $K(\mu_r)^{3/2}\leq \langle\hat{\mu}_r-I_0\rangle^3$ and $\frac{K(\tilde{J})}{k}\leq \Vert \tilde{J}\Vert_{\infty}^2$, 
\begin{align*}
O\left(\frac{K(\mu_r)^{3/2}K(\tilde{J})^{3/2}}{k^3}+\frac{Var(\mu_r)^{3/2}\vert\langle \lambda_J\rangle\vert^3}{k^3}\right)=O\left(\langle\hat{\mu}_r-I_0\rangle^3\Vert \tilde{J}\Vert_{\infty}^3+\frac{Var_3(\mu_r)\vert\langle \lambda_J\rangle\vert^3}{k^3}\right).
\end{align*}
Hence, since $K(\tilde{J})\leq k\Vert \tilde{J}\Vert_{\infty}^2$ by Lemma \ref{lem:bound_K_norm_inf},
\begin{align*}
&e^{-2i\pi\frac{(\langle\mu_r\rangle-\langle I_0\rangle)\cdot\langle \lambda_J\rangle}{kn}}\Phi_{n}[\mu_r](J^{(n)})\\
&=\exp\Bigg[-\frac{(2\pi)^2}{2n^2}\left[\frac{K(\mu_r)K(\tilde{J})}{k^2-1}+\frac{Var(\mu_r)\langle \lambda_J\rangle^2}{k^2}\right]\\
&\hspace{2cm}+\frac{1}{n^3}O\left(k^{3/2}\langle\hat{\mu}_r-I_0\rangle^3\cdot K(\tilde{J})^{3/2}+\frac{Var_3(\mu_r)\cdot\vert\langle \lambda_J\rangle\vert^3}{k^3}\right)\Bigg].
\end{align*}
Using \eqref{eq:relation_conv_dft} for $1\leq r\leq m$ yields then 
\begin{align*}
e^{-2i \pi \frac{m\langle\mathbf{m}\rangle)\cdot\langle \lambda_J\rangle}{kn}}\Phi_{n}[*\mathbf{m}](J^{(n)})
=&\prod_{r=1}^me^{-2i\pi\frac{(\langle\mu_r\rangle-\langle I_0\rangle)\cdot\langle \lambda_J\rangle}{kn}}\Phi_{n}[\mu_r](J)\nonumber\\
=&\exp\Bigg[-\frac{m(2\pi)^2}{2n^2}\left[\frac{K(\mathbf{m})K(\tilde{J})}{k^2-1}+\frac{Var(\mathbf{m})\langle \lambda_J\rangle^2}{k^2}\right]\nonumber\\
&\quad+\frac{m}{n^3}O\left(k^{3/2}\langle\hat{\mathbf{m}}\rangle^3\cdot K(\tilde{J})^{3/2}+\frac{Var_3(\mathbf{m})\cdot\vert\langle \lambda_J\rangle\vert^3}{k^3}\right)\Bigg].
\end{align*}
The hypothesis $Var_{3}(\mathbf{m})\leq M Var(\mathbf{m})^{3/2}$ and $\langle\hat{\mathbf{m}}\rangle_3\leq M K(\mathbf{m})^{3/2}$ yields then the second statement of the proposition.
\end{proof}
Remark that the first statement of Proposition \ref{prop:convergence_fourier} is only relevant when $K(\tilde{J})\leq c\frac{n^2k^3}{M^2K(\mathbf{m})}$ for some numeric $c>0$. We conclude this section by a crude estimate of $\Phi_n[*\mathbf{m}]$ when $K(\tilde{J})\geq cn^2$.
\begin{lemma}\label{lem:bound_Fourier_large}
For any $0<c<1/2$, there exists $0<t<1$ depending on $c$ such that for $I,J\in B_{k,n}$ with $\sup( (I_i-I_{i+1}) \mod n)\geq cn$, 
$$\frac{\vert S_{J}(\xi_{n}(I))\vert}{S_{J}(\xi_{n}(I_0))}\leq \frac{\max\left(\frac{k-2}{k},t\right)}{1+O\left(\frac{k^2\langle \lambda_{\hat{J}}\rangle^2}{n^2}\right)}.$$

In particular, there exists $\eta,\delta>0$ only depending on $k$ and $c$ such that for any sequence $\mathbf{m}=(\mu_i)_{1\leq i\leq m}$ of $h$-probability measures with $\langle\hat{\mathbf{m}}\rangle_3\leq MK(\mathbf{m})^{3/2}$ and $I\in B_{k}$ such that  $\sup( (\hat{I}_i-\hat{I}_{i+1}) \mod n)\geq cn$,
$$\Phi_{n}[*\mathbf{m}](I^{(n)})\leq \exp\left(-m\delta+ \frac{m M}{\eta^3n^3}K(\mathbf{m})^{3/2}\right).$$
\end{lemma}
\begin{proof}
We prove the first part of the lemma in the case where $n/2\geq \lambda_1-\lambda_{2}\geq cn$, the other cases being deduced similarly. Set $\lambda=\lambda_{\hat{J}}$ and $\zeta_j=e^{2i\pi i_j/n}$ for $1\leq j\leq n$. Then, by \eqref{eq:invariance_character_centered},
\begin{align*}
\left\vert S_{J}(\xi_{n}(I))\right\vert=&\left\vert s_{\lambda}(\zeta_1,\zeta_2,\zeta_3,\ldots)\right\vert\\
\leq&\sum_{\substack{\mu\subset \lambda\\l(\mu)\leq 2}}\left\vert s_{\mu}(\zeta_1,\zeta_2)s_{\lambda/\mu}(\zeta_3,\ldots,\zeta_{k})\right\vert\\
\leq&\left\vert s_{\lambda}(\zeta_3,\ldots,\zeta_{k})\right\vert+\sum_{\substack{0\leq b< a+1\leq \lambda+1\\b\leq \lambda_2}}\left\vert s_{(a,b)}(\zeta_1,\zeta_2)s_{\lambda/(a,b)}(\zeta_3,\ldots,\zeta_{k})\right\vert\\
&+\sum_{0<a\leq \lambda_2}\left\vert s_{(a,a)}(\zeta_1,\zeta_2)s_{\lambda/(a,a)}(\zeta_3,\ldots,\zeta_{k})+s_{(a,a-1)}(\zeta_1,\zeta_2)s_{\lambda/(a,a-1)}(\zeta_3,\ldots,\zeta_{k})\right\vert\\
:=&S_1+S_2+S_3.
\end{align*}
First, by the Weyl dimension formula $s_{\lambda}(\mathbf{1}_k)=\prod_{1\leq i<j\leq k}\frac{(\lambda_i-\lambda_j+j-i)}{j-i}$ for $\lambda\not=\emptyset$,
\begin{align}
S_1\leq s_{\lambda}(\mathbf{1}_{k-2})=&\delta_{l(\lambda)\leq k-2}\frac{1}{\prod_{1\leq i\leq k-2}\frac{\lambda_i+k-1-i}{k-1-i}\frac{\lambda_i+k-i}{k-i}}\cdot s_{\lambda}(\mathbf{1}_{k})\nonumber\\
\leq &\frac{k-2}{k}s_{\lambda}(\mathbf{1}_{k}).\label{eq:lemma_large_S_1}
\end{align}
Then, if $0\leq b< a$, then 
$$\vert s_{(a,b)}(\zeta_1,\zeta_2)\vert=\vert e_2(\zeta_1,\zeta_2)\vert^b\vert h_{b-a}(\zeta_1,\zeta_2)\vert=\vert h_{b-a}(\zeta_1,\zeta_2)\vert,$$
where $e_l$ and $h_l$ are respectively the elementary and complete homogeneous symmetric polynomials. For $l>0$, $h_l(\zeta_1,\zeta_2)=\frac{\zeta_1^{l+1}-\zeta_2^{l+1}}{\zeta_1-\zeta_2}$ and thus, using that $i_1\geq i_2$,
$$\vert h_l(\zeta_1,\zeta_2)\vert =\frac{\vert \sin((l+1)\pi(i_1-i_2)/n)\vert}{\sin(\pi(i_1-i_2)/n)}.$$
Remark that for $\frac{\pi}{2(l+1)}\leq x\leq \pi/2$, the inequality $\sin(u)>u-\frac{u^3}{6}$ for $u\in]0,\pi/2]$ yields
$$\frac{\vert\sin((l+1)x)\vert}{\sin(x)}\leq \frac{1}{\sin(x)}\leq \frac{1}{\sin\left(\frac{\pi}{2(l+1)}\right)}\leq \frac{2(l+1)}{\pi\left(1-\left(\frac{\pi}{2(l+1)}\right)^2/6\right)}<u_1(l+1)$$
for $u_1=\frac{2}{\pi\left(1-\left(\frac{\pi}{4}\right)^2/6\right)}<1$ independent of $l$. Hence, if $l$ is such that $\frac{1}{2(l+1)} \leq\frac{i_1-i_2}{n}\leq 1-\frac{1}{2(l+1)}$, 
$$\vert h_l((\zeta_1,\zeta_2)\vert\leq t(l+1)=ts_{(a,b)}(1,1).$$
Since $c \leq\frac{i_1-i_2}{n}\leq 1-c$, there is a finite number only depending on $c$ and $l\in \mathbb{N}$ satisfying $\min\left(\frac{i_1-i_2}{n},1-\frac{i_1-i_2}{n}\right)\leq \frac{1}{2(l+1)} $. For each of them, the function $x\mapsto\frac{\sin((l+1)x)}{\sin x}$ is strictly decreasing on $[0,\pi/(2(l+1))]$ with maximum $(l+1)$ at $0$. Hence, there exists a constant $u_2<1$ only depending on $c$ such that for any $l$ such that  $\min\left(\frac{i_1-i_2}{n},1-\frac{i_1-i_2}{n}\right)\leq \frac{1}{2(l+1)} $,
$$\frac{\vert \sin((l+1)\pi(i_1-i_k)/n)\vert}{\sin(\pi(i_1-i_2)/n)}\leq u_2(l+1).$$
We have thus found a constant $u=(u_1\vee u_2)<1$ only depending on $c$ such that for all $l\geq 1$,
\begin{equation}\label{eq:bound_h_2}
\vert h_l((\zeta_1,\zeta_2)\vert\leq u(l+1).
\end{equation}
In particular, for $0\leq b< a$,
$$\vert s_{(a,b)}(\zeta_1,\zeta_2)\vert=\vert h_{b-a}(\zeta_1,\zeta_2)\vert\leq u(b-a+1)=us_{(a,b)}(1,1),$$
and summing on all contributions yields
\begin{equation}\label{eq:lemma_large_S_2}
S_2\leq \sum_{\substack{0\leq b< a+1\leq \lambda+1\\b\leq \lambda_2}} s_{(a,b)}(1,1)s_{\lambda/(a,b)}(\mathbf{1}_{k-2}).
\end{equation}
Then, for $a\leq \lambda_2$, remark that $s_{\lambda/(a,a)}(\mathbf{1}_{k-2})\leq s_{\lambda/(a,a-1)}(\mathbf{1}_{k-2})$. Indeed, any semistandard filling of $\lambda/(a,a)$ yields a semistandard filling of $\lambda/(a,a-1)$ by putting $1$ in the cell $(2,a)$. Hence, using that $\vert s_{(a,a)}(\zeta_1,\zeta_2)\vert=1$ and 
$$\vert s_{(a,a-1)}(\zeta_1,\zeta_2)\vert=\vert h_1(\zeta_1,\zeta_2)\vert\leq 2\cos(c\pi)=\cos(c\pi)s_{(a,a-1)}(1,1),$$ 
we get 
\begin{align*}
&\left\vert s_{(a,a)}(\zeta_1,\zeta_2)s_{\lambda/(a,a)}(\zeta_3,\ldots,\zeta_{k})+s_{(a,a-1)}(\zeta_1,\zeta_2)s_{\lambda/(a,a-1)}(\zeta_3,\ldots,\zeta_{k})\right\vert\\
&\hspace{4cm}\leq s_{(a,a)}(1,1)s_{\lambda/(a,a)}(\mathbf{1}_{k-2})+\cos(c\pi)s_{(a,a-1)}(1,1)s_{\lambda/(a,a-1)}(\mathbf{1}_{k-2})\\
&\hspace{4cm}\leq \frac{1+\cos(c\pi)}{2}\left(s_{(a,a)}(1,1)s_{\lambda/(a,a)}(\mathbf{1}_{k-2})+s_{(a,a-1)}(1,1)s_{\lambda/(a,a-1)}(\mathbf{1}_{k-2})\right).
\end{align*}
Summing on all $a\leq \lambda_2$ yields then
\begin{equation}\label{eq:lemma_large_S_3}
S_3\leq \frac{1+\cos(c\pi)}{2}\sum_{a\leq \lambda_2}\left(s_{(a,a)}(1,1)s_{\lambda/(a,a)}(\mathbf{1}_{k-2})+s_{(a,a-1)}(1,1)s_{\lambda/(a,a-1)}(\mathbf{1}_{k-2})\right).
\end{equation}
Putting \eqref{eq:lemma_large_S_1},\eqref{eq:lemma_large_S_2} and \eqref{eq:lemma_large_S_3} together yields the existence of $t<1$ depending only on $c$ such that 
\begin{align*}
S_1+S_2+S_3\leq \max\left(\frac{k-2}{k},t\right) \sum_{\substack{\mu\subset \lambda\\l(\mu)\leq 2}}s_{\mu}(1,1)s_{\lambda/\mu}(\mathbf{1}_{k-2})\leq \max\left(\frac{k-2}{k},t\right) s_{\lambda}(\mathbf{1}_{k}).
\end{align*}
On the other hand, by \eqref{eq:rho_final_asymptotic} we have 
$$\frac{S_{J}(\xi_{n}(I_0))}{S_{J}(0,\ldots,0)}=1+O\left(\frac{k^2\langle \lambda\rangle^2}{n^2}\right),$$
so that finally
$$\frac{\vert S_{J}(\xi_{n}(I))\vert}{S_{J}(\xi_{n}(I_0))}\leq \frac{\max\left(\frac{k-2}{k},t\right)}{1-c'\frac{k^2\langle \lambda_{\hat{J}}\rangle^2}{n^2}},$$
with $t<1$ depending on $c$ and $c'>0$ numeric.

For the second statement, recalling that $\Phi_n[\mu](I^{(n)})=\mathbb{E}\left[\frac{S_{I_{\mu}}(\xi_{n}(I^{(n)}))}{S_{I_{\mu}}(\xi_{n}(I^{(n)}))}\right]=\mathbb{E}\left[\frac{S_{I_{\mu}}(\xi_{n}(I))}{S_{I_{\mu}}(\xi_{n}(I))}\right]$ and that $\left\vert \frac{S_{J}(\xi_{n}(I))}{S_{J}(\xi_{n}(I_0))}\right\vert\leq 1$ for all $J\in B_{k,n}$, for $\eta>0$ to choose later we have 
$$\left\vert\Phi_n[\mu](I^{(n)})\right\vert\leq \mathbb{P}(\langle  \hat{I}_{\mu}\rangle >\eta n))+\mathbb{E}\left[\left\vert\frac{S_{I_{\mu}}(\xi_{n}(I))}{S_{I_{\mu}}(\xi_{n}(I))}\right\vert\mathbf{1}_{\langle  \hat{I}_{\mu}\rangle \leq \eta n}\right),$$
so that by Markov inequality and the first part of the statement, there exists $t>0$ only depending on $c$ such that
$$\left\vert\Phi_n[\mu](I^{(n)})\right\vert\leq \frac{\mathbb{E}\left[\langle  \hat{I}_{\mu}\rangle^3\right]}{\eta^3n^3}+\frac{\max(\frac{k-2}{k},t)}{1-c' k^2\eta^2}\leq 1-\delta+\frac{\mathbb{E}\left[\langle  \hat{I}_{\mu}\rangle^3\right]}{\eta^3n^3},$$
where we choose $\eta$ (only depending on $c$ and $k$) such that $\frac{\max(\frac{k-2}{k},t)}{1-c' k^2\eta^2}=1-\delta$ with $\delta>0$. 

Let $\mathbf{m}=(\mu_r)_{1\leq r\leq m}$ be a sequence of $h$-probability measures satisfying the hypotheses of the lemma. Applying the bound $1-\delta+\frac{E\left[\langle  \hat{I}_{\mu}\rangle^3\right]}{\eta^3n^3}\leq \exp\left(-\delta+\frac{\mathbb{E}\left[\langle  \hat{I}_{\mu}\rangle^3\right]}{\eta^3n^3}\right)$ and taking the product on all $\Phi_n[I_{\mu_r}](I^{(n)})$ yield then with \eqref{eq:relation_conv_dft}
\begin{align*}
\Phi_{n}[*\mathbf{m}](I^{(n)})\leq& \exp\left(-m\delta+ \frac{1}{n^3\eta^3}\sum_{r=1}^m\mathbb{E}\left[\langle\hat{I}_{\mu_r}\rangle^3\right]\right)\\
\leq &\exp\left(-m\left(\delta+ \frac{M}{n^3\eta^3}K(\mathbf{m})^{3/2}\right)\right).
\end{align*}
\end{proof}

\section{Berry-Esseen theorem on $B_{k,n}$}\label{sec:Berry-Esseen}

We state and prove in this section the first main result of the manuscript, which consists in a Berry-Esseen type theorem for the measure $\xi_n(*\mathbf{m})$. The complete result is a bit more general than the simplified version written in Theorem \ref{thm:Berry-Esseen_simplified}, in order to give a more precise constant.

Recall that the shift action of $\mathbb{R}$ on $T_k$ is given for $\vec{u}\in [0,2\pi[^k$ and $t\in \mathbb{R}$ by 
$$R_t( \vec{u})=Sort(\overline{u_1-t},\ldots,\overline{u_k-t}),$$ where $Sort(\vec{v})$ is the sequence obtained from $\vec{v}\in [0,2\pi[^k$ by sorting it in increasing order, and $R_t( \mu)$ for the image of the measure $\mu\in \mathcal{M}_1( T_{k})$ by the map $\vec{u}\mapsto R_t(\vec{u})$. We then have, for $J\in B_{k}$,
\begin{align}
\Phi[R_t(\mu)](J)=\int_{ T_{k}}S_{J}(\vec{u})\vert V(\vec{u})\vert^2dR_t\mu(\vec{u})=&\int_{ T_{k}}S_{J}(\vec{u}+t\mathbf{1}_{k})\vert V(\vec{u}+t\mathbf{1}_{k})\vert^2d\mu(\vec{u})\nonumber\\
=&\exp(i\langle \lambda_J\rangle t)\int_{ T_{k}}S_{J}(\vec{u})\vert V(\vec{u})\vert^2d\mu(\vec{u})\nonumber\\
=&\exp(i\langle \lambda_J\rangle t)\Phi[\mu](J).\label{eq:effect_rotation_Fourier}
\end{align}
Comparing Proposition \ref{prop:convergence_fourier} with \eqref{eq:Fourier_Dyson_Brownian_motion} suggests a convergence of the distribution of $\xi_n(*\mathbf{m})$, viewed as a probability distribution on $ T_{k}$, to the marginal distribution of a unitary Brownian motion $B^{\alpha,\gamma}$ at time $m/n^2$ with adequate parameters $\alpha,\gamma$. We will quantify this convergence in terms of the $1$-Wassertein metric on the space of probability measures on $ T_{k}$. Due to the circular nature of $ T_{k}$, define the metric
$$d_{ T_{k}}(z,z')=\sqrt{\inf_{\sigma\in S_k}\sum_{i=1}^kd_{\mathbb{R}}(z_{i}-z_{\sigma(i)}',2\pi\mathbb{Z})^2}$$
on $ T_{k}$. This is indeed a metric, since $d_{ T_{k}}(z,z')=W_2(\delta_{z},\delta_{z'})$, where the Wasserstein distance $W_2$ is taken on the metric space $\mathbb{R}/\mathbb{Z}$ with the natural distance and $\delta_{z}=\sum_{i=1}^k\delta_{z_i}$. Remark that when $z,z'\in ]0,1/2[^k$, one has $d_{ T_{k}}(z,z')=\Vert z-z'\Vert_2$.

We then define the 1-Wasserstein distance $W_{1}^{T_{k}}$ on $ T_{k}$ with respect to the metric $d_{ T_{k}}$ as  
$$W_{1}^{T_{k}}(\mu,\nu)=\sup\left(\left\vert \int_{ T_{k}}f d\mu -\int_{ T_{k}}f d\nu\right\vert, \, f: ( T_{k},d_{ T_{k}})\rightarrow(\mathbb{R},d_{\mathbb{R}}) \,\text{ $1-$Lipschitz}\right),$$
for $\mu,\nu\in  \mathcal{M}_1(T_{k})$.
\begin{theorem}[Berry-Esseen Theorem]\label{thm:Berry-Esseen}
Let $M>0$ and suppose that $\mathbf{m}=(\mu_r)_{1\leq r\leq m}$ is a sequence of $h$-probability measures on $ B_{k,n}$ such that $Var_{3}(\mathbf{m})\leq M Var\langle\mathbf{m}\rangle^{3/2}$ and $\langle\hat{\mathbf{m}}\rangle_3\leq MK(\mathbf{m})^{3/2}$, and set $t_0=\frac{(2\pi)^2m}{n^2}$. Then, there are numeric constants $c_1,c_2>0$ such that for all $I\in  B_{k,n}$, setting $\mu=R_{\frac{m\langle\mathbf{m}\rangle}{kn}}\left(\xi_{n}\left[(\ast \mathbf{m})*\delta_I\right]\right)$ and $\nu=B^{\alpha,\gamma}_{\xi_n(I)}\left(t_0\right)$ (viewed as a distribution on $T_k$),
$$W_1^{T_k}\left(\mu, \nu\right)\leq \frac{M}{n}\left(c_1(k^8\gamma+k\alpha)\log n+\frac{c_2k^{9/2}}{\sqrt{t_0}}K_{\gamma t_0/2}^{SU(k)}\left(\xi_n(I),\xi_n(I)\right)^{1/2}\right),$$
for some numeric constants $c_1,c_2>0$, with
$$\alpha=\frac{1}{2}Var(\mathbf{m}), \, \gamma=\frac{k}{2(k^2-1)}K(\mathbf{m}).$$ 
\end{theorem}
Theorem \ref{thm:Berry-Esseen_simplified} given in Section \ref{sec:intro_model} is then a simplified version of the latter result, using Lemma \ref{lem:bound_K_norm_inf} to bound $\langle \hat{\mathbf{m}}\rangle_3$ by $2\sqrt{2}k^3\Vert \tilde{\mathbf{m}}\Vert_3$. The main ingredients in the proof of Theorem \ref{thm:Berry-Esseen} is Proposition \ref{prop:convergence_fourier} together the adaption of an approach of \cite{BobLed} relating the Wassertein distance to the Fourier transform. 

\subsection{Wasserstein distance on $T_k$} 
In \cite{BobLed}, the authors related the Wassertein distance between two probability measures on $\mathbb{R}/2\pi\mathbb{Z}$ to a weighted norm of the difference of their Fourier transform. We will use the same method on $T_k$, taking advantage of the Fourier transform which has been defined in Section \ref{Sec:Fourier_continuous} (the reader should refer to this section for the notation used in this paragraph). For $J\in B_{k}$, set $\kappa(J):=\Vert J-\frac{k-1}{2}\mathbf{1}_k\Vert^2-\Vert \tilde{I}_0\Vert^2$. Then, according to \eqref{eq:Fourier_Dyson_Brownian_motion}, $\kappa(J)=\kappa_{k,k}(J)$ for $J\in B_{k,n}$. We first have the following result relating the Wassertein distance between two probability measures on $U(k)$ to the difference between their Fourier transform, in a similar way as in \cite{BobLed}.
\begin{lemma}
For $\mu,\nu\in\mathcal{M}_1( U(k))$ which are conjugation invariant and all $t>0$,
$$W_{1, U(k)}(\mu,\nu)\leq 2\sqrt{2t}k+\sqrt{\sum_{\substack{J\in B_{k}\\J\not=I_0}}\frac{\exp(-2\kappa(J)t)}{\kappa(J)}\left\vert \Phi[\mu](J)-\Phi[\nu](J)\right\vert^2},$$
where the Wassertein distance is taken with respect to the metric on $U(k)$ induced by $\langle \cdot,\cdot\rangle_{k,k}$ on $\mathfrak{u}_k$.
\end{lemma}
\begin{proof}
Note first that for any function $f:U(k)\rightarrow \mathbb{R}$ central and $1$-Lipschitz, $\nabla f$ is almost-everywhere defined and essentially bounded by $1$ so that
$$-\int_{U(k)}f\Delta(f)d\mu_H=\int_{U_k} \Vert\nabla f\Vert^2d\mu_H\leq 1.$$
Since $f$ is central, we have the expansion $f=\sum_{\lambda\in  B_{k}}\Phi[f](J)\chi_J$ where $(\chi_J)_{J\in B_{k}}$ is an orthonormal family of $L^2(U(k),\mu_H)$ satisfying $\Delta \chi_{J}=-\kappa(J) \chi_J$ (see \cite[Prop. 12.1.2]{Far}) and such that $\chi_{J}=S_J\circ p$, see Section \ref{Sec:Fourier_continuous}. Hence,
$$\sum_{J\in B_{k}}\kappa(J)\vert\Phi[f](J)\vert^2=-\int_{U(k)}f\Delta(f)d\mu\leq 1.$$
By using Parseval and Cauchy-Schwartz equalities, we deduce that for any central probability densities $g_1,g_2$ on $U(k)$ such that $g_1,g_2\in L^2(U(k),\mu_H)$, 
\begin{align}
\left\vert \int_{U(k)}fg_1d\mu_H-\int_{U(k)}fg_2d\mu_H\right\vert=&\left\vert \int_{U(k)}f(g_1-g_2)d\mu_H\right\vert\nonumber\\
=&\left\vert \sum_{\substack{J\in B_{k}\\J\not=I_0}}\overline{\Phi[f](J)}\Phi[g_1-g_2](J)\right\vert\nonumber\\
\leq&\sqrt{\sum_{\substack{J\in B_{k}\\J\not=I_0}}\kappa(J)\left\vert \Phi[f](J)\right\vert^2}\cdot\sqrt{\sum_{\substack{J\in B_{k}\\J\not=I_0}}\frac{1}{\kappa(J)}\left\vert \Phi[g_1](J)-\Phi[g_2](J)\right\vert^2}\nonumber\\
\leq&\sqrt{\sum_{\substack{J\in B_{k}\\J\not=I_0}}\frac{1}{\kappa(J)}\left\vert \Phi[g_1](J)-\Phi[g_2](J)\right\vert^2},\label{eq:Berry-Essen_bound_first_inequality}
\end{align}
where we used on the second equality that $\Phi[g_1-g_2](I_0)=0$ because $\int_{U(k)}g_1d\mu_H=\int_{U(k)}g_2d\mu_H$.

For arbitrary conjugation invariant probability measures $\mu,\nu$ on $U(k)$ we use as in \cite{BobLed} the heat kernel on $U(k)$ for regularization purpose. Such heat kernel corresponds to the special case $\alpha=\gamma=k$ of the family of heat kernels defined in Section \ref{Sec:Brownian_motion}. Denote by $\mu_t$ (resp. $\nu_t$) a solution at time $t\geq 0$ of the heat equation with initial measure $\mu$ (resp $\nu$). On the one hand, by the diffusion equation of the heat equation, see \eqref{eq:Fourier_Dyson_Brownian_motion} with $\alpha=\gamma=k$,
\begin{equation}\label{eq:fourier_heat_kernel}
\Phi[\mu_t](J)=e^{-\kappa(J)t}\Phi[\mu](J).
\end{equation}
On the other hand, denoting by $K_t(x,y)$ the heat kernel on $U(k)$ and $d(x,y)$ the distance on $U(k)$,
\begin{align*}
\left\vert\int_{U(k)}fd\mu_t-\int_{U(k)}fd\mu\right\vert=& \left\vert\int_{U(k)}\left(\int_{U(k)}K_t(x,y)f(y)d\mu_H(y)\right)d\mu(x)-\int_{U(k)}fd\mu\right\vert\\\leq&\int_{U_k}\left(\int_{U(k)}K_t(x,y)\left\vert f(y)-f(x)\right\vert d\mu_H(y)\right)d\mu(x)\\
\leq&\int_{U_k}\left(\int_{U(k)}K_t(x,y)d(x,y)d\mu_H(y)\right)d\mu(x)\\
\leq&\int_{U_k}\left(\int_{U(k)}K_t(e,y)d(e,y)d\mu_H(y)\right)d\mu(x)\\
\leq&\int_{U(k)}K_t(e,y)d(e,y)d\mu_H \leq \sqrt{\int_{U(k)}K_t(e,y)d(e,y)^2d\mu_H},
\end{align*}
where we used the fact that $f$ is $1$-Lipschitz on the second inequality and the bi-invariant property of the heat-kernel on $U_k$ on the third inequality. By \cite[Lemma 6]{Bord} (which is a consequence of the Laplacian comparison theorem) applied to the Riemannian manifold $U(k)$ of dimension $k^2$ and positive scalar curvature,
$$\int_{U(k)}K_t(e,y)d(e,y)^2d\mu_H\leq 2k^2t.$$
Hence,
\begin{equation}\label{eq:Laplacian_comparison_estimate}
\left\vert\int_{U(k)}fd\mu_t-\int_{U(k)}fd\mu\right\vert\leq k\sqrt{2t}.
\end{equation}
Using \eqref{eq:Berry-Essen_bound_first_inequality}, \eqref{eq:fourier_heat_kernel} and \eqref{eq:Laplacian_comparison_estimate} yields then for $t>0$
\begin{align*}
\left\vert\int_{U(k)}fd\mu-\int_{U(k)}fd\nu\right\vert\leq& \left\vert\int_{U(k)}fd\mu-\int_{U(k)}fd\mu_t\right\vert+\left\vert\int_{U(k)}fd\mu_t-\int_{U(k)}fd\nu_t\right\vert\\
&\hspace{7cm}+\left\vert\int_{U(k)}fd\nu-\int_{U(k)}fd\nu_t\right\vert\\
\leq& 2\sqrt{2t}k+\sqrt{\sum_{\substack{J\in B_{k}\\J\not=I_0}}\frac{\exp(-2\kappa(J)t)}{\kappa(J)}\left\vert \Phi[\mu](J)-\Phi[\nu](J)\right\vert^2}.
\end{align*}
Taking the supremum on all $f:U(k)\rightarrow \mathbb{R}$ which are $1$-Lipschitz yields then
$$W_{1,U(k)}(\mu,\nu)\leq 2\sqrt{2t}k+\sqrt{\sum_{\substack{J\in B_{k}\\J\not=I_0}}\frac{\exp(-2\kappa(J)t)}{\kappa(J)}\left\vert \Phi[\mu](J)-\Phi[\nu](J)\right\vert^2}$$
for $\mu,\nu$ central probability distributions on $U(k)$. To conclude, for $\mu,\nu\in \mathcal{M}_1( T_{k})$ we have by Lemma \ref{lem:coercive_wasserstein} which is proven below
$$W_{1, T_{k}}(\mu,\nu)\leq W_{1,U(k)}(p^{\star}(\mu),p^{\star}(\nu))\leq 2\sqrt{2t}k+\sqrt{\sum_{\substack{J\in B_{k}\\J\not=I_0}}\frac{\exp(-2\kappa(J)t)}{\kappa(J)}\left\vert \Phi[\mu](J)-\Phi[\nu](J)\right\vert^2}.$$
\end{proof}
The latter lemma involves measures on $U_k$, whereas we are rather interested in measures on $T_k$. Recall from Section \ref{Sec:Fourier_continuous} that any probability measure $\mu$ on $ T_{k}$ yields a probability measure $p^{\star}(\mu)$ on $U_k$ by considering the unique conjugation invariant measure whose image by $p$ is $\mu$. Next lemma shows that the map $p^{\star}$ is coercive from the metric space $(T_k,W_{1,T_k})$ to the metric space $(U_k,W_{1,U_k})$.
\begin{lemma}\label{lem:coercive_wasserstein}
For any $\mu,\nu\in \mathcal{M}_1( T_{k})$,
$$W_{1,U(k)}(p^{\star}(\mu),p^{\star}(\nu))\geq W_{1, T_{k}}(\mu,\nu).$$
\end{lemma}
\begin{proof}
For any function $f: T_{k}\rightarrow\mathbb{R}$, set $\hat{f}=f\circ p$, where $p:U_k\rightarrow  T_{k}$ is the projection on the ordered set of eigenvalues. Let $U,U'\in U(k)$. Remark first that $\Vert U-U'\Vert_{HS}\leq d_{U_k}(U,U')$, where the former norm is the Hilbert-Schmidt norm on $M_k(\mathbb{C})$, see \cite[Lemma 1.3]{Meck} (this can be deduced from the fact that $\Vert \cdot\Vert_{HS}$ is the Euclidean distance on $M_{n}(\mathbb{C})$ and $U(k)$ is a Riemannian submanifold of $M_n(\mathbb{C})$). Let $(\lambda_i)_{1\leq i\leq k}$ (resp. $(\mu_i)_{1\leq i\leq k}$) be the eigenvalues of $U$ (resp $U'$). Then, by Hoffman-Wielandt inequality,
$$\min_{\sigma \in S_{k}} \sum_{i=1}^k\vert \lambda_i-\mu_{\sigma_i}\vert^2\leq \Vert U-U'\Vert_{HS},$$
and, since we have $d(u_1-u_2,2\pi\mathbb{Z})\leq \vert \exp(i u_1)-\exp(iu_2)\vert$ for $u_1,u_2\in \mathbb{R}$,
$$d_{ T_{k}}(p(U),p(U'))\leq \min_{\sigma \in S_{k}} \sum_{i=1}^k\vert \lambda_i-\mu_{\sigma_i}\vert^2\leq \Vert U-U'\Vert_{HS}\leq d_{U_k}(U,U').$$
Hence, if $f$ is Lipschitz with Lipschitz constant equal to $1$, we have 
$$\vert \hat{f}(U)-\hat{f}(U')\vert=\vert f(p(U))-f(p(U')\vert\leq d_{T_k}(p(U),p(U'))\leq d_{U_k}(U,U'),$$
and $\hat{f}$ is again Lipschitz with Lipschitz constant equal to $1$. Since 
$$\int_{U_k}f\circ pdp^{\star}(\mu)=\int_{ T_{k}}fdp_{\star} p^{\star}(\mu)=\int_{ T_{k}}fd\mu,$$
we deduce that 
\begin{align*}
W_{1,U(k)}(p^{\star}(\mu),p^{\star}(\nu))&=\sup\left(\left\vert \int_{U_k}f dp^{\star}(\mu) -\int_{U_k}f dp^{\star}(\nu)\right\vert, \, f: (U_k,d_{U_k})\rightarrow(\mathbb{R},d_{\mathbb{R}}) \,\text{ $1-$Lipschitz}\right),\\
&\geq \sup\left(\left\vert \int_{U_k}\hat{f} dp^{\star}(\mu) -\int_{U_k}\hat{f} dp^{\star}(\nu)\right\vert, \, f: ( T_{k},d_{ T_{k}})\rightarrow(\mathbb{R},d_{\mathbb{R}}) \,\text{ $1-$Lipschitz}\right)\\
&\geq \sup\left(\left\vert \int_{ T_{k}}f d\mu -\int_{ T_{k}}f d\nu\right\vert, \, f: ( T_{k},d_{ T_{k}})\rightarrow(\mathbb{R},d_{\mathbb{R}}) \,\text{ $1-$Lipschitz}\right)\\
&\geq W_{1, T_{k}}(\mu,\nu).
\end{align*}
\end{proof}
Combining the two previous lemmas yields then the following proposition.
\begin{proposition}\label{prop:Ledoux_approach}
For $\mu,\nu\in\mathcal{M}_1(T_k)$ which are conjugation invariant and all $t>0$,
$$W_{1, U(k)}(\mu,\nu)\leq 2\sqrt{2t}k+\sqrt{\sum_{\substack{J\in B_{k}\\J\not=I_0}}\frac{\exp(-2\kappa(J)t)}{\kappa(J)}\left\vert \Phi[\mu](J)-\Phi[\nu](J)\right\vert^2}.$$
\end{proposition}
\subsection{Proof of Theorem \ref{thm:Berry-Esseen}}
We will use several times the heat kernel on the Riemannian manifolds $U(k)$ and $SU(k)$. In particular, we will need the following bound which is a straightforward deduction of the parabolic Harnack estimates of Li and Yau \cite[Thm 2.3]{LiYa}.
\begin{lemma}\label{lem:rough_bound_heat}
Let $\mathbf{K}_t$ be the heat kernel either on $SU(k)$ or $U(k)$. Then, for all $x\in SU(k)$ (resp. $U(k)$) and $0<t_1<t_2$
$$\mathbf{K}_t(x,x)\leq \left(\frac{t_2}{t_1}\right)^{(k^2-\epsilon)/2}\exp\left(\frac{k\pi^2}{4(t_2-t_1)}\right),$$
where $\epsilon=1$ (resp. $\epsilon=0$). 
\end{lemma}
\begin{proof}
Let $G$ be a compact Lie group and $\mathbf{K}_t$ the heat kernel on $G$ corresponding to a translation invariant distance $d_G$ on $G$. Since $G$ has positive curvature, taking the limit $\alpha\rightarrow 1$ in \cite[Theorem 2.3]{LiYa} yields that for all $y\in G$,
$$\mathbf{K}_{t}(x,x)\leq \mathbf{K}_{t}(x,y)\left(\frac{t_2}{t_1}\right)^{\dim G/2}\exp\left(\frac{d_{G}(x,y)^2}{4(t_2-t_2)}\right).$$
Hence, averaging $y$ on all $G$ and using that $\int_{G}\mathbf{K}_{t}(x,y)dg=1$, we deduce that
$$\mathbf{K}_{t}(x,x)\leq \left(\frac{t_2}{t_1}\right)^{\dim G/2}\exp\left(\frac{\text{diam}(G)^2}{4(t_2-t_2)}\right).$$
Since the diameter of $U(k)$ and $SU(k)$ with respect to $d_{U(k)}$ are smaller than $\sqrt{k}\pi$ and $\dim U(k)=k^2$ and $\dim SU(k)=k^2-1$, the result is deduced.
\end{proof}
\begin{proof}[Proof of Theorem \ref{thm:Berry-Esseen}]
Set $\mu=B^{\alpha,\gamma}_{\xi_n(I)}\left(t_0\right)$ and $\nu=R_{\frac{m\langle\mathbf{m}\rangle}{kn}}\left(\xi_{n}\left[\ast \mathbf{m}*\delta_I\right]\right)$, and recall that $\alpha=\frac{1}{2}Var(\mathbf{m})$ and $\gamma=\frac{k}{2(k^2-1)}K(\mathbf{m})$. By Proposition \ref{prop:Ledoux_approach}, for any $t>0$
\begin{align}
W_{1, T_{k}}(\mu,\nu)\leq 2\sqrt{2t}k+\sqrt{\sum_{\substack{J\in B_{k}\\J\not=I_0}}\frac{\exp(-2\kappa(J)t)}{\kappa(J)}\left\vert \Phi[\mu](J)-\Phi[\nu](J)\right\vert^2}.\label{eq:proof_thm_1_recall_Wasser}
\end{align}
By \eqref{eq:Fourier_Dyson_Brownian_motion}, 
\begin{equation}\label{eq:formula_Phi_mu}
\Phi[\mu](J)=\exp\left(-t_0\frac{\gamma}{k} K(\tilde{J})-t_0\frac{\alpha}{k^2}\langle\lambda_J\rangle^2\right)S_{J}(\xi_n(I)),
\end{equation}
and by Proposition \ref{prop:equiv_dft_cft}, when $\Vert \tilde{J}\Vert_{\infty}< \frac{n}{2}$ we have $\Phi[\nu](J)=S_{J}(\xi_n(I_0))\Phi_n[*\mathbf{m}*\delta_I](J^{(n)})$, where $J^{(n)}$ is the unique element of $B_{k,n}$ such that $\{J^{(n)}_i[n]\}_{1\leq i\leq k}=\{J_i[n]\}_{1\leq i\leq k}$. By \eqref{eq:relation_conv_dft} and \eqref{eq:Fourier_probability} with Lemma \ref{lem:formula_inversion_Schur}, we then have 
\begin{align}
\Phi[\nu](J)=&e^{-2i \pi \frac{m\langle\mathbf{m}\rangle\cdot\langle \lambda_J\rangle}{kn}}S_{J}(\xi_n(I_0))\Phi_n[*\mathbf{m}*\delta_I](J^{(n)})\nonumber\\
=&e^{-2i \pi \frac{m\langle\mathbf{m}\rangle\cdot\langle \lambda_J\rangle}{kn}}S_{J}(\xi_n(I_0))\Phi_n[*\mathbf{m}](J^{(n)})\frac{S_{J^{(n)}}(\xi_n(I_0))}{S_{J^{(n)}}(\xi_n(I_0))}\nonumber\\
=&e^{-2i \pi \frac{m\langle\mathbf{m}\rangle\cdot\langle \lambda_J\rangle}{kn}}\Phi_n[*\mathbf{m}](J^{(n)})S_{J}(\xi_n(I)),\label{eq:formula_Phi_nu}
\end{align}
where we used that $S_{J^{(n)}}(\xi_n(I))=S_{J}(\xi_n(I))$ on the last equality.
Let us split the latter sum as follows: introduce the thresholds $\theta_{\gamma}=\min(\frac{c_1}{k^3\sqrt{\gamma}M},\frac{1}{4},\theta_{\alpha})=\frac{c_2}{\sqrt{\alpha}M}$ for some numeric constants $c_1,c_2$ to choose later and set $\nu_{\gamma}= \frac{2c}{km^{1/3}M^{1/3}\sqrt{\gamma}}$ and $\nu_{\alpha}= 2c\frac{k}{m^{1/3}M^{1/3}\sqrt{\alpha}}$, where $c$ is the numerical constant given in Proposition \ref{prop:convergence_fourier}. Then, define
$$R_{A}=\left\{J\in B_{k}\setminus\{I_0\},\frac{\sqrt{K(\tilde{J})}}{\nu_{\gamma}}\vee\frac{\vert\langle \lambda_J\rangle\vert}{ \nu_{\alpha}}\leq n\right\},$$
$$R_{B}=\left\{J\in B_{k}\setminus\{I_0\},\frac{\sqrt{K(\tilde{J})}}{\nu_{\gamma}}\vee\frac{\vert\langle \lambda_J\rangle\vert}{ \nu_{\alpha}}\geq n,\frac{\sqrt{K(\tilde{J})}}{\theta_{\gamma}}\vee\frac{\vert\langle \lambda_J\rangle\vert}{ \theta_{\alpha}}\leq n\right\},$$
and 
$$R_{C}=\left\{J\in B_{k}\setminus\{I_0\},\frac{\sqrt{K(\tilde{J})}}{\theta_{\gamma}}\vee\frac{\vert\langle \lambda_J\rangle\vert}{ \theta_{\alpha}}\geq n\right\}.$$
We then have
\begin{align*}
\sum_{\substack{J\in B_{k}\\J\not=I_0}}\frac{\exp(-2\kappa(J)t)}{\kappa(J)}\left\vert \Phi[\mu](J)-\Phi[\nu](J)\right\vert^2\leq& \sum_{i\in\{A,B,C\}} \sum_{J\in R_i}\frac{\exp(-2\kappa(J)t)}{\kappa(J)}\left\vert \Phi[\mu](J)-\Phi[\nu](J)\right\vert^2\\
:=&S_A+S_B+S_C.
\end{align*}

\textbf{Bound of $S_C$:} 
Let us split $S_C$ as 
\begin{align*}
S_C=&\sum_{\substack{J\in R_C\\\sqrt{K(\tilde{J})}>\theta_{\gamma} n}}\frac{\exp(-2\kappa(J)t)}{\kappa(J)}\left\vert \Phi[\mu](J)-\Phi[\nu](J)\right\vert^2\\
&\hspace{5cm}+\sum_{\substack{J\in R_C\\\sqrt{K(\tilde{J})}\leq \theta_{\gamma} n}}\frac{\exp(-2\kappa(J)t)}{\kappa(J)}\left\vert \Phi[\mu](J)-\Phi[\nu](J)\right\vert^2\\
=&S_C^1+S_C^2.
\end{align*}
By definition of the Fourier transform on $\mathcal{M}_1(T_k)$ from \eqref{eq:Fourier_probability} and the bound $S_{J}(\vec{u})\leq S_{J}(0,\ldots,0)$ for $\vec{u}\in T_k$, for any $J\in B_{k}$ we have 
$$\left\vert \Phi[\mu](J)-\Phi[\nu](J)\right\vert^2\leq 2S_J(0,\ldots,0)^2.$$
Hence, by the fact that $\kappa(J)=K(\tilde{J})+\vert \lambda_J\vert^2/k$ for $J\in B_{k}$ and \eqref{eq:series_representation_Brownian}, we get the first bound
\begin{align*}
S_C^1\leq &4\sum_{\substack{J\in B_{k}\setminus\{I_0\}\\\sqrt{K(\tilde{J})}>\theta_{\gamma} n}}\frac{\exp(-2\kappa(J)t)}{\kappa(J)}S_J(0,\ldots,0)^2\\
\leq& 4\exp(-t\theta_{\gamma}^2n^2)\sum_{\substack{J\in B_{k}\setminus\{I_0\}\\K(\tilde{J})>\theta_{\gamma} n}}\exp(-\kappa(J)t)S_J(0,\ldots,0)^2\leq 4\exp\left(-t\theta_{\gamma}^2 n^2\right)K_{t}^{U(k)}(0,0).
\end{align*}
By Lemma \ref{lem:rough_bound_heat} for $t_1=t$ and $t_2=t+1$, when $t\leq 1$ we have
$$K_{t}^{U(k)}(0,0)\leq \left(\frac{t+1}{t}\right)^{k^2/2}\exp\left(\frac{k\pi^2}{4}\right)\leq \left(\frac{2}{t}\right)^{k^2/2}\exp\left(\frac{\pi^2}{4k}\right)^{k^2},$$
so that 
\begin{equation}\label{eq:bound_S_C_1}
S_C^1\leq 4\exp(-t\theta_{\gamma}^2 n^2-k^2/2(\log t-c))
\end{equation}
for some numeric constant $c\in \mathbb{R}$. 

By Proposition \ref{prop:convergence_fourier}, when $K(\tilde{J})\leq \theta_{\gamma}^2 n^2$ with $c_1$ small enough, we have 
\begin{align*}
\left\vert\Phi_{n}[*\mathbf{m}](J^{(n)})\right\vert \leq& \exp\left(-\frac{m(2\pi)^2}{2n^2}\frac{K(\tilde{J})K(\mathbf{m})}{k^2-1}+\frac{mM}{n^3}O\left(\Big[kK(\tilde{J})K(\mathbf{m})\Big]^{3/2}\right)\right)\\
\leq& \exp\left(-\frac{1}{2}\cdot\frac{m(2\pi)^2}{2n^2}\frac{K(\tilde{J})K(\mathbf{m})}{k^2-1}\right)=\exp\left(-\frac{1}{2}\cdot t_0\frac{\gamma}{k} K(\tilde{J})\right).
\end{align*}
Hence, by \eqref{eq:formula_Phi_mu} and \eqref{eq:formula_Phi_nu},
$$S_C^2\leq 4\sum_{\substack{J\in B_{k}\setminus\{I_0\}\\\sqrt{K(\tilde{J})}\leq \theta_{\gamma} n,\vert\langle \lambda_J\rangle\vert> \theta_{\alpha}n}}\frac{\exp(-2\kappa(J)t)}{\kappa(J)}\exp\left(-\frac{1}{2}\cdot t_0\frac{\gamma}{k} K(\tilde{J})\right)S_J(0,\ldots,0)^2.$$
Since, $\exp(-2\kappa(J)t)\leq \exp(-2t\langle J\rangle^2/k)$, decomposing $J\in B_{K}$ as $J=\hat{J}+\ell \mathbf{1}_k$ with $\ell\in\mathbb{Z}$ and summing on $\ell\in \mathbb{Z}$ such that $\vert\langle \lambda_J+kl\rangle\vert> \theta_{\alpha}n$ yields
\begin{align*}
S_C^2\leq & 4\sum_{\substack{\ell\in \mathbb{Z},J\in B_{k}\setminus\{I_0\}, J_{k}=0\\\sqrt{K(\tilde{J})}\leq \theta_{\gamma} n,\vert\langle \lambda_J+kl\rangle\vert> \theta_{\alpha}n}}\frac{k\exp(-2t\theta_{\alpha}^2n^2/k)}{\langle J+\ell k\rangle^2}\exp\left(-t_0\frac{\gamma}{k} K(\tilde{J})\right)S_J(0,\ldots,0)^2\\
\leq &c\frac{k}{n\theta_{\alpha}}\exp(-2t\theta_{\alpha}^2n^2/k)\sum_{J\in B_{k}\setminus\{I_0\}, J_{k}=0}\exp\left(- t_0\frac{\gamma}{k} K(\tilde{J})\right)S_J(0,\ldots,0)^2
\end{align*}
for some numeric constant $c>0$. By \eqref{eq:series_representation_Brownian_suk}, we have $K_{t_0}^{SU(k)}(0,0)$, so that when $t\leq t_0$,
\begin{equation}\label{eq:bound_S_C_2}
S_C^2\leq c\frac{k\exp(-2t\theta_{\alpha}^2n^2/k)}{n\theta_{\alpha}}K_{t}^{SU(k)}(0,0).
\end{equation}
When $\theta_{\gamma}\geq n^{-1}$, set
\begin{equation}\label{eq:choice_t}
t=\frac{2+k^2(\log n+\log \theta_{\gamma}+c_0)}{\theta_{\gamma}^2n^2}+\frac{k\log n}{\theta_{\alpha}^2n^2},
\end{equation}
with $c_0$ a numeric constant. Then, combining \eqref{eq:bound_S_C_1} and \eqref{eq:bound_S_C_2} yields for $c_0$ large enough
\begin{equation}\label{eq:bound_S_C}
S_C\leq \frac{c}{n^{2}}+\frac{ck}{\theta_{\alpha}^3n^3}.
\end{equation}

\textbf{Bound of $S_B$:}
The pattern to bound $S_B$ is similar to the one concerning $S_C$. Split $S_B$ as 
\begin{align*}
S_B=&\sum_{\substack{J\in R_B\\\sqrt{K(\tilde{J})}>\nu_{\gamma} n}}\frac{\exp(-2\kappa(J)t)}{\kappa(J)}\left\vert \Phi[\mu](J)-\Phi[\nu](J)\right\vert^2\\
&\hspace{5cm}+\sum_{\substack{J\in R_B\\\sqrt{K(\tilde{J})}\leq \nu_{\gamma}}}\frac{\exp(-2\kappa(J)t)}{\kappa(J)}\left\vert \Phi[\mu](J)-\Phi[\nu](J)\right\vert^2\\
=&S_B^1+S_B^2.
\end{align*}
The condition $\frac{\sqrt{K(\tilde{J})}}{\theta_{\gamma}}\vee\frac{\vert\langle \lambda_J\rangle\vert}{ \theta_{\alpha}}\leq n$ for $c$ small enough yields by Proposition \ref{prop:convergence_fourier}
\begin{align*}
\left\vert\Phi_{n}[*\mathbf{m}](J^{(n)})\right\vert \leq& \exp\Bigg(-\frac{m(2\pi)^2}{2n^2}\left(\frac{K(\tilde{J})K(\mathbf{m})}{k^2-1}+\frac{Var(\mathbf{m})\langle \lambda_J\rangle^2}{k^2}\right)\\
&\hspace{4cm}+\frac{mM}{n^3}O\left(\Big[kK(\tilde{J})K(\mathbf{m})\Big]^{3/2}+\frac{Var(\mathbf{m})^{3/2}\cdot\vert\langle \lambda_J\rangle\vert^3}{k^3}\right)\Bigg)\\
\leq &\exp\left(-\frac{m(2\pi)^2}{4n^2}\left(\frac{K(\tilde{J})K(\mathbf{m})}{k^2-1}+\frac{Var(\mathbf{m})\langle \lambda_J\rangle^2}{k^2}\right)\right)\\
\leq&\exp\left(-\frac{t_0}{2}\left(\frac{\gamma K(\tilde{J})}{k}+\frac{\alpha\langle \lambda_J\rangle^2}{k^2}\right)\right).
\end{align*}
Remark that $\nu_{\gamma}=\frac{2c}{km^{1/3}M^{1/3}\sqrt{\gamma}}=\frac{c'}{kn^{2/3}t_0^{1/3}M^{1/3}\sqrt{\gamma}}$ for some numeric constant $c'$. Hence, $\sqrt{K(\tilde{J})}\geq \nu_{\gamma}n$ implies that 
$$\sqrt{K(\tilde{J})}\geq \frac{c'n^{1/3}}{kt_0^{1/3}M^{1/3}\sqrt{\gamma}}.$$
Similarly, $\vert\langle \lambda_J\rangle\vert\geq \nu_{\alpha}$ implies that 
$$\vert\langle \lambda_{J}\rangle\vert^2 \geq \frac{ckn^{1/3}}{(t_0M)^{1/3}\sqrt{\alpha}}.$$
Decomposing $J\in B_{K}$ as $J=\hat{J}+\ell \mathbf{1}_k$ with $\ell\in\mathbb{Z}$ and summing on $\ell\in \mathbb{Z}$ yields then
\begin{align*}
S_B^1\leq & 4\sum_{\substack{\ell\in\mathbb{Z},J\in B_{k}\setminus\{I_0\}, J_{k}=0\\\sqrt{K(\tilde{J})}\geq \nu_{\gamma} n}}\frac{1}{\kappa(J)}\exp\left(-t_0\frac{\gamma K(\tilde{J})}{k}\right)\left\vert S_{J}(\xi_n(I))\right\vert^2\\
\leq &\sum_{\substack{\ell\in\mathbb{Z},J\in B_{k}\setminus\{I_0\}, J_{k}=0\\\sqrt{K(\tilde{J})}\geq \nu_{\gamma} n}}\frac{\exp\left(-c\frac{cn^{2/3}t_0^{1/3}}{k^3M^{2/3}} \right)}{\frac{n^{2/3}}{k^2(t_0M)^{2/3}\gamma}+\vert\langle J+\ell k\rangle\vert^2/k}\exp\left(-t_0/2\frac{\gamma K(\tilde{J})}{k}\right)\left\vert S_{J}(\xi_n(I))\right\vert^2\\
\leq &\frac{k^2(t_0M)^{1/3}\sqrt{\gamma}\exp\left(-c\frac{cn^{2/3}t_0^{1/3}}{k^3M^{2/3}} \right)}{n^{1/3}}K_{t_0/2}^{SU(k)}(\xi_n(I),\xi_n(I)).
\end{align*}
for some numeric constant $c>0$. Similarly,
\begin{align*}
S_B^2\leq & 4\sum_{\substack{\ell\in\mathbb{Z},J\in B_{k}\setminus\{I_0\}, J_{k}=0\\\vert\langle J+\ell\mathbf{1}_k\rangle\vert \geq \nu_{\alpha}n}}\frac{1}{\kappa(J)}\exp\left(-t_0\left(\frac{\gamma K(\tilde{J})}{k}+\frac{\alpha\langle \lambda_J\rangle\vert^2}{k^2}\right)\right)\left\vert S_{J}(\xi_n(I))\right\vert^2\\
\leq &\sum_{\substack{\ell\in\mathbb{Z},J\in B_{k}\setminus\{I_0\}, J_{k}=0\\\vert\langle J+\ell\mathbf{1}_k\rangle\vert \geq \nu_{\alpha}n}}\frac{\exp\left(-c\frac{cn^{2/3}t_0^{1/3}}{M^{2/3}} \right)}{\vert\langle J+\ell k\rangle\vert^2/k}\exp\left(-t_0\frac{\gamma K(\tilde{J})}{k}\right)\left\vert S_{J}(\xi_n(I))\right\vert^2\\
\leq &\frac{k^2(t_0M)^{1/3}\sqrt{\alpha}\exp\left(-c\frac{cn^{2/3}t_0^{1/3}}{M^{2/3}} \right)}{n^{1/3}}K_{t_0}^{SU(k)}(\xi_n(I),\xi_n(I)).
\end{align*}
Hence, we finally have 
$$S_{B}=S_{B}^1+S_{B}^2\leq \frac{ck^2(t_0M)^{1/3}(\sqrt{\gamma}+\sqrt{\alpha})\exp\left(-c\frac{cn^{2/3}t_0^{1/3}}{k^3M^{2/3}} \right)}{n^{1/3}}K_{t_0/2}^{SU(k)}(\xi_n(I),\xi_n(I))$$
for some numeric constant $c>0$, where we used that the heat kernel $K_{t}^{SU(k)}(0,0)$ is decreasing in $t$.

\textbf{Bound of $S_{A}$ :}
By Lemma \ref{lem:bound_K_norm_inf}, $\sqrt{K(\tilde{J})}\leq \nu_{\gamma}n$ implies that $\Vert \tilde{J}\Vert_{\infty}\leq \frac{cn}{m^{1/3}M^{1/3}K(\mathbf{m})^{1/3}}$. Hence, when $\frac{\sqrt{K(\tilde{J})}}{\nu_{\gamma}}\vee\frac{\vert\langle \lambda_J\rangle\vert}{ \nu_{\alpha}}\leq n$, Proposition \ref{prop:convergence_fourier} yields that
\begin{align*}
e^{-2i \pi \frac{m\langle\mathbf{m}\rangle\cdot\langle \lambda_J\rangle}{kn}}\Phi_{n}[*\mathbf{m}](J^{(n)})
=&\exp\Bigg[-\frac{m(2\pi)^2}{2n^2}\left[\frac{K(\mathbf{m})K(\tilde{J})}{k^2-1}+\frac{Var(\mathbf{m})\langle \lambda_J\rangle^2}{k^2}\right]\\
&+\frac{mM}{n^3}O\left(\frac{Var(\mathbf{m})^{3/2}\cdot\vert\langle \lambda_J\rangle\vert^3}{k^3}+\Big[k \cdot K(\mathbf{m})\cdot K(\tilde{J})\Big]^{3/2}\right)\Bigg].
\end{align*}
By \eqref{eq:formula_Phi_mu} and \eqref{eq:formula_Phi_nu}, we thus have 
\begin{align*}
\left\vert \Phi[\mu](J)-\Phi[\nu](J)\right\vert=&\exp\left(-\frac{m(2\pi)^2}{2n^2}\left[\frac{K(\mathbf{m})K(\tilde{J})}{k^2-1}+\frac{Var(\mathbf{m})\langle \lambda_J\rangle^2}{k^2}\right]\right)\left\vert S_{J}(\xi_n(I))\right\vert\\
&\hspace{1cm}\cdot\left\vert \exp\left(\frac{mM}{n^3}O\left(\frac{Var(\mathbf{m})^{3/2}\cdot\vert\langle \lambda_J\rangle\vert^3}{k^3}+\Big[k \cdot K(\mathbf{m})\cdot K(\tilde{J})\Big]^{3/2}\right)\right)-1\right\vert\\
\leq& \frac{t_0M}{n}\exp\left(-t_0\left(\frac{\gamma}{k}K(\tilde{J})+\frac{\alpha}{k^2}\vert\langle\lambda_J\rangle\vert^2\right)\right)\left\vert S_{J}(\xi_n(I))\right\vert\\
&\hspace{3cm} O\left(\frac{\alpha^{3/2}\cdot\vert\langle \lambda_J\rangle\vert^3}{k^3}+\Big[k^2 \cdot \gamma\cdot K(\tilde{J})\Big]^{3/2}\right),
\end{align*}
where we used on the last inequality that $\frac{mM}{n^3}O\left(\frac{Var(\mathbf{m})^{3/2}\cdot\vert\langle \lambda_J\rangle\vert^3}{k^3}+\Big[k \cdot K(\mathbf{m})\cdot K(\tilde{J})\Big]^{3/2}\right)=O(1)$ when $\frac{\sqrt{K(\tilde{J})}}{\nu_{\gamma}}\vee\frac{\vert\langle \lambda_J\rangle\vert}{ \nu_{\alpha}}\leq n$ and the bound $\vert e^u-1\vert=O(u)$ when $u=O(1)$. Hence,
\begin{align*}
S_A\leq \frac{t_0^2M^2}{n^2}\sum_{J\in R_{A}}\frac{e^{-2t_0\left(\frac{\gamma}{k}K(\tilde{J})+\frac{\alpha}{k^2}\vert\langle\lambda_J\rangle\vert^2\right)}}{\kappa(J)}\left\vert S_{J}(\xi_n(I))\right\vert^2\cdot O\left(\frac{\alpha^{3}\cdot\vert\langle \lambda_J\rangle\vert^6}{k^3}+\Big[k^2 \cdot \gamma\cdot K(\tilde{J})\Big]^{3}\right).
\end{align*}
Setting $g(u)=u^3e^{-u}$ for $u\geq 0$, we have 
$$\exp\left(-t_0\left(\frac{\gamma}{k}K(\tilde{J})+\frac{\alpha}{k^2}\vert\langle\lambda_J\rangle\vert^2\right)\right)O\left(\frac{\alpha^{3}\cdot\vert\langle \lambda_J\rangle\vert^6}{k^3}+\Big[k^2 \cdot \gamma\cdot K(\tilde{J})\Big]^{3}\right)\leq \frac{ck^9\Vert g\Vert_{\infty}}{t_0^3},$$
so that 
$$S_A\leq \frac{ck^9M^2}{t_0n^2}\sum_{J\in R_{A}}\frac{\exp\left(-t_0\frac{\gamma}{k}K(\tilde{J})\right)}{\kappa(J)}\left\vert S_{J}(\xi_n(I))\right\vert^2.$$
Decomposing $J\in B_{K}$ as $J=\hat{J}+\ell \mathbf{1}_k$ and using that $\sum_{\ell\in \mathbb{Z}}\frac{1}{\kappa(J+\ell\mathbf{1}k)}=\sum_{\ell\in \mathbb{Z}}\frac{1}{\Vert J+\ell\mathbf{1}k-\tilde{I_0}\Vert^2-\Vert \tilde{I}_0\Vert^2}\leq c$ for some numeric constant $c$ independent of $J\not=I_0$, we get 
$$S_A\leq \frac{ck^9M^2}{t_0n^2}\sum_{J\in B_k,J_k=0}\exp\left(-t_0\frac{\gamma}{k}K(\tilde{J})\right)\left\vert S_{J}(\xi_n(I))\right\vert^2\leq\frac{ck^9M^2}{t_0n^2}K_{t_0}^{SU(k)}(\xi_n(I),\xi_n(I)).$$
Putting all previous bounds together yields with \eqref{eq:proof_thm_1_recall_Wasser}
\begin{align}
W_{1}^{T_k}(\mu,\nu)&\leq ck\sqrt{t}+\frac{c}{n}+\frac{c\sqrt{k}}{(\theta_{\alpha}n)^{3/2}}\nonumber\\
&+\left(\frac{ck^2(t_0M)^{1/3}(\sqrt{\gamma}+\sqrt{\alpha})\exp\left(-c\frac{cn^{2/3}t_0^{1/3}}{k^3M^{2/3}} \right)}{n^{1/3}}+\frac{ck^9M^2}{t_0n^2}\right)^{1/2}K_{t_0/2}^{SU(k)}(\xi_n(I),\xi_n(I))^{1/2}.\label{eq:exact_bound}
\end{align}
Using the definition of $\theta_{\alpha}$ and the choice of $t$ from \eqref{eq:choice_t} yields for $n$ large enough, after simplifying,
$$W_{1}^{T_k}(\mu,\nu)\leq \frac{M}{n}\left(c_1(k^8\gamma+k\alpha)\log n+\frac{c_2k^{9/2}}{\sqrt{t_0}}K_{t_0/2}^{SU(k)}(\xi_n(I),\xi_n(I))^{1/2}\right),$$
for some numeric constants $c_1,c_2>0$.
\end{proof}

\section{Local limit theorem and its geometric consequences}\label{sec:LLT}
In this section, we prove a local limit theorem for the $h$-probability measure $(\ast \mathbf{m})*\delta_I$ and deduce Corollary \ref{cor:cohomology}, which gives asymptotics of structure coefficients in the ring $QH(G_{k,n})$. For simplicity, we assume that $Var(\mathbf{m})=0$ to avoid any lattice issue. 
\begin{theorem}
Let $M>0$ and suppose that $\mathbf{m}=(\mu_r)_{1\leq r\leq m}$ is a sequence of $h$-probability measures on $ B_{k,n}$ such that $Var_{2}(\mathbf{m})=0$ and $\langle\widehat{\mathbf{m}}\rangle_3\leq MK(\mathbf{m})$ and let $I\in B_{k,n}$. Then, 
$$\left(\left(\ast \mathbf{m}\right)\ast \delta_I\right)[I']=\delta_{\vert I'\vert=\langle I\rangle+m\langle\mathbf{m}\rangle)[n]}n\left[K_{t_0\gamma}^{SU(k)}(\xi_n(I),\xi_{n}(I'))+o\left(\frac{1}{\sqrt{n}}\right)\right],$$
where 
$$t_0=\frac{2\pi m}{n^2},\quad \gamma=\frac{k}{2(k^2-1)}K(\mathbf{m}),$$
and $o(\cdot)$ only depends on $\gamma$ ,$M$, $t_0$ and $k$. 
\end{theorem}
\begin{proof}
Recall that $\Phi_n[\delta_I](J)=\frac{S_{I}(\xi_{n}(J))}{S_{I}(\xi_{n}(I_0)}$ by \eqref{eq:Fourier_probability}. Then, by Proposition \ref{prop:Verlinde_formula_general}, Lemma \ref{lem:Perron-Frobenius} and Lemma \ref{lem:formula_inversion_Schur}, for $I'\in B_{k,n}$
\begin{align*}\left(\left(\ast \mathbf{m}\right)\ast \delta_I\right)[I']=&\sum_{J\in B_{k,n}}\Phi[\delta_I](J)\prod_{j=1}^m\Phi[\mu_j](J)\tilde{u}^J_{I'}\\
=&\sum_{J\in B_{k,n}}\prod_{j=1}^m\Phi[\mu_j](J)\frac{S_{I}(\xi_{n}(J))}{S_{I}(\xi_{n}(I_0))}\frac{\left\vert V(\xi_{n}(J))\right\vert^2\overline{S_{I'}(\xi_{n}(J))}}{\left\vert V(\xi_{n}(I_0))\right\vert^2S_{I'}(\xi_{n}(I_0))}\\
=&\sum_{J\in B_{k,n}}\prod_{j=1}^m\Phi[\mu_j](J)S_{J}(\xi_{n}(I))\overline{S_{J}(\xi_{n}(I')}).
\end{align*}
Set $t_0=\frac{(2\pi)^2m}{n^2}$ and $\gamma=\frac{kK(\mathbf{m})}{2(k^2-1)}$. Then, applying Proposition \ref{prop:convergence_fourier} with $Var_2(\mathbf{m})=0$ gives
\begin{align*}
\left(\left(\ast \mathbf{m}\right)\ast \delta_I\right)[I']=&\sum_{J\in B_{k,n}}\overline{S_{J}(\xi_{n}(I'))}S_{J}(\xi_{n}(I))e^{\frac{2i m\langle\mathbf{m}\rangle\langle J\rangle}{kn}}B(J)
\end{align*}
with 
$$B(J)\leq\exp\left(-t_0\frac{\gamma K(\tilde{J})}{k}+\frac{t_0M}{n}O\left(\left[kK(\tilde{J})K(\mathbf{m})\right]^{3/2}\right)\right)$$
and, when $\Vert \tilde{J}\Vert_{\infty}\leq c\frac{n}{m^{1/3}M^{1/3}K(\mathbf{m})^{1/2}}:=\theta_1 n^{1/3}$ (with $\theta_2$ depending on $t_0,\gamma,k$ and $M$)
$$B(J)=\exp\left(-t_0\frac{\gamma K(\tilde{J})}{k}+\frac{t_0M}{n}O\left(\left[kK(\tilde{J})K(\mathbf{m})\right]^{3/2}\right)\right).$$

Noting that $K(\tilde{J})\leq\Vert \tilde{J}\Vert_{\infty}^2=\Vert \hat{J}\Vert_{\infty}^2$, there is $\theta_2>0$ depending on $k,\gamma$ and $M$ so that 
$$-t_0\frac{\gamma K(\tilde{J})}{k}+\frac{t_0M}{n}O\left(\left[kK(\tilde{J})K(\mathbf{m})\right]^{3/2}\right)\leq -\frac{t_0}{2}\frac{\gamma K(\tilde{J})}{k}+\frac{t_0M}{n}O\left(\left[kK(\tilde{J})K(\mathbf{m})\right]^{3/2}\right)$$ 
when $\sup(\hat{J}_{i}-\hat{J}_{i+1} [n])\leq \theta_2n$. By Lemma \ref{lem:bound_Fourier_large}, for $\sup(\hat{J}_{i}-\hat{J}_{i+1} [n])\geq \theta_2n$, there exists $\delta,\eta>0$ depending on $t_0,k,\gamma$ and $M$ such that
\begin{equation}\label{eq:bound_Fourier_large_J}
B(J)\leq   \exp\left(-n^2\left(\delta+ \frac{\eta}{n^3}\right)\right).
\end{equation}
Let us split $\left(\left(\ast \mathbf{m}\right)\ast \delta_I\right)[I']$ as
\begin{align*}
\left(\left(\ast \mathbf{m}\right)\ast \delta_I\right)[I']=&\sum_{\substack{J\in B_{k,n}\\ \sup(\hat{J}_{i}-\hat{J}_{i+1} [n])> \theta_1 n^{1/3}}}S_{J}(\xi_{n}(I))\overline{S_{J}(\xi_{n}(I'))}e^{2i \pi \frac{m\langle\mathbf{m}\rangle\cdot\langle \lambda_J\rangle}{kn}}B(J)\\
+&\sum_{\substack{J\in B_{k,n}\\ \sup(\hat{J}_{i}-\hat{J}_{i+1} [n])< \theta_1 n^{1/3}}}S_{J}(\xi_{n}(I))\overline{S_{J}(\xi_{n}(I'))}e^{2i \pi \frac{m\langle\mathbf{m}\rangle\cdot\langle \lambda_J\rangle}{kn}}B(J):=S_1+S_2,
\end{align*}
and, using \eqref{eq:series_representation_Brownian_suk}, split $K^{SU(k)}_{\gamma t_0}\left(\xi_n(I'),R_{-\frac{m\langle\mathbf{m}\rangle}{kn}}\left(\xi_{n}\left(I\right)\right)\right)$ as 
\begin{align*}
K^{SU(k)}_{\gamma t_0}&\left(\xi_n(I'),R_{-\frac{m\langle\mathbf{m}\rangle}{kn}}\left(\xi_{n}\left(I\right)\right)\right)\\
=&\sum_{\substack{J\in  B_{k}\\J_k=0}}\exp(-\kappa_{0,\gamma}(J)t_0)\overline{S_{J}(\xi_{n}(I'))}S_{J}\left(R_{-\frac{m\langle\mathbf{m}\rangle}{kn}}\left(\xi_{n}\left(I\right)\right)\right)\\
=&\sum_{\substack{J\in  B_{k}\\J_k=0\\ \sup(J_{i}-J_{i+1} [n])<\theta_1 n^{1/3}}}e^{\frac{2i \pi m\langle\mathbf{m}\rangle)\langle J\rangle}{kn}-\kappa_{0,\gamma}(J)t_0}\overline{S_{J}(\xi_{n}(I'))}S_{J}(\xi_{n}(I))\\
&\hspace{3cm}+\sum_{\substack{J\in  B_{k}\\J_k=0\\ \sup(J_{i}-J_{i+1} [n])\geq \theta_1n^{1/3}}}e^{-\kappa_{0,\gamma}(J)t_0}\overline{S_{J}(\xi_{n}(I'))}S_{J}\left(R_{-\frac{m\langle\mathbf{m}\rangle}{kn}}\left(\xi_{n}\left(I\right)\right)\right)\\
:=&S'_1+S'_2.
\end{align*}
Then, we have 
\begin{align*}\left\vert \left(\left(\ast \mathbf{m}\right)\ast \delta_I\right)[I']- n\delta_{\langle I'\rangle=\langle I\rangle+m\langle\mathbf{m}\rangle)\rangle[n]}K^{SU(k),\gamma}_{t_0,\delta_I}(\xi_{n}(I'))\right\vert\leq \left\vert S_1-\delta_{\langle I'\rangle=\langle I\rangle+m\langle\mathbf{m}\rangle)[n]}S_1'\right\vert+\left\vert S_2\vert+\vert S_2'\right\vert.
\end{align*}
First, by Lemma \ref{lem:bound_K_norm_inf}, \eqref{eq:bound_Fourier_large_J} and the rough bound $S_{J}(\xi_{n}(J'))\leq d_J\leq n^{k(k-1)/2}$ for $J,J'\in B_{k,n}$,
\begin{align*}
\left\vert S_2\right\vert\leq& \sum_{\substack{J\in B_{k,n}\\ \theta_1 n^{1/3}<\sup(\hat{J}_{i}-\hat{J}_{i+1} [n])< \theta_2n}}d_{J}^2\exp\left(-\frac{t_0\gamma}{2k}K(\tilde{J})\right)\\
&\hspace{4cm}+ \sum_{\substack{J\in B_{k,n}\\ \sup(\hat{J}_{i}-\hat{J}_{i+1} [n])\geq \theta_2n}}d_{J}^2\exp\left(-n^2\left(\delta+ \frac{\eta}{n^3}\right)\right)\\
\leq &\#B_{k,n}n^{k(k-1)}\left[\exp\left(-\frac{t_0\gamma \theta_1^2}{k^2}n^{2/3}\right)+\exp\left(-n^2\left(\delta+ \frac{\eta}{n^3}\right)\right)\right]\\
\leq & n^{k^2}\exp\left(-Cn^{1/3}\right)=o\left(1\right)
\end{align*}
with $o\left(1\right)$ only depending on $\gamma, k, t_0$ and $M$. Likewise, by the absolute convergence of the Fourier expansion of the heat kernel on $SU(k)$ for $t\geq t_0$, see \cite{Far}, 
$$\left\vert S'_2\right\vert=o(1).$$

Next, using that $S_{J}(\xi_{n}(I))=S_{\hat{J}}(\xi_{n}(I))e^{\frac{2i\pi J_k\langle I\rangle}{n}}$ and summing on $J_k$ when $\sup(\hat{J}_{i}-\hat{J}_{i+1} [n])<\theta_1 n^{1/3}$ and for $n$ large enough (so that $\hat{J}+l\mathbf{1}_k\not=\hat{J}+l'\mathbf{1}_k$ when $0\leq l\not=l'\leq n-1$) yields
\begin{align*}
S_1=&\sum_{\substack{J\in B_{k,n}\\J_k=0\\ \sup(J_{i}-J_{i+1} [n])< \theta_1n^{1/3}}}B(J)S_{\hat{J}}(\xi_{n}(I))\overline{S_{\hat{J}}(\xi_{n}(I'))}\left(\sum_{l=0}^{n-1}e^{\frac{2i \pi (m\langle\mathbf{m}\rangle)/k(\langle J\rangle+kl)+(\langle I\rangle-\langle I'\rangle)l}{n}}\right)\\ 
=&n\delta_{m\langle\mathbf{m}\rangle)+\langle I\rangle=\langle I'\rangle[n]}\sum_{\substack{J\in B_{k,n}\\J_k=0\\ \sup(J_{i}-J_{i+1} [n])< \theta_1 n^{1/3}}}B(J)e^{\frac{2i \pi m\langle\mathbf{m}\rangle)\langle J\rangle}{kn}}S_{J}(\xi_{n}(I))\overline{S_{J}(\xi_{n}(I'))}.
\end{align*}
with
$$B(J)=\exp\left(-t_0\frac{\gamma K(\tilde{J})}{k}+\frac{t_0M}{n}O\left(\left[kK(\tilde{J})K(\mathbf{m})\right]^{3/2}\right)\right).$$
Hence, using that $\frac{t_0M}{n}O\left(\left[kK(\tilde{J})K(\mathbf{m})\right]^{3/2}\right)=O(1)$ when $\sup(J_{i}-J_{i+1} [n])< \theta_1 n^{1/3}$,
\begin{align*}
&\left\vert S_1-\delta_{\langle I'\rangle=\langle I\rangle+m\langle \mathbf{m}\rangle[n]}S_1'\right\vert\\
\leq& n \sum_{\substack{J\in B_{k,n}\\J_k=0\\ \sup(J_{i}-J_{i+1} [n])< \theta_1n^{1/3}}}\left\vert S_{J}\left(R_{-\frac{m\langle\mathbf{m}\rangle}{kn}}\left(\xi_{n}\left(I\right)\right)\right)\overline{S_{J}(\xi_{n}(I'))}\right\vert \exp\left(-t_0\frac{\gamma K(\tilde{J})}{2k}\right)\\
&\hspace{3cm}\exp\left(-t_0\frac{\gamma K(\tilde{J})}{k}\right)\left\vert \exp\left(\frac{t_0M}{n}O\left(\left[k^2\gamma K(\tilde{J})\right]^{3/2}\right)\right)-1\right\vert\\
\leq &\sum_{\substack{J\in B_{k,n}\\J_k=-(k-1)/2\\ \sup(J_{i}-J_{i+1} [n])< n^{1/6}}}d_J^2\exp\left(-t_0\frac{\gamma K(\tilde{J})}{2k}\right)O(1)\leq K^{SU(k)}_{\gamma t_0/2}(0,0)O(1)= O\left(1\right),
\end{align*}
with $O(1)$ depending on $M,t_0,\gamma$ and $k$. Finally, 
$$\left\vert \left(\left(\ast \mathbf{m}\right)\ast \delta_I\right)[I']- n\delta_{\langle I'\rangle=\langle I\rangle+m\langle\mathbf{m}\rangle[n]}K^{SU(k)}_{\gamma t_0}\left(\xi_{n}(I'),R_{-\frac{m\langle\mathbf{m}\rangle}{kn}}\left(\xi_{n}\left(I\right)\right)\right)\right\vert=O(1),$$
with $O(\cdot)$ depending on $k, \gamma, M$ and $t_0$.
\end{proof}
We can now turn to the proof of Corollary \ref{cor:cohomology}.
\begin{proof}[Proof of Corollary \ref{cor:cohomology}]
Recall that in the case where $\sum_{i=0}^{m+1} d_i\leq \dim Hol_d(G_{k,n})=k(n-k)+dn$ and for generic manifolds $M_0,\ldots, M_{m+1}$, we have by \eqref{eq:qring_combinatorics}
$$\#\mathcal{M}^{d,a_0,\ldots,a_{m+1}}_{M_0,\ldots,M_{m+1}}=\left\langle \left(\prod_{i=0}^{m} [M_i]\right)[q^d],[M_{m+1}]\right\rangle_{H^*(G_{k,n})},$$
where $\langle \cdot,\cdot\rangle$ denotes the Poincaré duality in $H^*(G_{k,n})$ and for $x\in QH(G_{k,n})=\mathbb{C}[q]\otimes H^*(G_{k,n})$, we write $x=\sum_{d\geq 0} x[q^d]\otimes q^d$. By the multiplication rule \eqref{eq:multiplication_quantum}, for $d'\geq 0$ we have $\left(\prod_{i=0}^{m} [M_i]\right)[q^{d'}]\in H^{d'n+k(n-k)-\sum_{i=0}^md_i}$ and thus, for $d'\not=d''$, $\left(\prod_{i=0}^{m} [M_i]\right)[q^{d'}]$ and $\left(\prod_{i=0}^{m} [M_i]\right)[q^{d''}]$ are in different dimension components of $H^*(G_{k,n})$. Hence, we deduce that 
\begin{align*}
\#\mathcal{M}^{d,a_0,\ldots,a_{m+1}}_{M_0,\ldots,M_{m+1}}=&\sum_{I\in B_{k,n},\langle I\rangle=d_{m+1}}a_{I}^{m+1}\left\langle \left(\prod_{i=0}^{m} [M_i]\right)[q^d],\sigma_{\lambda_I}\right\rangle_{H^*(G_{k,n})}\\
=&\sum_{I\in B_{k,n},\langle I\rangle=d_{m+1}}a_{I}^{m+1}\left\langle \overline{\left(\prod_{i=0}^{m} [M_i]\right)},\sigma_{\lambda_I}\right\rangle_{H^*(G_{k,n})},
\end{align*}
where $\overline{x}$ is the image of $x$ in $QH(G_{k,n})/\langle q=1\rangle$ and we recall that the $a_{I}^{m+1}$'s are the coefficients appearing in the decomposition $[M_{m+1}]=\sum_{I\in B_{k,n},\langle I\rangle=d_{m+1}}a_{I}^{m+1}\sigma_{\lambda_{I}}$. Then, recalling that the Poincaré duality satisfies $\langle \sigma_{\lambda},\sigma_{\mu}\rangle_{H^*(G_{k,n})}=\delta_{\lambda,\mu^c}$ (see for example the introduction of \cite{B}), we have 
$$\#\mathcal{M}^{d,a_0,\ldots,a_{m+1}}_{M_0,\ldots,M_{m+1}}=\sum_{I\in B_{k,n},\langle I\rangle=d_{m+1}}a_{I}^{m+1}\overline{c}_{\prod_{i=0}^{m} [M_i]}^{I^c},$$
where $\overline{c}_{\prod_{i=0}^{m} [M_i]}^{I}, I\in B_{k,n}$ are the coefficients of $\overline{\prod_{i=0}^{m} [M_i]}$ on the basis $\{\sigma_{\lambda_I}, I\in B_{k,n}\}$ and we recall that $I^c=(n-1-I_{k-j+1})_{1\leq j\leq k}$. By Proposition \ref{prop:equiv_graph_qh}, the linear map from $QH(G_{k,n})/\langle q=1\rangle$ to $\mathcal{B}(L^2(B_{k,n}))$ mapping  $\sigma_{\lambda}$ to $T(s_{\lambda})$ is an algebra isomorphism from $QH(G_{k,n})/\langle q=1\rangle$ to the algebra $\mathcal{A}\subset \mathcal{B}(L^2(B_{k,n}))$ for which $B_{k,n}$ is positively multiplicative, and $T(s_{\lambda})$ is the unique element of $\mathcal{A}$ sending $I_0$ to $I_{\lambda}$. Hence, by Proposition \ref{prop:structure_conv} and \eqref{eq:isomor_L2h_CA} we have 
$$\overline{c}_{\prod_{i=0}^{m} [M_i]}^{I}=h^r_{I}\left( (M^h)^{-1}([M_0])*\dots *(M^h)^{-1}([M_{m}])\right)(I),$$
where $(M^h)^{-1}([M_i])=\sum_{I\in B_{k,n}}\frac{1}{h^r_I}a_{I}^ie_I$ and $h^r_I=\frac{V(\xi_{n}(I_0)^2}{n^k}S_{I}(\xi_{n}(I_0))$ by Lemma \ref{lem:Perron-Frobenius}. Now, since $h^l_I=S_{I}(\xi_{n}(I_0))$ by the same lemma,
\begin{align*}
\Phi_{n}\left[(M^h)^{-1}([M_i])\right](I_0)=\left\langle \mathbf{1},(M^h)^{-1}([M_i])\right\rangle_h=&\left\langle \mathbf{1},\frac{1}{h^r}\left(\sum_{I\in B_{k,n}}a^i_Ie_I\right)\right\rangle_h\\
=&\left\langle h^l,\left(\sum_{I\in B_{k,n}}a^i_Ie_I\right)\right\rangle\\
=&\sum_{I\in B_{k,n}}a^i_IS_{I}(\xi_{n}(I_0))=qDim(M_i),
\end{align*}
and we deduce that $\frac{1}{qDim(M_i)}(M^h)^{-1}(M_i)$ is a $h$-probability measure on $B_{k,n}$. Then, recalling that $\delta_I=\frac{1}{\mu^h(I)}e_I=\frac{1}{h_l(I)h_r(I)}e_I$, we have 
$$(M^h)^{-1}(M_0)=\sum_{I\in B_k,n, \langle I\rangle=d_0}a_IS_{I}(\xi_{n}(I_0))\delta_I,$$
so that finally 
$$\overline{c}_{\prod_{i=1}^{m} [M_i]}^{I'}=\frac{V(\xi_{n}(I_0))^2S_{I'}(\xi_{n}(I_0))\prod_{i=0}^m qDim(M_i)}{n^k}\sum_{I\in B_k,n, \langle I\rangle=d_0}a_{I}S_{I}(\xi_{n}(I_0))\left(\delta_I*(*\mathbf{m})\right)(I'),$$
with $\mathbf{m}=(\mu_i)_{1\leq i\leq m}$ where $\mu_i=\frac{1}{qDim(M_i)}(M^h)^{-1}(M_i)$. Hence, applying Theorem \ref{thm:local-limit} with the notation of Corollary \ref{cor:cohomology}, we deduce that 
\begin{align*}
\overline{c}_{\prod_{i=0}^{m} [M_i]}^{I'}=&\frac{V(\xi_{n}(I_0))^2Q_{I'}(\xi_{n}(I_0))\prod_{i=1}^m qDim(M_i)}{n^{k-1}}\\&\hspace{1cm}\sum_{I\in B_k,n, \langle I\rangle=d_0}a_{I}s_{I}(\xi_{n}(I_0))\left(K_{\frac{(2\pi)^2m}{n^2}\gamma}^{SU(k)}\left(R_{-\frac{m\langle \mathbf{m}\rangle}{kn}}\left(\xi_{n}(I)\right),\xi_{n}(I')\right)+O\left(\frac{1}{n}\right)\right),
\end{align*}
with $O(\cdot)$ only depending on $\frac{m}{n^2},k$ and $\Vert \tilde{m}\Vert_{i}$, $i=2,3$. Then, writing $p^i_I=\frac{a_IS_I(\xi_{n}(I_0))}{qDim(M_i)}$ for $0\leq i\leq m+1$, so that $\sum_{I\in B_{k,n}}p_I^i=1$, we get 
\begin{align*}
\#\mathcal{M}^{d,a_0,\ldots,a_{m+1}}_{M_0,\ldots,M_{m+1}}=\frac{V(\xi_{n}(I_0))^2\prod_{i=0}^{m+1}qDim(M_i)}{n^{k-1}}\left(\sum_{I,I'\in B_{k,n}}p_I^{M_0}p_{I'}^{M_{m+1}}K_{\gamma t_0}^{SU(k)}\left(x^n_I,x^n_{I'}\right)+O\left(\frac{1}{n}\right)\right),
\end{align*}
with $x_{I}^n=R_{-\frac{m\langle \mathbf{m}\rangle}{kn}}\left(\xi_{n}(I)\right)$ and $y_{I'}^n=\xi_{n}((I')^c)$.
\end{proof}

\appendix
\addcontentsline{toc}{section}{Appendices}

\section{Small quantum cohomology of the Grassmannian}\label{sec:quantum_cohomology}
We recall here some basic facts about the small quantum cohomology of the Grassmannian.
\subsection{Partitions and symmetric functions}
Recall that a partition $\lambda$ is a nonincreasing sequence of integers $(\lambda_1\geq \lambda_2\geq \ldots\geq \lambda_k\geq 0)$. We denote then by $l(\lambda)$ the length of the partition $\lambda$. The integer $\lambda_i$ is called the $i$-th part of $\lambda$, and the size $\langle \lambda\rangle$ of $\lambda$ is the sum of its parts. A partition is depicted by a Young diagram, consisting of the drawing of $\langle \lambda\rangle$ left-aligned boxes with $\lambda_i$ boxes in the i-th row.

We denote by $R$ the set of all partitions, $\Lambda_{k}$ the set of partitions of length $k\geq 1$ and $\Lambda_{k,n}$ the set of partitions of length $k$ such that the first part is smaller that $n-k$. Alternatively, $\Lambda_{k,n}$ can be seen as the set of partitions whose Young diagram can be included in a rectangle of $k$ rows and $n-k$ columns. In the sequel, we will regularly identify $\lambda$ with its Young diagram. If $\lambda\in \Lambda_{k,n}$, the  complementary of $\lambda$ in $\Lambda_{k,n}$ is the partition $\lambda^c$ such that $\lambda_i+\lambda^c_{k-i+1}=n-k$ for all $1\leq i\leq k$. This is the partition we get by rotating upside down the set of boxes of the rectangle $k\times (n-k)$ not included in $\lambda$.

We say that $\lambda\subset \mu$ if the Young diagram of $\lambda$ is included in the one of $\mu$, which means that $\lambda_i\leq \mu_i$ for all $1\leq i\leq k$. 

Let $Sym[x_1,\ldots,x_k]$ denote the ring of symmetric polynomials in $k$ variables. This ring is freely generated by the elementary symmetric functions $e_j(x_1,\ldots,x_k)=\sum_{1\leq i_1<\ldots i_k\leq k}x_{i_1}\dots x_{i_k}$ for $1\leq j\leq k$. It also has a distinguished basis $\{s_{\lambda}\}_{\lambda\in \Lambda_{k}}$ given by the Schur polynomials, whose value at $(x_1,\ldots,x_k)$ is 
$$s_{\lambda}(x_1,\ldots,x_k)=\frac{\det(x_i^{\lambda_j+k-j})_{1\leq i,j\leq k}}{\det(x_i^{k-j})_{1\leq i,j\leq k}}.$$
The functions $\{s_{\lambda}\}_{\lambda\in \Lambda_{k}}$ are related to numerous fields of algebraic combinatorics, see \cite{Fulton1997}. In particular, this basis is positively multiplicative, meaning that 
$$s_{\lambda}\cdot s_{\mu}=\sum_{\nu\in \Lambda_{k}}c_{\lambda \mu}^\nu s_{\nu}$$
for some nonnegative coefficient $c_{\lambda\mu}^\nu$. The latter are actually nonnegative integers called Littlewood-Richardson coefficients.

\subsection{the Grassmannian, Schubert variety and intersection}
Let us review some aspects of the cohomology of the Grassmannian without going into details of homology theory, see \cite{Fulton1997} for more details.  Denote by $G_{k,n}$ the variety of $k$-dimensional subspaces of $\mathbb{C}^n$ and  $\mathbb{W}=\{(0\subset W_1\subset \ldots\subset W_n=\mathbb{C}^n)\}$ the variety of flags on $\mathbb{C}^n$. By a {\it generic} element of $\mathbb{W}$, we mean any element of $\mathbb{W}$ excepting elements of a submanifold of strictly smaller dimension. For a particular flag $W\in \mathbb{W}$ and a partition $\lambda\subset R$, we define the Schubert variety $\Omega_\lambda$ as
$$\Omega_{\lambda}^W=\{V\in G_{k,n}, \dim(V\cap W_{n-k+i-\lambda_i})\geq i,\, 1\leq i\leq k\}.$$
Then, $\Omega^W_{\lambda}$ is a sub-variety of $G_{k,n}$ of dimension $k(n-k)-\langle \lambda\rangle$. Moreover, if $\lambda\subset \mu$ then $\Omega_{\mu}^W\subset \Omega_{\lambda}^W$. One can then prove that $\overset{\circ}{\Omega}_{\lambda} {}^W:=\Omega_{\lambda}^W\setminus\bigcup_{\substack{\mu\subset\lambda\\ \mu\not =\lambda}}\Omega_{\mu}^W\simeq \mathbb{C}^{k(n-k)-\langle \lambda\rangle}$. In particular, 
$$\overset{\circ}{G_{k,n}}=\overset{\circ}{\Omega_{\emptyset}^W}=\Omega_{\emptyset}^W\setminus \Omega_{(1)}^W\simeq \mathbb{C}^{k(n-k)}$$
and $G_{k,n}$ is a variety of dimension $k(n-k)$.

If $W'$ is another generic flag in $\mathbb{W}$, $\Omega_{\lambda}^W$ intersects transversely $\Omega_{\mu}^{W'}$, which means that $\Omega_{\lambda}^W\cap \Omega_{\mu}^{W'}$ is a well-defined variety of dimension $k(n-k)-(\langle \lambda\rangle+\langle \mu\rangle)$. In particular, $\Omega_{\lambda}^W\cap \Omega_{\lambda^c}^{W'}$ is a $0$-dimensional variety which is actually reduced to a singleton. One can then decompose the intersection $\Omega_{\lambda}^W\cap \Omega_{\mu}^{W'}$ as a union of sub-varieties which are deformations of some $\Omega_{\nu}^W$ for $\nu\in \Lambda_{k,n}$. A simple way to see this decomposition is to use Poincaré duality which expresses the aforementioned fact that for $\lambda\in \Lambda_{k,n}$, the intersection $\Omega_{\lambda}^W\cap \Omega_{\lambda^c}^{W'}$ is a singleton. Then,  one says that a variety on type $\nu$ appears $c_{\lambda\mu}^\nu$ times in the decomposition of $\Omega_{\lambda}^W\cap \Omega_{\mu}^{W'}$ if for any generic $W''\in \mathbb{W}$, $\Omega_{\lambda}^W\cap \Omega_{\mu}^{W'}\cap \Omega_{\nu^c}^{W"}$ is a discrete set of points and
$$\#\left(\Omega_{\lambda}^W\cap \Omega_{\mu}^{W'}\cap \Omega_{\nu^c}^{W"}\right)=c_{\lambda\mu}^\nu.$$
Formally, this amounts to say that the cup-product of the corresponding cohomology cycles $\sigma_{\lambda}$ and $\sigma_{\mu}$, respectively dual to $\Omega^W_{\lambda}$ and $\Omega_{\mu}^{W'}$, decomposes in the cohomology ring $H^*(G_{k,n})$ as 
$$\sigma_{\lambda}\cdot\sigma_{\mu}=\sum_{\nu\in \Lambda_{k,n}}c_{\lambda,\mu}^\nu\sigma_{\nu}.$$
The coefficient $c_{\lambda,\mu}^\nu$ are the same Littlewood-Richardson coefficients which encodes the multiplication of Schur functions in the ring of symmetric functions. This is no surprising for the following reason: algebraically, $H^*(G_{k,n})\simeq Sym[x_1,\ldots,x_k]/< h_i=0, i>n-k>$, the algebra-isomorphism being given by the map $\sigma_\lambda\mapsto s_{\lambda}$ (which is known as Giambelli's formula).

\begin{example}\label{Example:P1}
In the simplest case $k=1$, we have $G(1,n+1)=\mathbb{P}^n(\mathbb{C})$, the projective space of dimension $n$. Then, for $W\in\mathbb{W}$ and $\lambda\in R_{1,n+1}$, 
$\Omega_{\lambda}^W=\{x\in \mathbb{P}^n(\mathbb{C}),x\subset W_{n+1-\lambda_1}\}$. More concretely, for each $1\leq i\leq n+1$, let $e_i\in \mathbb{C}^{n+1}$ be a vector in $W_{i}\setminus W_{i-1}$. Then, $(e_1,\ldots,e_{n+1})$ is a basis of $\mathbb{C}^{n+1}$ and for $1\leq k\leq n$,
$$\Omega_{(k)}^W=\left\{[x_1,\ldots,x_{n+1-k},0,\ldots,0], (x_i)_{1\leq i\leq n+1-k}\in \mathbb{C}^{n+1-k}\setminus\{0\}\right\},$$
where $[x_1,\ldots,x_{n+1}]$ denotes the line generated by $(x_1,\ldots,x_{n+1})$. In particular, $\Omega_{(n)}^W$ is the singleton $\{[1,0,\ldots,0]\}$. If $W'$ is another flags, then $\Omega_{(k_1)}^W\cap\Omega_{(k_2)}^{W'}=\{x\in \mathbb{P}^n(\mathbb{C}),x\subset W_{n+1-k_1}\cap W_{n+1-k_2}'\}$. Remark that for $W'$ generic, $\dim W_{n+1-k_1}\cap W_{n+1-k_2}'=n+1-(k_1+k_2)$. If $W''$ is a third generic flag, $\dim W_{n+1-k_1}\cap W_{n+1-k_2}'\cap W_{n+1-k_3}''=\max(n+1-(k_1+k_2+k_3),0)$: then, when $k_1+k_2\leq n$, $k_3=n-k_1-k_2$ we have 
$$\dim W_{n+1-k_1}\cap W_{n+1-k_2}'\cap W_{n+1-k_3}''=1,$$
and for any other configuration of $(k_1,k_2,k_3)$ this dimension either vanishes or is larger than $1$. Hence, the intersection $\Omega_{(k_1)}^W\cap\Omega_{(k_2)}^{W'}\cap \Omega_{(k_3)}^{W''}$ is discrete and finite only when $k_1+k_2\leq n$ and $k_3=n-(k_1+k_2)$. We deduce that 
$$\sigma_{(k_1)}\cdot\sigma_{(k_2)}=\delta_{k_1+k_2\leq n}\,\sigma_{(k_1+k_2)}.$$
This is precisely the multiplication in $Sym[x_1]/<x_1^i=0,i\geq n+1>$ when sending $\sigma_{(k)}$ to $x_1^k$ for $1\leq k\leq n$.
\end{example}

\subsection{Quantum cohomology ring of the Grassmannian} The small quantum cohomology ring of $G_{k,n}$ is an enrichment of $H^*(G_{k,n})$ that counts more broadly maps from $\mathbb{P}^1(\mathbb{C})$ to $G_{k,n}$ with prescribed values at specific points rather than intersection of subvarieties of $G_{k,n}$. Refer to \cite{B,BCF} for a general introduction to the small quantum cohomology. Note that one can generally define two quantum cohomologies for a given projective variety: the {\it big} quantum cohomology and the {\it small} one. In this paper we will only be interested in the {\it small} quantum cohomology which is much easier to handle, and we will simply drop the term {\it small} in the sequel (see \cite{KV} for an introduction to the big quantum cohomology). Denote by $Hol(G_{k,n})$ the set of rational maps from $\mathbb{P}^1(\mathbb{C})$ to $G_{k,n}$. We define the degree of a map $f\in Hol(G_{k,n})$ as 
$$\deg f=\#f^{-1}\left(\Omega_{(1)}^W\right),$$
and we can show this is independent of a generic choice of $W$. We denote by $Hol_{d}(G_{k,n})$ the space of rational maps from $\mathbb{P}^1(\mathbb{C})$ to $G_{k,n}$ of degree $d$. The latter is a quasi-projective variety (i.e, an open subset of a projective variety) of dimension $k(n-k)+dn$ whose closure yields a projective variety $\overline{Hol_{d}}(G_{k,n})$. Note that other compactifications of $Hol_{d}(G_{k,n})$ can be achieved with better smoothness, see \cite{B}. The choice of compactification will not play any role here, since we will study intersection of subvarieties of $\overline{Hol_{d}}(G_{k,n})$ which are generically located in $Hol_{d}(G_{k,n})$. For $a\in\mathbb{P}^1(\mathbb{C}), \, \lambda\in \Lambda_{k,n}$ and $W\in \mathbb{W}$, denote by $V_{a,\lambda}^W$ the subvariety 
$$V_{a,\lambda}^W=\left\{f\in Hol_d(G_{k,n}),\, f(a)\in\Omega_{\lambda}^W\right\}.$$
Then, $V_{a,\lambda}^W$ is a subvariety of $Hol_{d}(G_{k,n})$ of dimension $nd+k(n-k)-\langle \lambda\rangle$.
\begin{definition}[\cite{B}]
The quantum cohomology ring $QH(G_{k,n})$ is the vector space $\mathbb{C}[q]\otimes H^*(G_{k,n})$ with the multiplication
\begin{equation}\label{eq:multiplication_quantum}(q^k\sigma_\lambda)\cdot (q^{k'}\sigma_\mu)=q^{k+k'}\sum_{d\geq 0,\nu\in \Lambda_{k,n}}\langle\lambda,\mu,\nu^c\rangle_dq^d\sigma_{\nu},
\end{equation}
where $\langle\lambda,\mu,\nu^c\rangle_d=0$ if $\langle \lambda\rangle+\langle\mu\rangle+\langle \nu^c\rangle\not=k(n-k)+nd$, and otherwise
$$\langle\lambda,\mu,\nu\rangle_d=\# \left(V_{0,\lambda}^W\cap V_{1,\mu}^{W'}\cap V_{\infty,\nu}^{W''}\right).$$
\end{definition}
Remark that $d=0$ corresponds to constant maps on $G_{k,n}$ and we recover with $\langle\lambda,\mu,\nu^c\rangle_0$ the intersection numbers given in the previous paragraph. The fact that the multiplication is indeed an associative product on $\mathbb{C}[q]\otimes H^*(G_{k,n})$ is a highly nontrivial property. This associativity is however crucial for enumerative properties of the quantum cohomology ring. Indeed, for any $p\geq 2, \lambda^{1},\ldots,\lambda^p\in \Lambda_{k,n}$, then we have in $QH(G_{k,n})$
\begin{equation}\label{eq:qring_combinatorics}\sigma_{\lambda^1}\dots \sigma_{\lambda^p}=\sum_{d\geq 0, \,\lambda_{p+1}^c\in \Lambda_{k,n}}\langle\lambda_1,\ldots,\lambda_{p+1}\rangle_dq^d\sigma_{\lambda_{p+1}^c},
\end{equation}
where $\langle\lambda_1,\ldots,\lambda_{p+1}\rangle_d=0$ if $\sum_{i=1}^{p+1}\langle \lambda_i\rangle\not=k(n-k)+dn$ and otherwise
$$\langle\lambda_1,\ldots,\lambda_{p+1}\rangle_d=\#\left(\bigcap_{i=1}^{p+1}V_{a_i,\lambda_i}^{W_i}\right),$$
where $a_1,\ldots,a_{p+1}$ are $p+1$ generic points of $\mathbb{P}^1(\mathbb{C})$ and $W_{1},\ldots,W_{p+1}$ are $(p+1)$ generic flags in $\mathbb{W}$. 

The algebraic structure of $QH(G_{k,n})$ has been investigated by several authors \cite{B,BCF,BKPT}. This algebraic study culminated with the combinatorial description of the coefficients $\langle\lambda,\mu,\nu^c\rangle_d$ in \cite{BKPT} in terms of puzzles of the $2$-steps flag manifold (see also \cite{Bu} for a refined version). For our probabilistic purpose, one of the most important descriptions of $QH(G_{k,n})$ is given by its algebraic presentation
$$QH(G_{k,n})\simeq Sym[x_1,\ldots,x_k]/\langle h_i=0,i=n-k+1,\ldots n-1\rangle$$
through the map $\sigma_\lambda\mapsto s_{\lambda}, q\mapsto (-1)^{k+1}h_{n}$. This algebraic presentation can be further simplified, see \cite{Ri}, as
\begin{align*}
Sym[x_1,\ldots,x_k]/&\langle h_i=0,i=n-k+1,\ldots n-1, h_{n}=(-1)^{k+1}q\rangle\\
=&Sym[x_1,\ldots,x_k]/\langle (x_1,\ldots,x_k)=q^{1/k}\left(e^{it_1},\ldots e^{it_k}\right), \vec{t}\in T_{k,n}\rangle,
\end{align*}
and then $q=(-1)^{k+1}x_i^n$ for all $1\leq i\leq n$. Note that the relation between both presentations of $QH(G_{k,n})$ is not trivial.
\begin{example}\label{Example:QuantumCohomP1}
Let us resume Example \ref{Example:P1} for $k=1$, where $G_{1,n+1}=\mathbb{P}^n(\mathbb{C})$. Then, a rational map $h$ from $\mathbb{P}^1(\mathbb{C})$ to $\mathbb{P}^n(\mathbb{C})$ is given by $n+1$ homogeneous polynomials $f_1(x,y),\ldots,f_{n+1}(x,y)$ of same degree which don't have any common zero. Denote by $d$ the degree of the polynomials $f_j,\, 1\leq j\leq n$. Up to a generic choice of basis $(e_1,\ldots,e_{n+1})$ of $\mathbb{P}^n(\mathbb{C})$, one can assume that all $f_j$ have nonzero coefficient for the monomial $y^d$ and that furthermore $f_{n+1}(0,1)\not=0$.  Let $W$ be the flag given by the basis $(e_1,\ldots,e_{n+1})$, with $W_{i}=Vect\langle e_1,\ldots,e_{i}\rangle$.  Recall then from Example \ref{Example:P1} that $\Omega_{(1)}^W=\left\lbrace [x_1,\ldots,x_{n},0], (x_i)_{1\leq i\leq n}\in \mathbb{C}^n\setminus\{0\}\right\}$. Hence, $h^{-1}(\Omega_{(1)}^W)=\{[x,y]\in\mathbb{P}^1(\mathbb{C}), \,f_{n+1}(x,y)=0\}$. By the choice of a generic basis, $f_{n+1}(0,1)\not=0$ and thus $(0,1)\not\in h^{-1}(\Omega_{(1)}^W)$. Then, on the open chart $\{x\not=0\}$, the  maps $h$ is given by 
$$h: (1,y)\mapsto [f_1(1,y),\ldots,f_{n+1}(1,y)],$$
where each $f_j(1,y)$ are univariate polynomials of degree $d$. Hence, 
$$h^{-1}(\Omega_{(1)}^W)\cap \{x\not= 0\}=\{[1,y], \, f_{n+1}(1,y)=0\}.$$ 
Since $f_{n+1}(1,y)$ is a degree $d$ polynomial, the latter equation has exactly $d$ solutions, and finally 
$$\#h^{-1}\left(\Omega_{(1)}^W\right)=d.$$
Hence, we see that in the simplest case $k=1$, the degree of a holomorphic map coincides with the algebraic degree of the map from $\mathbb{P}^1(\mathbb{C})$ to $\mathbb{P}^n(\mathbb{C})$. The ring $QH(G_{1,n+1})$ is then the ring of univariate polynomials $\mathbb{C}[x]$, with the map $\sigma_{k}\mapsto x^k$ for $1\leq k\leq n$ and $q\mapsto x^{n+1}$.

In the simplest case $n=1$, we can even recover a classical result. Let $x_1,\ldots,x_{2p+1}$ and $y_1, \ldots, y_{2p+1}$ be generic elements of $\mathbb{P}(\mathbb{C})$. Each $y_i\in\mathbb{P}^1(\mathbb{C})$ defines a flag in $\mathbb{C}^2$ given by $W_i=\{0\subset y_i\subset \mathbb{C}^2\}$. Moreover, saying that a map $h\in Hol(\mathbb{P}^1(\mathbb{C}),\mathbb{P}^1(\mathbb{C}))$ satisfies $f(x_i)=y_i$ exactly means that $f(x_i)\in \Omega_{(1)}^{W_i}$. By \eqref{eq:qring_combinatorics},
$$\sigma_{(1)}^{2p}=\sum_{\substack{d\geq 0, \,\lambda_{2p+1}^c\in R_{1,2}\\ 2p+\vert \lambda_{2p+1}\vert=2d+1}}\#\left(\Omega_{\lambda_{2p+1}}^{W_{2p+1}}\cap \bigcap_{1\leq i\leq 2p} V_{x_i,(1)}^{W_i}\right)q^d\sigma_{\lambda_{2p+1}^c},$$
and by the latter description of $QH(G_{1,2})$,
$$\sigma_{(1)}^{2p}=q^p.$$
Hence, the only non zero terms arise for $d=p$, and then $\lambda_{2p+1}=(1)$. We deduce that there is a unique rational map of degree $p$ such that $f(x_i)=y_i$ for generic $\{x_i,y_i\}_{1\leq i\leq 2p+1}\in\mathbb{P}^1(\mathbb{C})$.
\end{example}

\pagebreak

\printnomenclature


\begin{thebibliography}{99}                                                                                               %

\bibitem{AGZ} G. W. Anderson, A. Guionnet and O. Zeitouni. An introduction to random matrices,  118, Cambridge university press, 2018.

\bibitem {Beauv}A. Beauville, Conformal blocks, Fusion rings and the Verlinde
formula. Proc. of the Hirzebruch 65 Conf. on Algebraic Geometry, Israel Math.
Conf. Proc. 9, 75-96, 1996.

\bibitem{Bel1}P. Belkale. Quantum generalization of the Horn conjecture. Journal of the American Mathematical Society, 21(2), 365-408, 2008.

\bibitem{Bel2}P. Belkale. Rigid local systems and the multiplicative eigenvalue problem. Annals of Mathematics, 195(3), 911-995, 2022.

\bibitem{B} A. Bertram. Quantum Schubert calculus. arXiv preprint alg-geom/9410024, 1994.

\bibitem {BCF}A. Bertram, I. Ciocan-Fontamine and W. Fulton, Quantum
multiplication of Schur polynomials, 219(2), 728-746, 1999.

\bibitem{Bi2}P. Biane. Free brownian motion, free stochastic calculus and random matrice, in free probability theory. Fields Inst. Commun., 12, 1-19, 1997.
\bibitem{Bi}P. Biane. Quantum random walk on the dual of SU (n). Probability theory and related fields, 89(1), 117-129, 1991.

\bibitem{BobLed}S. G. Bobkov and M. Ledoux. A simple Fourier analytic proof of the AKT optimal matching theorem. The Annals of Applied Probability, 31(6), 2567-2584, 2021.

\bibitem{Bord}B. Borda. Empirical measures and random walks on compact spaces in the quadratic Wasserstein metric. arXiv preprint arXiv:2110.00295, 2021.

\bibitem{BJ}P. Bougerol and T. Jeulin. Paths in Weyl chambers and random matrices. Probability theory and related fields, 124(4), 517-543, 2002.

\bibitem{BdT} C. Boutillier and B. De Tiliere. Loop statistics in the toroidal honeycomb dimer model. The Annals of Probability, 37(5), 1747-1777, 2009.


\bibitem{Bro} K. S Brown, Cohomology of groups. Springer Science, 2012.

\bibitem {Bu}Buch, A. S., Mutations of puzzles and equivariant cohomology of
two-step flag varieties. Annals of Mathematics, 173-220, 2015.

\bibitem {BKPT}A.S Buch, A. Kresch, K. Purbhoo, H. Tamvakis, The puzzle
conjecture for the cohomology of two-step flag manifolds. Journal of Algebraic
Combinatorics 44.4, 973-1007, 2016.

\bibitem{BG}A. Bufetov and V. Gorin. Fourier transform on high-dimensional unitary groups with applications to random tilings. Duke Mathematical Journal, 168(13), 2559-2649, 2019.


\bibitem{CL}E. Cépa and D. Lépingle. Brownian particles with electrostatic repulsion on the circle: Dyson's model for unitary random matrices revisited. ESAIM: Probability and Statistics, 5, 203-224, 2001.

\bibitem{De}M. Defosseux. Fusion coefficients and random walks in alcoves. 
Annales de l'Institut Henri Poincaré, Probabilités et Statistiques, Vol. 52, No. 4, pp. 1515-1534, 2016.

\bibitem{DeWa15}D. Denisov and V. Wachtel. Random walks in cones. The Annals of Probability, 43(3), 992-1044, 2015.


\bibitem{Do-87} J.L. Doob. 
Classical potential theory and its probabilistic counterpart. 
Vol. 549. New York: Springer, 1984.

\bibitem{Far} J. Faraut. Analysis of Lie Groups. Cambridge: Cambridge University Press, 2008.

\bibitem{Fulm}M. Fulmek. Non-intersecting lattice paths on the cylinder. Séminaire Lotharingien de Combinatoire, 52, B52b-16, 2004.

\bibitem {Fulton1997}W. Fulton. \newblock {\em {Young Tableaux}}. \newblock
Cambridge University Press, 1997.

\bibitem{G} D. Gepner. Fusion rings and geometry. Communications in Mathematical Physics, 141(2), 381-411, 1991.

\bibitem{GK} V. Gorbounov and C. Korff. Quantum integrability and generalised quantum Schubert calculus. Advances in Mathematics, 313, 282-356, 2017.

\bibitem{Gra} D. J. Grabiner. Brownian motion in a Weyl chamber, noncolliding particles, and random matrices. In Annales de l'IHP Probabilités et statistiques, 35(2), pp. 177-204, 1999.

\bibitem{Gra2}D. J. Grabiner. Random walk in an alcove of an affine Weyl group, and non-colliding random walks on an interval. Journal of Combinatorial Theory, Series A, 97(2), 285-306, 2002.


\bibitem{GLT}J. Guilhot, C. Lecouvey and P. Tarrago. Basics on positively multiplicative graphs and algebras. arXiv preprint arXiv:2205.08889, 2022.

\bibitem{GLT2}J. Guilhot, C. Lecouvey and P. Tarrago. Homology rings of affine the Grassmannians and positively multiplicative graphs, 2023.

\bibitem{HW}D. G. Hobson and W. Werner. Non-colliding Brownian motions on the circle. Bulletin of the London Mathematical Society, 28(6), 643-650, 1996.

\bibitem {KacB}V. G. Kac, Infinite dimensional Lie algebras, Cambridge
University Press, third edition, 1989.

\bibitem{KW}R. W. Kenyon and D. B. Wilson. Critical resonance in the non-intersecting lattice path model. Probability theory and related fields, 130(3), 289-318, 2004.

\bibitem{KV}J. Kock and I. Vainsencher, An invitation to quantum cohomology: Kontsevich's formula for rational plane curves. Springer Science and Business Media, 2007.

\bibitem {Korff}C. Korff, The $su(n)$ WZNW fusion ring as integrable model: a
new algorithm to compute fusion coefficients, RIMS Kokyuroku Bessatsu Bx, 2011.

\bibitem {KS}C. Korff and C. Stroppel, The $\widehat{\mathfrak{sl}%
}(n)$-WZNW fusion ring: a combinatorial construction and a realisation as
quotient of quantum cohomology. Advances in Mathematics, 225(1), 200-268, 2010.

\bibitem{Kos}B. Kostant. The solution to a generalized Toda lattice and representation theory. Advances in Mathematics, 34(3), 195-338, 1979.

\bibitem{Kra} C. Krattenthaler. Asymptotics for random walks in alcoves of affine Weyl groups. Séminaire Lotharingien de Combinatoire, 52, B52i-72, 2004.


\bibitem {LM}L. Lapointe and J. Morse, Quantum cohomology and the $k$-Schur basis, Trans. AMS, 360(4), 2021–2040, 2008.


\bibitem{L} T. Lévy. The master field on the plane, Astérisque, 388, 2008.

\bibitem{LiYa}P. Li, P. and S. T. Yau. On the parabolic kernel of the Schrödinger operator. Acta Mathematica, 156, 153-201, 1986.

\bibitem {LW}K. Liechty and D. Wang. Nonintersecting Brownian motions on the unit circle. The Annals of Probability, 44(2), 1134-1211, 2016.


\bibitem{Meck} E. S. Meckes. The random matrix theory of the classical compact groups, 218,  Cambridge University Press, 2019.

\bibitem{MOW} A. P. Metcalfe, N. O’Connell and J. Warren. Interlaced processes on the circle. In Annales de l'IHP Probabilités et statistiques, 45(4), 1165-1184, 2009.

\bibitem{Mu}A. B. Mukhin. Local limit theorems for lattice random variables. SIAM, Theory of Probability And Its Applications, 36(4), 698-713, 1992.

\bibitem{OkOl} A. Okounkov and G. Olshanski,  Shifted schur functions. Algebra i Analiz, 9(2), 73-146, 1997.

\bibitem{RRW} M. Ravi, J. Rosenthal and X. Wang. Degree of the generalized Plücker embedding of a Quot scheme and quantum cohomology. Mathematische Annalen, 311.1, 11-26, 1998.

\bibitem {Ri}K. Rietsch, Quantum Cohomology of Grassmannians and Total Positivity, Duke Math., 110(3), 523-553, 2001.

\bibitem{Sch} I. Scherbak. Rational functions with prescribed critical points. Geometric and Functional Analysis GAFA, 12(6), 1365-1380, 2002.

\bibitem{Sot}F. Sottile. Rational curves on Grassmannians: systems theory, reality, and transversality.Previato, Advances in Algebraic Geometry Motivated by Physics: AMS Special Session on Enumerative Geometry in Physics, 2000.

\bibitem{We} D. Wehn, Limit distributions on Lie groups. Ph.D. thesis, Yale Univ., Ann Arbor, MI, 1960.


\end{thebibliography}
\end{document}